\documentclass{amsart}
\usepackage[utf8]{inputenc}
\usepackage[maxnames=4,maxalphanames=4,style=alphabetic-verb]{biblatex}
\bibliography{main.bib}
\usepackage[margin=1in]{geometry}
\usepackage{graphicx}
\usepackage{amsmath,amsfonts,amssymb,amsthm,amsaddr,etoolbox}
\usepackage{latexsym,hyperref}
\usepackage{xcolor}
\usepackage{scrextend}
\usepackage{tikz-cd}
\usepackage{quiver}
\usepackage{enumitem}
\usepackage[normalem]{ulem}
\usepackage{mathtools}
\usepackage[linewidth=1pt]{mdframed}

\author[Srivastava, Sinha and Kuber]{Suyash Srivastava, Vinit Sinha and Amit Kuber}
\address{Department of Mathematics and Statistics\\Indian Institute of Technology, Kanpur\\ Uttar Pradesh, India}
\email{suyash20@iitk.ac.in, vinitsinha20@iitk.ac.in, askuber@iitk.ac.in}
\title[Stable radical for special biserial algebras]{{On the stable radical of the module category for special biserial algebras}}
\keywords{string algebra, special biserial algebra, hammock, stable rank, radical of the module category}
\subjclass[2020]{16G20,16S90,06A06,16G60}

\begin{document}

\def\rat{\mathbb{Q}}
\def\tblue#1{\color{blue}#1\color{black}}
\def\abs#1{\lvert#1\rvert}
\def\MSB{MSB }
\def\pivot{\mathsf{pivot}}
\def\dim{\mathsf{dim}}
\def\MS{\mathsf{MS}}
\def\res{\upharpoonright}
\def\tensor{\bigotimes}
\def\defeq{\vcentcolon=}
\def\meet{\wedge}
\def\join{\vee}
\def\dLOfpb#1#2{\mathrm{dLO}_\mathrm{fp}^{{#1}{#2}}}
\def\corner{F}
\def\less{\prec}
\def\frame#1{\begin{mdframed}#1\end{mdframed}}
\def\hdammock{$\Tilde{h}$ammock }
\renewcommand{\labelitemii}{$ \circ $}
\def\sd{sd}
\def\eqdef{=\vcentcolon}
\def\upset{\uparrow}
\def\downset{\downarrow}
\def\length{\mathsf{length}}
\def\gray{\textcolor{gray}}
\def\teal{\textcolor{teal}}
\def\lime{\textcolor{lime}}
\def\magenta{\textcolor{magenta}}
\def\orange{\textcolor{orange}}
\newcommand{\Il}{\pi_l(\ii)}
\newcommand{\Ir}{\pi_r(\ii)}
\newcommand{\htIl}{\htpi_l(\ii)}
\newcommand{\htIr}{\htpi_r(\ii)}
\newcommand{\STT}{\mathsf{long}}
\newcommand{\smol}{\mathsf{short}}
\newcommand{\beeg}{\mathsf{long}}
\newcommand{\eqvl}{\equiv_H^l}
\newcommand{\eqvr}{\equiv_H^r}
\newcommand{\eqvlr}{\equiv_H^{lr}}
\newcommand{\eqv}{\equiv_{H}}

\newcommand{\Lo}{\mathsf{L^0}}
\newcommand{\Lp}{\mathsf{L^+}}
\newcommand{\Lm}{\mathsf{L^-}}
\newcommand{\sfL}{\mathsf{L}}
\newcommand{\Ro}{\mathsf{R^0}}
\newcommand{\Rp}{\mathsf{R^+}}
\newcommand{\Rm}{\mathsf{R^-}}
\newcommand{\sfR}{\mathsf{R}}

\newcommand{\DL}{\Delta_\sfL}
\newcommand{\DR}{\Delta_\sfR}

\newcommand{\hb}{\widebar{H}}
\newcommand{\hht}{\widehat{H}}
\newcommand{\hd}{\widetilde H}
\newcommand{\Hb}{\overline{\mathcal H}}
\newcommand{\Hht}{\widehat{\mathcal H}}
\newcommand{\Hd}{\widetilde{\mathcal H}}
\newcommand{\infb}{\,^\infty\bb}
\newcommand{\binf}{\,\bb^\infty}
\newcommand{\qb}{\mathsf Q^\mathsf {Ba}}
\newcommand{\suc}{\mathfrak{succ}}
\newcommand{\pred}{\mathfrak{pred}}
\newcommand\A{\mathcal{A}}
\newcommand\B{\mathsf{B}}
\newcommand\BB{\mathcal{B}}
\newcommand\C{\mathcal{C}}
\newcommand\Pp{\mathcal{P}}
\newcommand\D{\mathcal{D}}
\newcommand\Hamm{\hat{H}}
\newcommand\hh{\mathfrak{h}}
\newcommand\HH{\mathcal{H}}
\newcommand\RR{\mathcal{R}}
\newcommand\Red[1]{\mathrm{R}_{#1}}
\newcommand\HRed[1]{\mathrm{HR}_{#1}}
\newcommand\K{\mathcal{K}}
\newcommand\LL{\mathcal{L}}
\newcommand\Lim{\text{\normalfont Lim}}
\newcommand\M{\mathcal{M}}
\newcommand\Q{\mathcal{Q}}
\newcommand\SD{\mathcal{SD}}
\newcommand\MD{\mathcal{MD}}
\newcommand\SMD{\mathcal{SMD}}
\newcommand\T{\mathcal{T}}
\newcommand\TT{\mathfrak T}
\newcommand\ii{\mathcal I}
\newcommand\UU{\mathcal{U}}
\newcommand\VV{\mathcal{V}}
\newcommand\ZZ{\mathcal{Z}}
\newcommand{\N}{\mathbb{N}} 
\newcommand{\R}{\mathbb{R}}
\newcommand{\Z}{\mathbb{Z}}
\newcommand{\bb}{\mathfrak b}
\newcommand{\qq}{\mathfrak q}
\newcommand{\ch}{\circ_H}
\newcommand{\cg}{\circ_G}
\newcommand{\bua}[1]{\mathfrak b^{\alpha}(#1)}
\newcommand{\falpha}{{\mathfrak{f}\alpha}}
\newcommand{\fgamma}{\gamma^{\mathfrak f}}
\newcommand{\fbeta}{{\mathfrak{f}\beta}}
\newcommand{\bub}[1]{\mathfrak b^{\beta}(#1)}
\newcommand{\bla}[1]{\mathfrak b_{\alpha}(#1)}
\newcommand{\blb}[1]{\mathfrak b_{\beta}(#1)}
\newcommand{\lmin}{\lambda^{\mathrm{min}}}
\newcommand{\lmax}{\lambda^{\mathrm{max}}}
\newcommand{\xmin}{\xi^{\mathrm{min}}}
\newcommand{\xmax}{\xi^{\mathrm{max}}}
\newcommand{\lbmin}{\bar\lambda^{\mathrm{min}}}
\newcommand{\lbmax}{\bar\lambda^{\mathrm{max}}}
\newcommand{\ff}{\mathfrak f}
\newcommand{\cc}{\mathfrak c}
\newcommand{\dd}{\mathfrak d}
\newcommand{\sqsf}{\sqsubset^\ff}
\newcommand{\rr}{\mathfrak r}
\newcommand{\pp}{\mathfrak p}
\newcommand{\uu}{\mathfrak u}
\newcommand{\vv}{\mathfrak v}
\newcommand{\ww}{\mathfrak w}
\newcommand{\xx}{\mathfrak x}
\newcommand{\yy}{\mathfrak y}
\newcommand{\zz}{\mathfrak z}
\newcommand{\MM}{\mathfrak M}
\newcommand{\mm}{\mathfrak m}
\newcommand{\sbq}{\mathfrak s}
\newcommand{\tbq}{\mathfrak t}
\newcommand{\Spec}{\mathbf{Spec}}
\newcommand{\Br}{\mathbf{Br}}
\newcommand{\sk}[1]{\{#1\}}
\newcommand{\Prime}{\mathbf{Pr}}
\newcommand{\Parent}{\mathbf{Parent}}
\newcommand{\Uncle}{\mathbf{Uncle}}
\newcommand{\Cousin}{\mathbf{Cousin}}
\newcommand{\Nephew}{\mathbf{Nephew}}
\newcommand{\Sibling}{\mathbf{Sibling}}
\newcommand{\uc}{\mathrm{uc}}
\newcommand{\MCP}{\mathrm{MCP}}
\newcommand{\MSCP}{\mathrm{MSCP}}
\newcommand{\TTT}{\widetilde{\T}}
\newcommand{\la}{l}
\newcommand{\ra}{r}
\newcommand{\lb}{\bar{l}}
\newcommand{\rb}{\bar{r}}
\newcommand{\tBa}{\varepsilon^{\mathrm{Ba}}}
\def\ii{\mathcal{I}}
\newcommand{\brac}[2]{\langle #1,#2\rangle}
\newcommand{\braket}[3]{\langle #1\mid #2:#3\rangle}
\newcommand{\fin}{fin}
\newcommand{\inff}{inf}
\newcommand{\Zg}{\mathrm{Zg}(\Lambda)}
\newcommand{\Zgs}{\mathrm{Zg_{str}}(\Lambda)}
\newcommand{\STR}[1]{\mathrm{Str}(#1)}
\newcommand{\dmod}{\mbox{-}\mathrm{mod}}
\newcommand{\HQ}{\mathcal{HQ}^\mathrm{Ba}}
\newcommand{\bHQ}{\overline{\mathcal{HQ}}^\mathrm{Ba}}
\newcommand\Af{\mathcal{A}^{\ff}}
\newcommand\AAf{\bar{\mathcal{A}}^{\ff}}
\newcommand\Hf{\mathcal{H}^{\ff}}
\newcommand\Rf{\mathcal{R}^{\ff}}
\newcommand\Tf{\T^{\ff}}
\newcommand\Uf{\mathcal{U}^{\ff}}
\newcommand\Sf{\mathcal{S}^{\ff}}
\newcommand\Vf{\mathcal{V}^{\ff}}
\newcommand\Zf{\mathcal{Z}^{\ff}}
\newcommand\bVf{\overline{\mathcal{V}}^{\ff}}
\newcommand\bTf{\overline{\mathcal{T}}^{\ff}}
\newcommand{\fmin}{\xi^{\mathrm{fmin}}}
\newcommand{\fmax}{\xi^{\mathrm{fmax}}}
\newcommand{\xif}{\xi^\ff}
\newcommand{\mycomment}[1]{}
\newcommand{\LOfp}{\mathrm{LO}_{\mathrm{fp}}}

\newcommand\Tl{\mathbf{T}}
\newcommand\Tla{\mathbf{T}_{\la}}
\newcommand\Tlb{\mathbf{T}_{\lb}}
\newcommand\Ml{\mathbf{M}}
\newcommand\Mla{\mathbf{M}_{\la}}
\newcommand\Mlb{\mathbf{M}_{\lb}}
\newcommand\OT{\mathcal{O}}
\newcommand\LO{\mathbf{LO}}

\def\tblue{\textcolor{blue}}

\newcommand\rad{\mathrm{rad}_\Lambda} 

\newcommand{\fork}[1]{\mathrm{Str}_{\text{Fork}}^{\la}(#1)}
\newcommand{\BaB}{\mathsf{Ba(B)}}
\newcommand{\CycB}{\mathsf{Cyc(B)}}
\newcommand{\ExtB}{\mathsf{Ext(B)}}
\newcommand{\St}{\mathsf{St}(\Lambda)}
\newcommand{\STB}[2]{\mathsf{St}_{#1}(#2)}
\newcommand{\BST}[1]{\mathsf{BSt}_{#1}(\xx_0,i;\sB)}
\newcommand{\CSt}[1]{\mathsf{CSt}_{#1}(\sB)}
\newcommand{\CST}[1]{\mathsf{CSt}_{#1}(\xx_0,i;\sB)}
\newcommand{\ASt}[1]{\mathsf{ASt}_{#1}(\sB)}

\newcommand{\OSt}[2]{{\overline{\mathsf{St}}_{#1}(#2)}}
\newcommand{\lB}{\ell_{\sB}}
\newcommand{\lbB}{\overline{\ell}_{\sB}}
\newcommand{\LB}{l_{\sB}}
\newcommand{\LbB}{\overline{l}_{\sB}}

\newcommand{\BalB}{\mathsf{Ba}_l(\sB)}
\newcommand{\BalbB}{\mathsf{Ba}_{\lb}(\sB)}
\newcommand{\BarB}{\mathsf{Ba}_r(\sB)}
\newcommand{\BarbB}{\mathsf{Ba}_{\bar{r}}(\sB)}

\newcommand{\sB}{\mathsf{B}}
\newcommand{\Str}[1]{\mathrm{Str}_{#1}(\xx_0,i;\sB)}
\newcommand{\Strd}[1]{\mathrm{Str}'_{#1}(\xx_0,i;\sB)}
\newcommand{\Strdd}[1]{\mathrm{Str}''_{#1}(\xx_0,i;\sB)}
\newcommand{\Cent}{\mathrm{Cent}(\xx_0,i;\sB)}
\newcommand{\Start}{\mathrm{Start}(\xx_0,i;\sB)}
\newcommand{\End}{\mathrm{End}(\xx_0,i;\sB)}
\newcommand{\braclB}{\brac{1}{\lB}}
\newcommand{\braclbB}{\brac{1}{\lbB}}
\newcommand{\QBa}{\Q^{\mathrm{Ba}}}
\newcommand{\Ba}{\mathcal Q_0^\mathrm{Ba}}
\newcommand{\Cyc}{\mathsf{Cyc}(\Lambda)}
\newcommand{\lazy}{1_{(v, i)}}
\newcommand{\dLOfd}{\mathrm{dLO_{fd}}}
\newcommand{\dLOfdb}[2]{\mathrm{dLO}_{\mathrm{fd}}^{{#1}{#2}}}

\newcommand{\LOfd}{\mathrm{LO}_{\mathrm{fd}}}
\newcommand{\Balr}{((\BalB\cap \BarB) \cup (\BalbB \cap \BarbB))}
\newcommand{\Hbar}{\overline{H}_l^i(\xx_0)}
\newcommand{\Hhat}{\widehat{H}_l^i(\xx_0)}
\newcommand{\HOST}[1]{\widehat{\overline{\mathsf{St}}}_{#1}(\xx_0,i;\sB)}
\newcommand{\HHlix}{\widehat{H}_l^i(\xx_0)}

\def\lan{\langle}
\def\ran{\rangle}
\newcommand{\s}[1]{{\mathsf{#1}}}
\newcommand{\f}[1]{\mathfrak {#1}}
\newcommand{\ol}[1]{\overline{#1}}

\newcommand{\gst}{\s{St}}
\newcommand{\stbar}{\ol\gst}
\mathchardef\mh="2D 
\newcommand{\allst}{(\le\omega)\mh\s{St}(\s B)}
\def\da{\Downarrow}
\def\ua{\Uparrow}
\def\comp{\circ}
\def\aa{\mathfrak a}
\def\to{\rightarrow}
\def\HR{\textsf{HR}}
\def\parens#1{{(#1)}}
\def\mbf{\mathbf}
\newcommand{\rank}{\mathsf{rank}}
\newcommand{\mfa}{\mathfrak{a}}
\newtheorem{defn}{Definition}[subsection]
\newtheorem{definitions}[defn]{Definitions}
\newtheorem{lem}[defn]{Lemma}
\newtheorem{conj}{Conjecture}
\newtheorem{manualthm}{Theorem}
\newtheorem{construction}[defn]{Construction}
\newtheorem{prop}[defn]{Proposition}
\newtheorem*{prop*}{Proposition}
\newtheorem*{thm*}{Theorem}
\newtheorem{thm}[defn]{Theorem}
\newtheorem{cor}[defn]{Corollary}
\newtheorem{claim}[defn]{Claim}
\newtheorem*{claim*}{Claim}
\newtheorem{algo}[defn]{Algorithm}
\theoremstyle{remark}
\newtheorem{rem}[defn]{Remark}
\theoremstyle{remark}
\newtheorem{remarks}[defn]{Remarks}
\theoremstyle{remark}
\newtheorem{notation}[defn]{Notation}
\theoremstyle{remark}
\newtheorem{exmp}[defn]{Example}
\theoremstyle{remark}
\newtheorem{examples}[defn]{Examples}
\theoremstyle{remark}
\newtheorem{dgram}[defn]{Diagram}
\theoremstyle{remark}
\newtheorem{fact}[defn]{Fact}
\theoremstyle{remark}
\newtheorem{illust}[defn]{Illustration}
\theoremstyle{remark}
\newtheorem{que}[defn]{Question}
\numberwithin{equation}{section}
\newtheorem{example}[defn]{Example}
\newtheorem{exercise}[defn]{Exercise}
\def\st{\mathrm{st}}
\def\rk{\mathrm{rk}}
\makeatletter
\let\save@mathaccent\mathaccent
\newcommand*\if@single[3]{%
  \setbox0\hbox{${\mathaccent"0362{#1}}^H$}%
  \setbox2\hbox{${\mathaccent"0362{\kern0pt#1}}^H$}%
  \ifdim\ht0=\ht2 #3\else #2\fi
  }
\newcommand*\rel@kern[1]{\kern#1\dimexpr\macc@kerna}
\newcommand*\widebar[1]{\@ifnextchar^{{\wide@bar{#1}{0}}}{\wide@bar{#1}{1}}}
\newcommand*\wide@bar[2]{\if@single{#1}{\wide@bar@{#1}{#2}{1}}{\wide@bar@{#1}{#2}{2}}}
\newcommand*\wide@bar@[3]{%
  \begingroup
  \def\mathaccent##1##2{%
    \let\mathaccent\save@mathaccent
    \if#32 \let\macc@nucleus\first@char \fi
    \setbox\z@\hbox{$\macc@style{\macc@nucleus}_{}$}%
    \setbox\tw@\hbox{$\macc@style{\macc@nucleus}{}_{}$}%
    \dimen@\wd\tw@
    \advance\dimen@-\wd\z@
    \divide\dimen@ 3
    \@tempdima\wd\tw@
    \advance\@tempdima-\scriptspace
    \divide\@tempdima 10
    \advance\dimen@-\@tempdima
    \ifdim\dimen@>\z@ \dimen@0pt\fi
    \rel@kern{0.6}\kern-\dimen@
    \if#31
      \overline{\rel@kern{-0.6}\kern\dimen@\macc@nucleus\rel@kern{0.4}\kern\dimen@}%
      \advance\dimen@0.4\dimexpr\macc@kerna
      \let\final@kern#2%
      \ifdim\dimen@<\z@ \let\final@kern1\fi
      \if\final@kern1 \kern-\dimen@\fi
    \else
      \overline{\rel@kern{-0.6}\kern\dimen@#1}%
    \fi
  }%
  \macc@depth\@ne
  \let\math@bgroup\@empty \let\math@egroup\macc@set@skewchar
  \mathsurround\z@ \frozen@everymath{\mathgroup\macc@group\relax}%
  \macc@set@skewchar\relax
  \let\mathaccentV\macc@nested@a
  \if#31
    \macc@nested@a\relax111{#1}%
  \else
    \def\gobble@till@marker##1\endmarker{}%
    \futurelet\first@char\gobble@till@marker#1\endmarker
    \ifcat\noexpand\first@char A\else
      \def\first@char{}%
    \fi
    \macc@nested@a\relax111{\first@char}%
  \fi
  \endgroup
}
\makeatother
\newcommand\test[1]{%
$#1{M}$ $#1{A}$ $#1{g}$ $#1{\beta}$ $#1{\mathcal A}^q$
$#1{AB}^\sigma$ $#1{H}^C$ $#1{\sin z}$ $#1{W}_n$}

\maketitle

\begin{abstract}
Suppose $\Lambda$ is a special biserial algebra over an algebraically closed field. Schr\"oer showed that if $\Lambda$ is domestic then the radical of the category of finitely generated (left) $\Lambda$-modules is nilpotent, and the least ordinal, denoted $\st(\Lambda)$, where the decreasing sequence of powers of the radical stabilizes satisfies $\st(\Lambda)<\omega^2$.  With Gupta and Sardar, the third author conjectured that if $\Lambda$ has at least one band then $\omega\le\st(\Lambda)<\omega^2$ even when $\Lambda$ is non-domestic. In this paper we settle this conjecture in the affirmative. We also describe an algorithm to compute $\st(\Lambda)$ up to a finite error. We also show that for each $\omega\leq\alpha<\omega^2$ there is a finite-dimensional tame representation type algebra $\Gamma$ with $\st(\Gamma)=\alpha$.
\end{abstract}

\section{Introduction}
Fix an algebraically closed field $\K$. Let $\Lambda$ be a finite-dimensional, and hence an Artin algebra over $\K$. We denote the category of all finitely generated(=finite-dimensional) left $\Lambda$-modules by $\Lambda\dmod$. A classical problem in representation theory is to classify all indecomposable objects in $\Lambda\dmod$ and morphisms between them. Auslander-Reiten theory, introduced in a series of articles in the 1970s, has become a fundamental tool for studying these types of problems for Artin algebras. The Auslander-Reiten (A-R) quiver of $\Lambda$ has as its vertices the isomorphism classes of indecomposables in $\Lambda\dmod$ while the set of arrows from a vertex to another is in bijective correspondence with the basis of the vector space of irreducible morphisms between the corresponding modules.

Denote by $\rad$ the radical of $\Lambda\dmod$, which is the two-sided (additive) ideal of this category generated by all non-invertible(=non-split) morphisms between indecomposable modules. Recall that for an ideal $J$ of $\Lambda\dmod$ and $n >0$, $J^n$ denotes the ideal generated by all the $n$-fold compositions of morphisms in $ J $. The A-R quiver yields a description of $\rad/\rad^\omega$, where $ \rad^\omega := \bigcap_{n \in \N} \rad^n$. Schr\"{o}er \cite{schroer2000infinite} built upon the idea of Simmons (see \cite[\S~10.1]{prest2003model}) to define $\rad^\nu$ for each ordinal $\nu$ using transfinite recursion as follows:
\begin{itemize}
    \item $\rad^0 := \Lambda\dmod$;
    \item {if $1\leq\nu<\omega$ then $\rad^\nu :=\rad\rad^{\nu-1}$;}
    \item if $\nu>0$ is a limit ordinal then $\rad^\nu := \bigcap_{\lambda <\nu}\rad^\lambda$;
    \item if $\nu = \lambda + n$ for some limit ordinal $\lambda$ and $0<n<\omega$ then $\rad^\nu:= \left(\rad^{\lambda}\right)^{n+1}. $
\end{itemize}
Set $\rad^\infty := \bigcap_{\nu\in\mathbb{ON}}\rad^\nu$. Clearly $\lan\rad^\nu\mid \nu \in \mathbb{ON}\ran$ is a decreasing sequence of ideals. Since $\Lambda\dmod$ is a skeletally small category, this sequence must stabilise--the smallest ordinal where its stabilises, denoted $\st(\Lambda)$, is called the \emph{stable rank of the radical} (\emph{stable rank}, for short) so that $\rad^\infty=\rad^{\st(\Lambda)}$. For a morphism $f \in \rad$, we further define the \emph{rank} of $f$ as
$$ \rk(f) := \begin{cases}
    \nu \quad&\text{if }f \in \rad^\nu\setminus\rad^{\nu+1}\text{ for some }\nu\in \mathbb{ON};\\
    \infty \quad&\text{if }f \in \rad^\infty.
\end{cases} $$

Drozd \cite{drozd1977matrix} and Crawley-Boevey \cite{bocses} classified each finite-dimensional algebra $\Lambda$ into one of three representation types, namely finite, tame, and wild, depending on the complexity of the classification problem of the indecomposable objects in $\Lambda\dmod$. Recall that a $\K$-algebra $\Lambda$ is said to be of the \emph{tame representation type} if for any $d \in\N^+$ there exists finite many $\Lambda$-$\K[T]$-bimodules $M_1 , M_2, \cdots , M_{\mu (d)}$, which are free and {of} finite rank as $\K[T]$-modules, such that all but finitely many $d$-dimensional indecomposables in $\Lambda\dmod$ are isomorphic to some $M_i\otimes_{\K[T]}S$ for some $i \in[\mu(d)]:=\{1,2\cdots, \mu (d)\}$ and a simple $\K[T]$-module $S$. We assume $\mu (d)$ to be the minimal natural number with this property. A tame representation type algebra is said to be \emph{domestic} if the function $\mu$ is bounded, otherwise \emph{non-domestic}.

The main objects of study in this paper are two important subclasses of the class of tame representation type algebras, consisting of string algebras and special biserial algebras, where various conjectures regarding the tame representation type can be tested due to the explicit combinatorial description of their A-R quivers. The class of string algebras is a subclass of the class of special biserial algebras. Schr\"oer showed \cite[Theorem~2]{schroer2000infinite} that a special biserial algebra $\Lambda$ is domestic if and only if $\rad^\infty=0$ if and only if $\rad^{\omega^2}=0$. Furthermore, if $\rad^\infty=0$ . For a domestic special biserial algebra $\Lambda$ he showed that $\st(\Lambda)$ is a successor ordinal \cite[Lemma~6.1(iii)]{schroer2000infinite} and that $\omega\cdot n\leq\st(\Lambda)<\omega\cdot(n+1)$ for some $n\in\N$ \cite[Theorem~3]{schroer2000infinite}, and also provided an algorithm to compute the value of $n$ using the maximal length of a path in a finite combinatorial gadget, called the \emph{bridge quiver}, associated to $ \Lambda $.

Assume that a string algebra is presented as $\Lambda=\K\Q/\lan\rho\ran$, where $\Q$ is a finite connected quiver and $\rho$ is a finite set of relations. Combinatorial entities called \emph{strings} and \emph{bands}, which are certain walks on the underlying quiver, provide a complete classification of the vertices of the A-R quiver for a string algebra--this result is essentially due to Gel'fand and Ponomarev \cite{G&P} whereas the classification of the arrows of the A-R quiver is due to Butler and Ringel \cite{butler1987auslander}. Bridge quivers for domestic string algebras have (certain equivalence classes of) bands as their vertices, and some special strings, called \emph{bridges}, as their arrows. Following Brenner \cite{brenner1986combinatorial}, Schr\"{o}er \cite{Schroer98hammocksfor} associated a poset, whose elements are pairs of strings, to each vertex of $\Q$; such posets are called \emph{hammock posets}. Since there are only finitely many vertices of $\Q$, he estimated the rank of a domestic string algebra as the maximum of the \emph{densities} of the hammock posets associated to these simple modules, up to a finite error term; here the density of a poset is a measure of its complexity taking value $\infty$ or an ordinal.

With Gupta and Sardar \cite{GKS20} the third author partially extended Schr\"{o}er's result \cite[Theorem~2]{schroer2000infinite} by computing the stable ranks for some non-domestic string algebras. They introduced a bridge quiver \cite[\S~3.2]{GKS20} for all string algebras using a finite set of bands called \emph{prime bands} that generate all bands. Under certain conditions on directed cycles in the bridge quiver they showed that $\omega\leq\st(\Lambda)\leq\omega +2$ \cite[Corollary~4.3.4]{GKS20}. They conjectured \cite[Conjecture~4.4.1]{GKS20} that the conclusion of Schr\"{o}er's result would be true for even non-domestic special biserial algebras, which is the content of the following theorem.
\begin{manualthm}\label{oldconj} For any special biserial algebra $ \Lambda $ with at least one band, we have $ \omega \leq \st(\Lambda) < \omega^2 $.
\end{manualthm}
The main goal of this paper is to settle this conjecture in the affirmative. The majority of the effort goes in the proof of this theorem for the case of string algebras. The extension to special biserial algebras is straightforward using the results from \cite{skowronski1983representation} and \cite{wald1985tame}. Additionally, we also provide an algorithm to estimate the stable rank up to a finite error term. In this manner this paper is a sequel to both \cite{schroer2000infinite} and \cite{GKS20}. Yet another important contribution of this paper is the following extension of \cite[Theorem~1]{schroer2000infinite}.
\begin{manualthm}\label{reversest}
For any $\omega\leq\alpha <\omega^2$, there exists a finite-dimensional tame representation type algebra $\Lambda$ such that $\st(\Lambda) = \alpha$. The algebra $\Lambda$ can be chosen to be a special biserial algebra.
\end{manualthm}

The ideas used in the proof of Theorem \ref{oldconj} could be traced back to \cite{sardar2021variations} and \cite{SardarKuberHamforDom} where Sardar and the third author meticulously computed the order types(=order-isomorphism classes) of left and right hammock linear orders for a fixed string $\xx_0$, i.e., sets of strings with $\xx_0$ as a starting or ending with $\xx_0$ respectively, for domestic string algebras. An equivalence relation on the set of all strings, called the \emph{$H$-equivalence} \cite[\S~5]{sardar2021variations}, captures when two strings have naturally isomorphic (left/right) hammocks--this equivalence is significant for there are only finitely many $H$-equivalence classes. The main result of these two papers states that the order types of left/right hammock linear orders lie in the class $\dLOfpb{1}{1}$ of bounded discrete finitely presented linear orders. The class $\LOfp$ is defined to be the smallest class of linear orders containing finite linear orders, and that is closed under isomorphisms, finite order sums, and anti-lexicographic products with $\omega$ and $\omega^*$ \cite[\S~5]{agrawal2022euclidean}.

This line of inquiry was continued in \cite{SKSK}, where the second and the third author together with Sengupta and Kale showed \cite[Theorem~11.9]{SKSK} that for any non-domestic string algebra the order types of left and right hammocks lie in the class $\dLOfdb{1}{1}$ of bounded discrete finite description linear orders. The class $\LOfd$ is the smallest class of linear orders obtained by closing $\LOfp$ under shuffles of finite subsets \cite[\S~2]{SKSK}. This result forms the backbone of the current paper. In fact as far as the proof of Theorem \ref{oldconj} is concerned, \cite{SKSK} is Part I while the current paper is Part II. 

The classes $\LOfp\subset\LOfd$ were introduced earlier in a model-theoretic study of linear orders by L\"{a}uchli and Leonard \cite{leonard1968elementary} using notations $\mathcal M_0\subset\mathcal{M}$ \cite[Definition~7.19]{rosenstein} to study a graded version of elementary equivalence between linear orders described via Ehrenfeucht-Fra\"{i}ss\'{e} games. In Section \ref{sec: new conjecture}, we build upon this connection between the logic of linear orders and the representation theory of finite-dimensional algebras to conjecture a stable rank characterization of the trichotomy of Drozd and Crawley-Boevey (Conjecture \ref{newconj})).

\subsection*{A sketch of the proof of Theorem \ref{oldconj}}

Suppose $ \Lambda = \mathcal K \mathcal Q / \lan \rho \ran $ is a string algebra. Given indecomposable modules $M,N$ in $\Lambda\dmod$, a basis for the finite-dimensional vector space $ \mathrm{Hom}_\Lambda(M, N) $ was given by \cite{crawley1989maps} and \cite{krause1991maps} (Theorem \ref{hammockgraphmap}). Schr\"oer \cite[Chapter~2]{Schroer98hammocksfor} associated three hammock posets to a vertex $v$ of the quiver $\mathcal Q$. The largest of these, $ \HH(v) $, is the set of isomorphism classes of all $ v $-pointed indecomposables in $ \Lambda\dmod $ with respect to the existence of pointed homomorphisms.
The simplified hammock $ \hb(v) $ is obtained by identifying all band modules in $ \HH(v) $ corresponding to the same band, while the hammock $ H(v) $ is obtained by removing all points in $ \hb(v) $ corresponding to band modules.

The poset $ H(v) $ is a subposet of the Cartesian product of two discrete bounded finite description linear orders. We generalize this to study $ n $-dimensional analogues called  \emph{abstract $ n $-hammocks} (Definition \ref{hammocks}) that are certain finite unions of ``cuboidal" posets. In this language, $ H(v) $ is an abstract $ 2 $-hammock.

If $ \Lambda $ is non-domestic, then $ H(v) $ is not necessarily scattered, however the stable rank computation only depends on the scattered regions in $ H(v) $ (Theorem \ref{noninfscatteredequivalence}). To leverage this, we introduce \emph{maximal scattered boxes} (\emph{MSBs}, for short) (Definition \ref{msbdefn}) as the convex scattered subposets of an $ n $-hammock that are maximal under inclusion. 
If $ L \in \LOfd $ then it is constructed using only finitely many operations. As a consequence, the number of isomorphism classes of maximal scattered intervals in $L$ is finite. We say that an abstract $n$-hammock is of \emph{finite type} (Definition \ref{fintypedefn}) if there are only finitely MSBs up to order isomorphism. In fact, we show that if all projections of an abstract hammock are finite description linear orders then it is of finite type (Lemma \ref{hammockfintype}). In particular, $H(v)$ is of finite type.

We next recall from \cite{Schroer98hammocksfor} the definition of \emph{density} of a hammock poset, which is a measure of the complexity of its linear suborders, and generalize it to all posets. Schr\"oer gave an algorithm (Theorem \ref{Schroer formula}) to compute the density of  $ H(v) $ for a domestic string algebra.
In fact, his algorithm works for any bounded $ n $-hammock. Since an MSB in $ H(v) $ could be an unbounded poset, but has periodic \emph{corners}, we give a formula (Corollary \ref{about density computation}) to compute the density of such MSBs.
We show that the density of an MSB in an abstract $ n $-hammock with finite description projections is strictly less than $ \omega^2 $ (Theorem \ref{fpdensity}).
Through the generalization to abstract $ n $-hammocks, we are able to avoid a lot of the technicalities of string algebras.

Denote the closure in $ \hb(v) $ of an MSB $ \ii $ in $ H(v) $ by $ \ol\ii $. The set $ \QBa_0 $ of all bands for $ \Lambda $ is equipped with a natural ``reachability" preorder , whose posetal reflection $ \QBa $ is a finite poset \cite[Proposition~5.6]{SKSK}--this poset is a simplification of the bridge quiver. Associated to each $ \sB \in \QBa $ are four finite subsets which control the size of $ \ol\ii \setminus \ii $ for an MSB $ \ii $ (Lemma \ref{excptfin}).

The stable rank $ \st(\Lambda) $  differs from the maximum density of an MSB $ \ii $ included in some hammock $ H(v) $  only by the size of $ \bigcup (\ol\ii \setminus \ii) $, where $ \ii $ varies over all MSBs in $ H(v) $ (Theorem \ref{density=rank-1}). We already argued the finiteness of this union, and thus the bound on the density of an MSB in $ H(v) $ completes the proof of Theorem \ref{oldconj}.

Several finiteness results lie at the heart of this proof. We describe in detail algorithms for computing these finite sets in \S~\ref{sec: alg}, thereby facilitating effective stable rank computation.

\subsection*{Contents of sections} 
After recalling the basic vocabulary of linear orders in \S~\ref{subsection: linear orders}, 
we introduce and study \emph{abstract $ n $-hammocks} along with \emph{maximal scattered boxes} and the important notion of \emph{finite type} in \S~\ref{subsection: abstract hammocks}. We show that any maximal scattered box in an $ n $-hammock is itself so (Proposition~\ref{boxes are hammocks}), and that finite description hammocks are of finite type (Lemma~\ref{hammockfintype}). We recall and generalize density in \S~\ref{density}, which allows us to also define Hausdorff rank for all posets (Definition \ref{rank of poset}). We then compute the density of \emph{corners} (Theorem~\ref{Our formula}) and show how to compute the density of an unbounded finitely presented hammock using any of its \emph{corner decompositions} (Corollary~\ref{about density computation}).

The ultimate goal of \S~\ref{section: Proofs} is to prove Theorem \ref{oldconj}. The combinatorial and representation-theoretic vocabulary pertaining to special biserial algebras and string algebras, including the description of the A-R quiver for a string algebra, is recalled in \S~\ref{sec: Preliminaries}. The study of three different types of hammocks and the interplay between them forms the crux of \S~\ref{sec: hammock lin order}-\ref{conjproofstringalg}. The results of \S~\ref{order theory} for concrete hammock posets for string algebras are collected in Theorem~\ref{hammockordermain}. We prove a crucial finiteness result (Lemma \ref{excptfin}) using the finiteness of some sets of bands. We then establish a connection between the ranks of morphisms in the category $\Lambda\dmod$ and the densities of intervals in hammocks (Theorems \ref{density=rank-1}, \ref{noninfscatteredequivalence}), and show that the stable rank of a string algebra $\Lambda$ can be computed using the densities of finitely many scattered boxes in extended hammocks (Equation \eqref{stablerank7}). The proof of Theorem \ref{oldconj} for string algebras is completed at the end of \S~\ref{conjproofstringalg}, and its extension to special biserial algebras is achieved in \S~\ref{biserial}. 

The purpose of \S~\ref{sec: alg} is to present an algorithm to estimate the stable rank of a string algebra up to finite error (Equation \ref{stablerank8}). In \S~\ref{sec:H-eq}, we study different versions of $H$-equivalence relations on strings; each one of which partitions the set of strings into finitely many equivalence classes. In the next two technical subsections, we prove several other finiteness results to describe a way to effectively compute the density of an MSB, and hence to estimate the stable rank. A key ingredient is an adaptation of the algorithm by Schr\"oer \cite[\S~4]{Schroer98hammocksfor} to compute the density of a scattered hammock linear order discussed at the end of \S~\ref{sec:alg msi lin ord}. Through an explicit computation of the stable rank for a string algebra using this algorithm (Example \ref{ex: Proof of theorem 2}), we provide a proof of Theorem \ref{reversest}.

The last section, \S~\ref{sec: new conjecture}, is reserved to state the conjecture on the characterization of Drozd and Crawley-Boevey's trichotomy of finite-dimensional algebras in terms of stable ranks (Conjecture \ref{newconj}) after a few reasons for our belief in it.  

\subsection*{Acknowledgements}
The third author would like to thank Marcus Tressl for elaborating on the connection between order-theory and logic, and Charles Paquette for answering questions about special biserial algebras.

\section{Computing the densities of abstract hammocks}
The goal of this section is to develop order-theoretic tools necessary for the proof of Theorem \ref{oldconj}.

We begin by recalling the definitions of the classes $ \LOfp $ and $ \LOfd $ in \S~\ref{subsection: linear orders}.
We  then introduce and study in \S~\ref{subsection: abstract hammocks} \emph{abstract $ n $-hammocks}  (Definition~\ref{hammocks}), which are not-necessarily-bounded $ n $-dimensional generalizations of Schr\"oer's hammock posets \cite[\S~3.8]{Schroer98hammocksfor}  .
A key difference in our treatment is that we construct hammocks by taking finite unions of ``cuboidal" subposets instead of removing finitely many       ``antidiagonal-corners" from the ambient poset.

We also introduce \emph{maximal scattered boxes} (Definition~\ref{msbdefn}), which are those convex scattered subposets of an $ n $-hammock that are maximal under inclusion. These feature prominently in the stable rank computation in \S~\ref{excpoints}, \ref{conjproofstringalg}. 
Two key results are Theorem~\ref{boxes are hammocks}, which implies that maximal scattered boxes in an abstract $ n $-hammock are so themselves, and Theorem~\ref{lofdfintype}, which states that finite description $ n $-hammocks are of \emph{finite type}, i.e., have finitely many maximal scattered boxes up to order isomorphisms.
These allow us to reduce the study of a finite type
$ n $-hammock to that of finitely many scattered $ n $-hammocks.

In \S~\ref{subsection: density} we recall  the notion of \emph{density} from \cite{Schroer98hammocksfor} and generalize it to all posets. 
Finally, in \S~\ref{subsection: corners} we compute the densities of certain unbounded hammocks with periodic \emph{corners}. The major takeaway from this section is Theorem~\ref{fpdensity} which plays a central role in the proof of Theorem~\ref{oldconj} in \S~\ref{section: Proofs} by providing a bound on
the density of a  \emph{finitely presented hammocks}, i.e., a hammock all of whose projections are \emph{finitely presented linear orders}.

\label{order theory}
\subsection{Linear orders}
\label{subsection: linear orders}
We start by recalling some basic facts and notions abour linear orders.

We denote the set of natural numbers, i.e., the set of all non-negative integers, by $ \N $ and the set of all positive (resp. negative) integers by $ \N^+ $ (resp. $ \N^- $). For each $n\in\N$, we will use the notation $ [n] $ for the set $ \{1, 2, \dots, n\} $.

We will use standard order-theoretic notation, for which the reader is referred to \cite{rosenstein}.

\begin{defn}
Given a linear order, we define its \emph{order type} to be its equivalence class under order isomorphisms.
    For any natural number $ n $, we denote by $ \mbf{n} $ the order type of the finite linear order with exactly $ n $ elements.
The order types of $ \N^+ $, $ \rat$, $\N^- $ and $ \Z $ with respect to their standard orderings are denoted by $ \omega, \omega^*,  \eta $ and $ \zeta $ respectively.
\end{defn}
\begin{defn}\label{defn: dense and scattered}
A linear order $ L $ is said to be
    \begin{itemize}
        \item \emph{dense} if $ (\forall x, y \in L) [(x < y) \Rightarrow (\exists z \in L) (x < z < y)] $;
        \item \emph{discrete} if each element that is not minimum (resp. maximum) has an immediate predecessor (resp. immediate successor);
        \item \emph{scattered} if it admits no embedding of $ \eta $.
    \end{itemize}
\end{defn}

Linear orders can be added and multiplied.

\begin{defn}
    Given a linear order $L$ and a sequence $ \lan L_a \mid a \in L \ran $ of linear orders, we define the ordered sum $ \sum\limits_{a \in L}L_a $
    to be the set $ \{(a, x) \mid a \in L, x \in L_a\} $ with the  lexicographic order relation $ (a, x) < (a', x') $ if and only if $ a < a' $ or $( a = a' $ and $ x < x')$.

    Given linear orders $ L_1 $ and $ L_2 $, we define the \emph{sum} $ L_1 + L_2 $ to be the ordered sum $ \sum_{i \in \mbf{2}} L_i $, and the \emph{product} $ L_1 \cdot L_2 $ to be the ordered sum $ \sum_{x \in L_2} L_1 $.

    The \emph{dot sum} $L_1 \dot+ L_2 $ of linear orders $ L_1, L_2 $ is obtained from $ L_1 + L_2 $ by identifying the maximum of $ L_1 $ with the minimum of $ L_2 $, if these exist.
\end{defn}
\begin{rem}\label{rem: plus and plusdot}
    If $L_1, L_2, \cdots L_n$ are bounded discrete linear orders and $L_1 + L_2 + \cdots + L_n$ is an infinite order, then $L_1 \dot+ L_2 \dot+ \cdots \dot+ L_n \cong L_1 + L_2 + \cdots + L_n $.
\end{rem}

If $L=L_1+L_2$ then we say that $L_1$ is a \emph{prefix} of $L$ and $L_2$ is a \emph{suffix} of $L$.

The following is an important closure property of the class of scattered linear orders.
\begin{prop}\cite[Proposition~2.17]{rosenstein}\label{prop: SCattered sum of linear orders}
    If $L$ is a scattered linear order and $\lan L_i\mid i \in L\ran$ is a sequence of linear orders indexed by $L$, then $ \sum\limits_{i \in L} L_i $ is scattered if and only if $L_i$ is scattered for every $i \in L$.
\end{prop}
Recall the structure theorem for linear orders  due to Hausdorff \cite{Hausdorff1908}.
\begin{thm}\cite[Theorem~4.9]{rosenstein}\label{Hausdorffstrthm}  Every linear order can be written as a dense sum of scattered linear orders, i.e., for any linear order $L$ there exists a dense linear order $L'$ and a sequence $\lan L_i \mid i \in L'\ran$ of scattered intervals of $L$ such that $L \cong \sum \limits_{i \in L'} L_i$.
\end{thm}

The linear order $L'$ in the above theorem is (isomorphic to) the \emph{scattered condensation} of $L$.
\begin{defn}\cite[\S~4.5]{rosenstein}\label{sccondlinorder}
Given a linear order $L$ and $x\in L$, define its \emph{scattered condensation class} to be the set $c_S(x):= \{y \mid [\min\{x, y\},\max\{x,y\}]\text{ is scattered}\}$. Clearly $c_S(x)$ is an interval in $L$ and the set $c_S(L):=\{c_S(x)\mid x\in L\}$ equipped with  a linear order structure defined by $c_S(x)<c_S(y)$ if $x'<y'$ for each $x'\in c_S(x)$ and $y'\in c_S(y)$ is called the \emph{scattered condensation of $L$}.
\end{defn}
The linear order $c_S(L)$ is dense and the map $x\mapsto c_S(x):L\to c_S(L)$ is a surjective monotone map. 

The class of all linear orderings is very complex, so we restrict our attention to those linear orders which can be described by a finite amount of information. It is no surprise that such linear orders are ubiquitous. Two important such classes are the classes of finitely presented and finite description linear orders.

\begin{defn} \cite[Definition~2.3]{SKSK}
The class $ \LOfd $ of finite description linear orders is the smallest class closed under isomorphisms containing $ \mbf 0, \mbf 1 $ and closed under the following operations:
\begin{enumerate}
    \item $ L \mapsto \omega \cdot L $;
    \item $ L \mapsto \omega^* \cdot L $;
    \item $ \lan L_1, L_2\ran \mapsto L_1 + L_2 $;
    \item for each $n\in\mathbb N^+$, $ \lan L_1, L_2, \cdots, L_n\ran \mapsto \Xi(L_1, L_2, \cdots, L_n) $, where the $ n $-ary operator $ \Xi(L_1, L_2, \dots, L_n) $ denotes the shuffle of the linear orders $ L_1, L_2, \dots, L_n $.
\end{enumerate} 
The subclass $ \LOfp $ of finitely presented linear orders is obtained after omitting $(4)$ from the above definition.
\end{defn}

The next observation is the link between finite description and finitely presented linear orders.
\begin{rem}\label{LOfp=LOfd cap Scattered}
A finite description linear order is scattered if and only if it is finitely presented.
\end{rem}

The proof of the following proposition is an obvious extension of the proof of a similar result \cite[Proposition~5.7]{agrawal2022euclidean} for finitely presented linear orders, where the extension uses 4-signed rooted trees for dealing with the shuffle operator. 
\begin{prop}
\label{fd intervals}
    If $L\in\LOfd$ and $ L = L_1 + L_2 $, then $ L_1,L_2\in\LOfd$. Consequently if $I\subseteq L$ is an interval then $I\in\LOfd$.
\end{prop}

\subsection{Abstract hammocks}
\label{subsection: abstract hammocks}
\label{clats}
\label{abstract hammocks}
In this section we study certain posets built from finite Cartesian products of linear orders, which we call \emph{$n$-hammock posets} (Definition \ref{hammocks}). Next we define \emph{boxes} (Definition \ref{boxdefine}) as ``cuboidal" subposets of hammocks and show that any box in a hammock is a hammock (Proposition \ref{boxes are hammocks}).

We extend the scattered condensation operation from linear orders to $ n $-hammocks (Definition \ref{sccondhammocks}) and 
show that the scattered condensation classes of an $ n $-hammock are \emph{maximal scattered boxes} (Remark \ref{classes are msbs}). The highlight of this section is Proposition \ref{hammockfintype} which states that a finite description $ n $-hammock, i.e., an $n$-hammock whose projections are all finite description linear orders, is of finite type, i.e., has finitely many maximal scattered boxes up to order isomorphism (Definition \ref{fintypedefn}).

\begin{defn}\label{hammock define}
For $n\in\N^+$, an \emph{$n$-dimensional Cartesian lattice} (or \emph{clat}, for short) is a poset (isomorphic to one) of the form $\left(\prod\limits_{i \in [n]} L_i, \leq:=\prod\limits_{i\in[n]}\leq_i\right)$ for a finite sequence $\lan (L_i, \leq_i)\mid i \in [n]\ran$ of linear orders.

For each $i \in [n]$ there is a natural monotone projection map $\pi_i$ defined as $\pi_i(\bar x) = x_i$ for $\bar x\in\prod_{i\in[n]}L_i$.

A subposet $S \subseteq \prod\limits_{i \in [n]} L_i$ is called a \emph{sub-$n$-clat} if $S =\prod\limits_{i\in[n]} \pi_i(S)$.
\end{defn}

\begin{rem}\label{clatlat}
An $n$-clat is indeed a distributive lattice with respect to componentwise $ \min $ and $ \max $ operations as meet and join respectively.
\end{rem}

\begin{defn}
For any subset $A$ of a linear order $L$, define $A^\da := \{b \in L \mid (\forall a \in A) (b < a)\}$. Similarly define $A^\ua := \{b \in L \mid (\forall a \in A) (b > a)\}$.
\end{defn}

Now we are ready to define the titular object of this subsection.
\begin{defn}
\label{hammocks}
A subposet $H$ of an $n$-clat $C$ is called an \emph{abstract $n$-hammock} (\emph{$n$-hammock}, for short) if there exists a finite sequence of sub-$n$-clats $\lan S_j \mid j \in [k] \ran$ such that the following hold.

\begin{enumerate}
    \item $H = \bigcup\limits_{j \in [k]} S_j$.
    \item For every $i \in [n]$, $\pi_i(H) = \pi_i(C)$, i.e., $H$ has full projections.
    \item  For every $i \in [n]$, $\pi_i(S_1) \cap \pi_i(S_n) \not = \emptyset$.
    \item For every $i \in [n]$ and $j \in [k - 1]$, the following hold:
    \begin{itemize}
        \item $\pi_i(S_j)^\da \cap \pi_i(S_{j+1})= \emptyset$;
        \item $\pi_i(S_j)^\ua \cap \pi_i(S_{j+1})^\da = \emptyset$;
        \item $\pi_i(S_j)\cap \pi_i(S_{j+1})^\ua = \emptyset$.
    \end{itemize}
\end{enumerate}
\end{defn}

For the rest of this section, fix an $n$-hammock $H=\bigcup_{j\in[k]}S_j$ in an $n$-clat $C$ with projections $(L_i,\leq_i)$ for $i \in [n]$ and a subclat decomposition $\lan S_j \mid j \in [k]\ran$ of $H$ as in the definition.

Continuing from Remark \ref{clatlat}, we get more examples of distributive lattices.
\begin{prop}\label{hammocksublat}
The $n$-hammock $H$ is a sublattice of $C$, and hence is distributive. \end{prop}

\begin{proof}
We will show that for any $\bar a,\bar b \in H $, $\bar a \wedge\bar b \in H $; the proof of closure under joins is symmetrical.

If $\bar a \leq\bar b $ then we are done, so assume that $\bar a $ and $\bar b $ are incomparable. Choose and fix $ i, j \in [n] $ such that $ a_i < b_i $ and $ a_j > b_j $.
Because $ a_i < b_i $, there exist $ r, s \in [k] $ such that $ r \leq s $, $\bar a \in S_r $ and $\bar b \in S_s $.
We claim that $\bar c := (a_1, a_2, \dots, a_{j-1}, b_j, a_{j+1}, \dots, a_n) \in H $.

Suppose not. Then $ b_j \in \pi_j(S_s) \setminus \pi_j(S_r) $. Then $ s > r $ and $ b_j \in \pi_j(S_r)^\ua \cup \pi_j(S_r)^\da $. Since $ \pi_j(S_r)^\da \cap \pi_j(S_s) = \emptyset $, we have  $ b_j \in \pi_j(S_r)^\ua $, and hence $ b_j > a_j $, which is a contradiction to our assumption.

Since $n$ is finite, the conclusion is immediate by repeated application of the claim.
\end{proof}

\begin{defn}\label{def:discrete-finpresented hammock}
Say $H$ is \emph{discrete} (resp. \emph{finitely presented, finite description)} if all its projections are so.
\end{defn}

\begin{rem}\label{hammockbdd}
A hammock $H$ is bounded, i.e., it contains a maximum and a minimum element, if and only if all its projections are so.
\end{rem}

\begin{prop}
The hammock $H$ is dense as a poset, i.e, every maximal chain (with respect to inclusion) in $H$ is dense, if and only if all its projections are dense.
\end{prop}
\begin{proof}
($\Rightarrow$) Suppose $H$ is dense and $y_j<z_j$ in some $L_j$, $j\in[n]$. For each $i\neq j$, choose and fix $x_i\in\pi_i(S_1)\cap\pi_i(S_k)$ using condition $(3)$ of the definition of a hammock. Then using the other conditions of the same definition, $(x_1,\cdots,x_{j-1},y_j,x_{j+1},\cdots,x_n)<(x_1,\cdots,x_{j-1},z_j,x_{j+1},\cdots,x_n)$ in $H$. Therefore, the density of $H$ implies the density of $L_j$.

($\Leftarrow$) Suppose $L_i$ is dense for each $i\in[n]$. If $\bar y<\bar z$ in $H$ then for each $i\in[n]$, whenever $y_i<z_i$ then choose $y_i<w_i<z_i$; otherwise choose $w_i:=y_i=z_i$, so that $\bar x<\bar w<\bar z$.
\end{proof}

We will prove a similar result for scattered hammocks, but in much more generality. For this purpose, we now introduce boxes as ``cuboidal" intervals in $H$. For intervals $I_i \subseteq L_i$, define $\bigotimes\limits_{i \in [n]} I_i := H \cap \prod\limits_{i \in [n]} I_i$.
\begin{defn}
\label{boxdefine}
Say that a subposet $X$ of $H$ is a \emph{box} if $\pi_i(X)$ is an interval in $L_i$ for each $i \in [n]$ and $X = \bigotimes\limits_{i\in [n]}\pi_i(X)$.
\end{defn}

\begin{rem}
\label{box intersection}
    The set of boxes in a hammock is closed under finite intersections.
\end{rem}

\begin{prop}
\label{stdproj}
    A box $X$ in $H$ is scattered if and only if all its projections are scattered.
\end{prop}
\begin{proof}
    ($\Leftarrow$) We will prove the stronger statement that an arbitrary subset of $H$ is scattered if all its projections are scattered using induction on $n$. The base case, when $n = 1$, is trivial. Assume that the statement holds when $n=m$ and we will prove it for $n = m+1$.
    
    Suppose $f:\eta\to H$ is a strict monotone map. We will show that $\pi_i(f(\eta))$ is not scattered for some $i \in [n]$. If  $\pi_1(f(\eta))$ is scattered, then $\eta \cong f(\eta) = \sum\limits_{a \in \pi_1(X)} (\pi_1^{-1}(a) \cap f(\eta))$. Using Proposition \ref{prop: SCattered sum of linear orders}, choose $a_1 \in \pi_1(X)$ such that $X_1 := \pi_1^{-1}(a_1) \cap f(\eta)$ is not scattered.

    Let $C' := \prod\limits_{2 \leq i \le n} \pi_i(H) $ and consider the monotone embedding $\phi := (\pi_2, \hdots, \pi_{n}): X_1 \to C'$ sending $(a_1, \hdots, a_n)$ to $ (a_2, \hdots, a_n)$.
        It is straightforward to verify that $\phi(H)$ is an $(n-1)$-hammock in the $(n-1)$-clat $C'$, and that $\phi(X_1)$ is a box in $\phi(H)$.
    Hence by the inductive hypothesis, we have $\pi_i(\phi(X_1)) = \pi_i(X_1) \subseteq \pi_i(f(\eta))$ is not scattered for some $2 \le i \le n$. 

    $(\Rightarrow)$
    Without loss, suppose $\pi_1(X)$ is  not scattered. Using the axiom of choice,  pick an element $b_a \in X \cap \pi_1^{-1}(a)$ for each $a \in \pi_1(X)$. Then $\sum\limits_{a \in \pi_1(X)} \{b_a\} \cong \eta \subseteq X$.
\end{proof}

Now we show that boxes in hammocks are hammocks themselves. We start with an easy observation about intersecting intervals in linear orders.

\begin{prop}
\label{sunflower}
    If $ L, M, R $ are intervals in a linear order such that their pairwise intersections are nonempty, then $ L \cap M \cap R $ is also nonempty.
\end{prop}
\begin{proof}
    Let $ x \in L \cap M $, $ y \in L \cap R $ and $ z \in M \cap R $.
    Suppose $ x \not \in  R $. Then either $ x < \min \{y, z\} $ or $ x > \max \{y, z\}$. Without loss of generality assume the former.
    If $ y \in M $ then we are done. Otherwise, we have $ x < z < y $, so that $ z \in [x, y] \subseteq L $. Then $ z \in L \cap M \cap R$, as needed.
\end{proof}
The following corollary follows immediately by induction.
\begin{cor}
    If $n \in \N$ and $ \lan L_i \mid i \in [n] \ran $ is a finite sequence of intervals in a linear order with pairwise nonempty intersections, then $ \bigcap_{i \in [n]} L_i \not= \emptyset $. 
\end{cor}

\begin{prop}
\label{assume intersection}
   For $ H $ and $ \lan S_j \mid j \in [k] \ran $ as fixed above, if $ m, M  \in [k] $ are such that $ m \le M $ then $ \bigcup_{m \le j \le M} S_j $ is a hammock in the clat $ \prod_{i \in [n]} \bigcup_{m \le j \le M} \pi_i(S_j) $.
\end{prop}
\begin{proof}
    First assume that $ m = 2 $ and $ M = k $. 
    
Let $ i \in [n] $.
Because $ \lan S_j \mid j \in [k] \ran $ is a sub-$ n $-clat decomposition, $ \pi_i(S_1) \subseteq (\pi_i(S_2) \cup \pi_i(S_2)^\da) $.
Similarly $ \pi_i(S_n) \subseteq (\pi_i(S_2) \cup \pi_i(S_2)^\ua) $.
    Therefore $ \pi_i(S_1) \cap \pi_i(S_n) \subseteq \pi_i(S_2) \cap \pi_i(S_n)$. The left hand side is nonempty because $ \lan S_j \mid j \in [k] \ran $ is a sub-$ n $-clat decomposition, hence so is the right hand side.

Because our choice of the hammock $ H $ was arbitrary, the result holds for all hammocks. Therefore by induction, the result holds for all $ 1 \le  m \le M \le k$.
\end{proof}
\begin{prop}
\label{boxes are hammocks}
    If $ \ii $ is a nonempty box in a hammock $ H $, then $ \ii $ is a hammock in $ \prod_{i \in [n]} \pi_i(\ii) $.
\end{prop}
\begin{proof}
    Consider the sequence $ \lan S'_j \mid j \in [k] \ran $ of clats in $ H $, where $ S'_j \defeq S_j \cap \ii $.

\noindent{\textbf{Claim:}}
        $ J \defeq \{ j \in [k] \mid S'_j \not= \emptyset \} $ is a contiguous sequence.

\noindent{\textit{Proof of the claim.}} 
        Let $ m \defeq \min J $, $ M \defeq \max J $ and $ m < k < M $. Then by Proposition \ref{sunflower}, $ \pi_i(\ii) \cap \pi_i(S_m) \cap \pi_i(S_M) \cap \pi_i(S_k) \not= \emptyset $ for each $i\in[n]$.\hfill$\vartriangle$
        
    Thanks to Proposition \ref{assume intersection}, without loss of generality, we can assume that $ J = [k] $. 
    Note that $ S'_j = \prod_{i \in [n]} (\pi_i(S_j) \cap \pi_i(\ii)) $ is a clat for every $ j \in [k] $.
    It remains to show that $ \lan S'_j \mid j \in [k]\ran $ is a subclat decomposition for $ \ii $.
    It is clear that conditions $(1), (2)$ and $(4)$ of Definition \ref{hammocks} are met.
    To see that condition $(3)$ also holds, note that for each $ i \in [n] $, $ \pi_i(S'_1) \cap \pi_i(S'_k) = \pi_i(S_1) \cap \pi_i(\ii) \cap \pi_i(S_k) \not= \emptyset$, thanks to Proposition \ref{sunflower}.
\end{proof}

Using Proposition \ref{hammocksublat} the definition of scattered condensation (Definition \ref{sccondlinorder}) can be generalized to hammocks since they are distributive, and hence modular, lattices.
\begin{defn}
\label{sccondhammocks}
For $\bar x\in H$, define $c_S(\bar x):=\{\bar y\in H\mid[\bar x\wedge\bar y,\bar x\vee\bar y]\text{ is scattered}\}$.
\end{defn}
The next result is an immediate consequence of Proposition \ref{stdproj}.
\begin{prop}\label{fintypeproj}
Given an $n$-hammock $H$ and $\bar x,\bar y\in H$, $c_S(\bar x)=c_S(\bar y)$ in $c_S(H)$ if and only if $c_S(x_i)=c_S(y_i)$ in $c_S(L_i)$ for each $i\in[n]$. As a consequence, $c_S(\bar x)$ is a box.
\end{prop}
In view of the above proposition, Theorem \ref{Hausdorffstrthm} generalizes to hammocks.  
\begin{prop}
The set $c_S(H):=\{c_S(\bar x)\mid\bar x\in H\}$ equipped with a partial order structure $c_S(\bar x)\leq c_S(\bar y)$ if $\bar x'\leq\bar y'$ for some $\bar x'\in c_S(\bar x)$ and $\bar y'\in c_S(\bar y)$ is a dense hammock.
\end{prop}
The partial order structure on $c_S(H)$ in the above proposition can also be obtained by considering the congruence relation corresponding to the class of all scattered modular lattices in the sense of \cite[\S~10.1]{prest2003model}. Since this class is closed under sublattices and quotients, we can apply \cite[Lemma~10.3]{prest2003model} to the modular lattice $ H $ to see that $ x \approx y \iff x \in c_S(y) $ indeed defines a congruence relation on $ H $.

We are interested in those hammocks which can be described using only a finite amount of scattered data.
\begin{defn}
\label{fintypedefn}
Say that a hammock $H$ is \emph{of finite type} if the set of order types(=order-isomorphism classes) of elements of $c_S(H)$ is finite.
\end{defn}

If $H$ is a finitely presented hammock then $c_S(H)=\mathbf{1}$, and hence $H$ is of finite type. In fact the conclusion can be generalized to all finite description clats.

\begin{lem}\label{lofdfintype}
If $C$ is a finite description $n$-clat then $C$ is of finite type.
\end{lem}
In view of Proposition \ref{fintypeproj}, it is enough to show that if $L\in\LOfd$ then $L$ is of finite type--the proof of this statement is immediate once the next proposition is established. For a linear order $L$ of finite type, let $m_S(L)$ denote the number of order types of maximal scattered intervals.

\begin{prop}\label{ordtypecomplexity}
Let $L,L',L_1,\cdots,L_p$ be linear orders of finite type. Then the following hold.
\begin{itemize}
    \item $m_S(\mathbf 0)=m_S(\mathbf 1)=1$.
    \item $m_S(L+L')\leq m_S(L)+m_S(L')$.
    \item $m_S(L\cdot\omega)\leq m_S(L)+1$ and $m_S(L\cdot\omega^*)\leq m_S(L)+1$.
    \item $m_S(\Xi(L_1,\cdots,L_p))\leq\sum_{j\in[p]}m_S(L_j)$.
\end{itemize}
\end{prop}
The proof of this result is a straightforward exercise.

\begin{defn}
\label{msbdefn}
A box $\ii$ in $H$ is said to be a \emph{maximal scattered box} (\emph{MSB}, for short) if it is maximal under inclusion among all scattered boxes in $H$.
\end{defn}

\begin{rem}
\label{classes are msbs}
    Let $\ii$ be a scattered box in $H$ containing $\bar x$. Since $\ii$ is a distributive lattice, for any $\bar y \in \ii$ the interval $[\bar x \wedge \bar y, \bar x \vee \bar y]$ is contained in $\ii$, and therefore scattered. Hence $\ii \subseteq c_S(\bar x)$. As a consequence, $c_S(\bar x)$ is an MSB for each $\bar x\in H$.
\end{rem}

Proposition \ref{stdproj} yields a sufficient condition for a box to be an MSB.
\begin{cor}\label{stdprojcor}
If $ \ii $ is a box such that $ \pi_i(\ii) $ is a maximal scattered interval in $ L_i $ for each $ i\in [n] $, then $ \ii $ is a maximal scattered box.
\end{cor}

\begin{proof}
Suppose not. Let $\ii'$ be a scattered box such that $\ii \subsetneq \ii'$. Since a box is completely determined by its projections, choose $i \in [n]$ such that $\pi_i(\ii)\subsetneq \pi_i(\ii')$. Since $\pi_i(\ii)$ is a maximal scattered interval, $\pi_i(\ii')$ is not scattered. Then by Proposition \ref{stdproj}, $\ii'$ is not scattered as well, which is a contradiction.
\end{proof}

However, the converse of the above result fails as the following example demonstrates.
\begin{example}
Let orders $A, B, C, D$ be copies of $\omega^*,\omega,\zeta,\eta$ respectively. Consider the $2$-clat $(A+B) \times (C+D)$. Fix a hammock $H$ with subclat decomposition $\lan A \times (C + D), (A+B) \times D\ran$. Then $A \times C$ is a maximal scattered box in $H$ but $A$ is not a maximal scattered interval in $(A+B) \cong \zeta$.
\end{example}

The next result is what we promised at the beginning of the section.

\begin{lem}\label{hammockfintype}
    If $H$ is a finite description $ n $-hammock then $H$ is of finite type.
\end{lem}

\begin{proof}
    We use the subclat decomposition $ \lan S_j \mid j \in [k] \ran $ for $ H $.
    For each $ i \in [n] $, divide $ \pi_i(H) $ into intervals $ \lan I^i_t \mid t \in [N] \ran $ such that 
     for every $ t \in [N] $ and $ j \in [k] $, either $ I^i_t \cap \pi_i(S_j) = \emptyset $ or $ I^i_t \subseteq \pi_i(S_j) $.
    (We are assuming that $ N $ does not depend upon $ i $ since we allow empty intervals.)
    Each $I^i_t$ is then a finite description linear order by Proposition \ref{fd intervals}.

    For each function $ f \colon [n] \to [N] $, let $ \ii_f \defeq \bigotimes_{i \in [n]} I^i_{f(i)} $.
    It is clear that $ \{\ii_f \mid f \colon [n] \to [N]\} $ partitions $ H $.
    The claim below shows that each $ \ii_f $ is a finite description $ n $-clat, and hence of finite type by Proposition \ref{lofdfintype}. 
    
\noindent{\textbf{Claim:}} If $ f \colon [n] \to [N] $ is such that $ \ii_f \not= \emptyset $ then $ \ii_f = \prod_{i \in [n]} I^i_{f(i)} $.

\noindent{\textit{Proof of the claim.}} Let $ \bar x \in \ii_f $. Let $ j \in [k] $ be such that $ \bar x \in S_j $. Then for each $ i \in [n] $, $ I^i_{f(i)} \cap \pi_i(S_j) \not= \emptyset $, and hence $ I^i_{f(i)} \subseteq \pi_i(S_j) $. Therefore, $ \prod_{i \in [n]} I^i_{f(i)} \subseteq S_j \subseteq H $.\hfill$\vartriangle$
    
    Let $ A $ be an \MSB in $ H $. 
    Then $ \{A \cap \ii_f \mid f \colon [n] \to [N]\} $ is a finite partition of $A$ into boxes, thanks to Remark \ref{box intersection}. 
    To complete the proof it suffices to show that, for each $ f $, the box $ A \cap \ii_f $ is either empty or an \MSB in $ \ii_f $.
     
    Let $ f \colon [n] \to [N] $ and $ \bar x \in H $ be such that $ \bar x \in A \cap \ii_f $. 
    Then $ c_S^f(\bar x) \defeq c_S(\bar x) \cap \ii_f $ is the \MSB in $ \ii_f $ containing $ \bar x $, where $c_S$ is the scattered condensation operator on the $ n $-clat $C$ containing $H$.
    Clearly $ c_S^f(\bar x) \subseteq A \cap \ii_f $.
    If $ \bar y \in (A \cap \ii_f) \setminus c_S^f(\bar x) $ then $ [\bar x \wedge \bar y, \bar x \vee \bar y] $ is not scattered.
    However, since $ A \cap \ii_f $ is a box, $ [\bar x \wedge \bar y, \bar x \vee \bar y] \subseteq A\cap \ii_f $, contradicting the hypothesis that $A$ is an MSB. 
    Therefore $ A \cap \ii_f = c_S^f(\bar x) $ is an \MSB in $ \ii_f $.
\end{proof}

\subsection{The density of an abstract hammock}
\label{subsection: density}
\label{density}
The Hausdorff rank of a linear order is a measure of its complexity. The notion of density, introduced in \cite[\S~3.2]{Schroer98hammocksfor} for linear orders and certain bounded posets, is a finer invariant than the Hausdorff rank. We recall some elementary facts about Hausdorff ranks and density before extending both these notions to all posets (Definitions \ref{posetdensity} and \ref{rank of poset}). Our main focus is on the computation of the density of a poset using not-necessarily-bounded linear suborders to set the ground for the computation of the density for unbounded abstract $n$-hammock posets in the next section. At the end we recall Schr\"oer's formula (Theorem \ref{Schroer formula}) for the hammock poset of a domestic string algebra.

The following definition is motivated from \cite[\S~10.2]{prest2003model}.
\begin{defn}[Hausdorff condensation]
Let $(L, <)$ be a linear order. 
Inductively define for each ordinal $\alpha$ an equivalence relation $\sim^L_\alpha$ on $L$ such that the following hold for every $x, y \in L$ with $ x \le y $.
\begin{enumerate}
    \item We have $x \sim^L_0 y$ if and only if $x = y$.
    \item For every ordinal $\alpha$, $ x \sim^L_{\alpha+1} y$ if and only if $[x, y]/\sim^L_\alpha$ is finite.
    \item For every limit ordinal $\alpha$, $x \sim^L_\alpha y$ if and only if $x \sim^L_\beta y$ for some $ \beta < \alpha $.
\end{enumerate}
For an ordinal $\alpha$ and an element $x$ in $L$, let $[x]_\alpha$ denote the $\sim^L_\alpha$ equivalence class of $x$ in $L$. When $L$ is clear from context, it is dropped from the notation $\sim^L_\alpha$. Clearly these equivalence classes are convex, so the order relation $<$ on $L$ naturally projects onto an order relation, denoted $<^\parens\alpha$, on $L^\parens \alpha := \{[x]_\alpha\mid x \in L\}$. In simple words, for elements $ x, y $ in $ (L, <) $, we say that $ x <^\parens \alpha y $ if $ x < y $ and $ x \not\sim_\alpha y $.

The least $\alpha$, if it exists, such that $L^\parens\alpha$ is finite is called the \emph{Hausdorff rank} of $L$ and we write $\HR(L) = \alpha$. If no such $\alpha$ exists then we write $\HR(L) = \infty$.
\end{defn}
\begin{rem}
    \label{rank is increasing}
    If $ L $ is a linear order and $ I \subseteq L $ then $ \HR(I) \le \HR(L) $.
\end{rem}
\begin{rem}\label{HRfindesc}
If $L\in\LOfp$ then $\HR(L)<\omega$. 
\end{rem}

\begin{prop}
\label{sim and interval rank}
    If $ L $ is a linear order, $ x,y\in L $ and $n\in\N^+$ then $ x <^\parens n y $ if and only if $ \HR((x, y)) \ge n $.
\end{prop}
\begin{proof}
    We have 
    $x <^\parens n y \iff [x, y]/\sim_{n-1} \text{ is infinite } \iff (x, y)/\sim_{n-1} \text{ is infinite } \iff \HR((x, y)) \ge n.$
    Here the first and third equivalences follow directly from the definitions. For the second equivalence, note that $ ([x, y]/\sim_{n-1}) \setminus ((x, y)/\sim_{n-1}) \subseteq \{[x]_{n-1}, [y]_{n-1}\} $.
\end{proof}

Let us now recall the definition of the density of a linear order from \cite[\S~3.2]{Schroer98hammocksfor}.

\begin{defn}
\label{linear density}
Suppose $ L $ is a linear order. If $ \HR(L) = \infty $, then define its \emph{density} as $d(L) \defeq \infty$.
Otherwise, if $ \HR(L) = \alpha $, let $p$ be an integer such that $L^\parens \alpha = \mbf {p+1}$. Then the \emph{density} of $L$, denoted $d(L)$ is defined to be $\omega\cdot \alpha + p$. In particular, $d(\mathbf 0)=-1$.
\end{defn}

\begin{rem}
\label{density is increasing}
If $ L $ is a linear order and  $I \subseteq L$ then $d(I) \le d(L)$.
\end{rem}

\begin{prop}
\cite[Exercise~5.12(1)]{rosenstein}
\label{absoluteness of sim}
    If $ I \subseteq L $ is convex then for every ordinal $ \alpha$,  $ \sim^I_\alpha = \sim^L_\alpha \cap (I \times I) $.
\end{prop}
The next result is a well-known theorem due to Hausdorff \cite{Hausdorff1908}.
\begin{thm}
\label{rank infinity iff not scattered}
    A linear order $ L $ has $ \HR(L) = \infty $, or equivalently $ d(L) = \infty $, if and only if it is not scattered.
\end{thm}

We note an immediate observation for the density of bounded linear order.
\begin{prop}\cite[Lemma~p.30]{Schroer98hammocksfor}\label{bddlindensitysuccessor}
Suppose $L$ is a bounded linear order with at least two elements. If $d(L)\neq\infty$ then $d(L)$ is a successor ordinal.
\end{prop}

The density operation is behaves particularly nicely with dot sums of bounded discrete linear orders. The proofs of the next two results are obvious generalizations of similar results from \cite[\S~3.6]{Schroer98hammocksfor}.
\begin{prop}\label{prop: density of sum of orders}
    If $L_1$ and $L_2$ are two bounded discrete linear orders with $d(L_i) = \omega \cdot \alpha_i + m_i$ for $i \in \{1,2\}$, then 
    $$ d(L_1 \dot+ L_2) = \begin{cases}
        \omega \cdot \alpha_1 + m_1 \quad&\text{if } \alpha_1>\alpha_2;\\
        \omega \cdot \alpha_2 + m_2 \quad&\text{if } \alpha_2>\alpha_1;\\
        \omega \cdot \alpha_1 + (m_1+m_2) \quad &\text{if }\alpha_1=\alpha_2.
    \end{cases} $$
\end{prop}
\begin{cor}\label{order irrelevant}
    The density of a finite dot sum of bounded linear orders does not depend on the permutation of the orders. 
\end{cor}

Here is a test for a lower bound on the density.
\begin{prop}
\label{witness for density}
For a linear order $L$, we have $d(L) > \omega \cdot \alpha + p $ if and only if there exists a chain $x_1 <^\parens \alpha x_2 <^\parens \alpha \dots <^\parens \alpha x_{p+2}$ in $L^{(\alpha)}$.
\end{prop}
\begin{proof}
    $(\Leftarrow)$
    Let $ x_1, x_2, \dots, x_{p+2} $ be as in the statement. 
    Notice that $ \HR(L) \ge \alpha $, for else we would have had $ [x_1]_\alpha = [x_2]_\alpha $ and $ x_1 \sim_\alpha x_2 $.
    If $ \HR(L) > \alpha $ then we are done, so assume that $ \HR(L) = \alpha $.
    Then $ [x_1]_\alpha < [x_2]_\alpha < \dots < [x_{p+2}]_\alpha $ are $ p+2 $ distinct elements in $ L^\parens \alpha $, so $ d(L) \ge \omega \cdot \alpha + (p + 1) $.

    $(\Rightarrow)$
    Suppose $ d(L) > \omega \cdot \alpha + p $. Then $ L^\parens \alpha $ contains at least $ (p + 2) $ many $\sim_\alpha$-equivalence classes, so choose elements $ x_1, x_2, \dots, x_{p+2} $ in $ L $ such that $ x_1 < x_2 < \dots < x_{p+2} $ and $[x_1]_\alpha \not= [x_2]_\alpha \not= \dots \not= [x_{p+2}]_\alpha$.
\end{proof}
The sequence of elements $\lan x_j \mid j \in [p + 2] \ran $ in the above proposition is said to \emph{witness} $ d(L) > \omega \cdot \alpha + p$.

The densities of bounded suborders of a linear order $ L $ determine the density of $ L $.

\begin{prop}
\label{bounded doesnt matter}
    For any linear order $ L $, $ d(L) = \sup\{d(L') \mid L' \subseteq L$ is a bounded linear suborder$\} $.
\end{prop}

\begin{proof}
    It is clear from Remark \ref{density is increasing} that the left hand side is an upper bound of the right hand side. For the other inequality, suppose that $ \alpha , p $ are such that $ d(L) > \omega \cdot \alpha + p $.
    It suffices to show that if  $ d(L) > \omega \cdot \alpha + p $, then there exists a bounded linear suborder $ L' $ of $ L $ with $ d(L') > \omega \cdot \alpha + p $.
     Use Proposition \ref{witness for density} to choose $ x_1 <^\parens \alpha x_2 <^\parens \alpha \dots <^\parens \alpha x_{p+2} $ in $ L $. 
    Let $ L' \defeq [x_1, x_{p+2}] $ so that $ L' $ is a bounded interval.
    By Proposition \ref{absoluteness of sim}, $ x_1 <^\parens \alpha x_2 <^\parens \alpha \dots <^\parens \alpha x_{p+2} $ in $ L' $ witness $ d(L') > \omega \cdot \alpha + p $, again by Proposition \ref{witness for density}.
\end{proof}
Motivated by the Proposition \ref{bounded doesnt matter}, let us define the density of partial orders.
\begin{defn}
\label{posetdensity}
    Define the \emph{density} of a poset $ P $ as $ d(P)\defeq \sup\{d(L) \mid L \subseteq P$ is a bounded linear suborder$\}$.
\end{defn}

The following proposition shows that we do not lose any information while working with bounded linear suborders.

\begin{prop}\label{densityposetunbdd}
Let $P$ be a poset. Then $d(P) = \sup\{d(T) \mid T$ is a linear suborder of $P\}$.
\end{prop}
\begin{proof}
    Since the set of bounded linear suborders of $P$ is included in the set of linear suborders of $P$, the right hand side is an upper bound for the left hand side. 
    
    For the other direction, let $T \subseteq P$ be an unbounded linear order. Then by Proposition \ref{bounded doesnt matter} there exists a set $\mathcal{L}$ of bounded linear suborders of $T$, and hence of $P$ with $\sup \{d(L) \mid L \in \mathcal{L}\} = d(T)$.
    Therefore, $ d(P) \ge d(T) $.
\end{proof}

Definition \ref{linear density} and Proposition \ref{bounded doesnt matter} also motivate us to define to associate a Hausdorff rank to general posets in terms of their density.

\begin{defn}
\label{rank of poset}
    If $ (P, <) $ is a poset with $ d(P) = \infty $ then define the \emph{Hausdorff rank} of $ P $ as $ \HR(P) \defeq \infty $. Otherwise, set $ \HR(P) \defeq \alpha $ where $ \alpha $ is the unique ordinal such that for some $ p < \omega $, $ d(P) = \omega \cdot \alpha + p $.
\end{defn}

Now we state some results on the density and the Hausdorff rank of a bounded discrete $n$-hammock, the proofs of which are easy generalizations of the corollary and the lemma on page 37 of \cite{Schroer98hammocksfor}.

\begin{thm}
\label{Schroer formula}
   If $H$ is a bounded discrete $n$-hammock, then $d(H) = d(\pi_1(H) \dot + \pi_2(H) \dot + \hdots \dot + \pi_n(H))$.  
\end{thm}
\begin{cor}
\label{rank of bdd hammock}
    If $H$ is a bounded discrete $n$-hammock, then $ \HR(H) = \max_{i \in [n]} \HR(\pi_i(H)) $.
\end{cor}

\subsection{Corners}
\label{subsection: corners}
\label{corners}
The hammock posets for non-domestic string algebras will contain MSBs that are finitely presented but not necessarily bounded, so Theorem \ref{Schroer formula} cannot directly be applied to them. 
However they will have periodic \emph{corners}, allowing us to devise a method to calculate their densities. 
In this section we define corners (Definition \ref{corner}), calculate their densities (Proposition \ref{Our formula}), and then calculate the density of an unbounded hammock using a choice of bounded subhammocks and corners (Corollary \ref{about density computation}). As a consequence, we obtain a strict upper bound in the density of a finitely presented hammock (Theorem \ref{fpdensity}).

Unbounded discrete finitely presented linear orders have been shown to have periodic prefixes and suffixes.
\begin{prop}
\cite[Proposition~11.11]{SKSK}
\label{period}
    Suppose $L(\neq\mathbf 0)\in\dLOfpb{}{}$.
    \begin{itemize}
        \item If $L\in\dLOfpb{0}{1}$ then there exist $L_-\in \dLOfpb{1}{1}$ and $L_0\in\dLOfpb{1}{1}\cup\{\mathbf{0}\}$ such that $L \cong L_- \cdot \omega^* + L_0$.
        \item If $L\in\dLOfpb{1}{0}$ then there exist $L_+ \in \dLOfpb{1}{1}$ and $L_0\in\dLOfpb{1}{1}\cup\{\mathbf{0}\}$ such that $L \cong L_0 + L_+ \cdot \omega$.
        \item If $L \in \dLOfpb{0}{0}$ then there exist $L_-,L_+ \in \dLOfpb{1}{1}$ and $L_0\in\dLOfpb{1}{1}\cup\{\mathbf{0}\}$ such that $L \cong L_- \cdot \omega^* + L_0 + L_+ \cdot \omega$.
    \end{itemize}
\end{prop}

The appropriate generalizations of prefixes and suffixes to $ n $-hammocks are bottom and top corners respectively.

\begin{defn}\label{corner}
Let $H$ be an $n$-hammock. A \emph{bottom} (resp. \emph{top}) \emph{corner} $\corner$ of $H$ is a sub-$n$-box of $H$ such that $ \corner $ is an $ n $-clat and
for every $i \in [n]$, $ \pi_i(\corner) $ is a prefix of $ \pi_i(H) $ isomorphic to $ L'_i \cdot \omega^*$ (resp. $ L'_i \cdot \omega $) for some linear order $ L'_i $.
\end{defn}

A bottom or top corner may not exist for a hammock. A necessary condition for a hammock $H$ to have a bottom corner is that all its projections are unbounded below. Interestingly, this is also a sufficient condition for the case of finitely presented hammocks, giving us the promised extension of Proposition \ref{period}.

\begin{cor}
If $H$ is a finitely presented hammock all of whose projections are unbounded below, then $H$ has a bottom corner.
\end{cor} 

\begin{proof}
Let $\lan S_j \mid j \in [k] \ran$ be a sub-$ n $-box-decomposition of $H$. 
For each $ i \in [n] $ let $L_i := \pi_i(S_1)$, so that $L_i$ is a convex subset of the finitely presented linear order $\pi_i(H)$, and thus is finitely presented.
Use Proposition \ref{period} to obtain, for every $ i \in [n] $ a prefix $ L'_i \cdot \omega^* \subseteq L_i $.
Then $ \corner :=\prod_{i \in [n]} (L'_i\cdot\omega^*)$ is (isomorphic to) a bottom corner of $S_1$, and hence of the hammock $H$. 
\end{proof}

We show in Proposition \ref{corner equimorphism} below that a bottom corner of the hammock $H$ is uniquely defined up to equimorphism, i.e., bi-embeddability--this result follows from the $n=1$ version below.
\begin{prop}
\label{equimorphism}
Suppose $L,L_-$ are non-empty linear orders such that $L_-$ is isomorphic to a prefix of $L\cdot\omega^*$. Then $L_-$ has a prefix isomorphic to $L\cdot\omega^*$.
\end{prop}
\begin{proof}
Fix suffixes $ \lan L_k \mid k < \omega \ran $ of $L\cdot \omega^* $ satisfying $ L_k \cong L\cdot k $. Then $ \bigcup_{k' < \omega} L_{k'} = L\cdot\omega^* \cong (L\cdot \omega^*)\setminus L_k $ for each $k\in\omega$. 
Let $ x \in L_- $ and $ k < \omega $ be such that $ x \in L_k $. Then $ L \cdot\omega^* \cong (L\cdot \omega^*)\setminus L_k \subseteq x^\da \subseteq L_- $. 
\end{proof}

\begin{prop}
\label{corner equimorphism}
Let $ \corner^0_- $ and $ \corner^1_- $ be bottom corners of an $ n $-hammock $ H$. Then $ \corner^0_- $ embeds into $ \corner^1_- $.
\end{prop}
\begin{proof}
For each $ i \in [n] $ use the definition of a bottom corner and Proposition \ref{equimorphism} to obtain a prefix $ L'_i $ of $ \pi_i(\corner^1_-) $ order isomorphic to $ \pi_i(\corner^0_-) $. Then $  \corner^0_- \cong \prod_{i \in [n]} L'_i \subseteq \corner^1_- $ as required.
\end{proof}

Since a box in a hammock is a hammock itself, thanks to Proposition \ref{boxes are hammocks}, we can speak of a corner without mentioning the ambient hammock. Moreover, being clats, corners are the simplest nonempty unbounded hammocks.

The next result describes how to compute the density of a corner.
\begin{prop}
\label{Our formula}
Let  $\corner$ be a top $n$-corner with $\pi_i(\corner) \cong L_i\cdot \omega$ for every $i \in [n]$. Then $d(\corner)$ is the limit ordinal $\omega\cdot (\HR(L_1 \dot+ L_2 \dot+ \hdots \dot+ L_n) + 1) $. Moreover, there exists a linear suborder $ L \subseteq \corner $ with $ d(L) = d(\corner)$. As a consequence, $ \HR(\corner) = \max\{\HR(\pi_i(F)) \mid i \in [n]\} $.
\end{prop}
\begin{proof}
    Let $\lan E_k \mid k < \omega \ran$ be a sequence of sub-$n$-boxes of $\corner$ such that for every $k < \omega$ and every $i \in [n]$, $\pi_i(E_k) \cong L_i \cdot k$ is a prefix of $\pi_i(\corner)$.
    Clearly $\bigcup_{k < \omega} E_k = \corner$. Let $h:=\HR(L_1 \dot+ L_2 \dot+ \hdots \dot+ L_n)$. 
    
    Let $L'$ be a bounded linear suborder of $\corner$. Choose $k < \omega$ such that $L' \subseteq E_k$. Then $d(L') \le d(E_k)$. But Proposition \ref{Schroer formula} gives $d(E_k) = d(L'_0 \cdot k \dot+ L'_1\cdot k \dot+ \hdots \dot+ L'_{n-1}\cdot k) < \omega\cdot (h+1)$, and hence $ d(\corner) \le \omega \cdot (h + 1)$.
    
    For the other direction, for each $m \in \N^+$, let $L'_m$ be a bounded linear suborder of $\corner$ such that $L'_m \cong (L_1 \dot+ L_2 \dot+ \hdots \dot+ L_n) \cdot m$ and $ L'_m \subseteq L'_{m+1} $.
    Then for every $ m \in \N^+ $, $\HR(L'_m) = \HR(L'_1) = h$.
    Furthermore, the lemma on Page 31 in \cite{Schroer98hammocksfor} yields $n_m \ge m$ such that $L_m/\sim_h = \mbf n_m$.
    Thus for each $m\in\N^+$, we have $d(L_m) \ge \omega\cdot h + m$. 
    Letting $ L \defeq \bigcup_{m \in \N} L'_m $ we see that $ d(L) \ge \omega \cdot (h + 1) $.
    Therefore $ d(L) = d(\corner) = \omega \cdot (h + 1) $. 

    Hence by Definition \ref{rank of poset}, $ \HR(F) = h + 1 = \HR(L_1 \dot+ L_2 \dot+ \dots \dot+ L_n) + 1 = \max \{\HR(L_i) : i \in [n]\} + 1  = \max\{\HR(L_i \cdot \omega) : i \in [n] \} = \max \{\HR(\pi_i(F)) : i \in [n]\} $, where the fourth equality follows  by taking $ n = 1 $ in the preceding three paragraphs.
\end{proof}

\begin{cor}\label{prop: omega times linear order}
    For a linear order $L$, we have $d (L \cdot \omega) = \omega\cdot (\HR(L)+1)$.
\end{cor}

For the rest of this section, let $ H $ denote a discrete $ n $-hammock with all projections unbounded.
\begin{defn}
    A \emph{corner decomposition} of the $n$-hammock $ H $ is a tuple $ (\corner_-, H_0, \corner_+) $, where
    \begin{itemize}
        \item $\corner_-, \corner_+$ are bottom and top corners of $H$;
        \item $ H_0 $ is a bounded subhammock of $ H $;
        \item for each $ i \in [n] $
        \begin{itemize} 
            \item $ \pi_i(\corner_-), \pi_i(H_0) $ are bounded above,
            \item $ \pi_i(H_0), \pi_i(\corner_+) $ are bounded below,
            \item $ \pi_i(\corner_-) \dot+ \pi_i(H_0) \dot+ \pi_i(\corner_+) = \pi_i(H) $.
        \end{itemize}
    \end{itemize}
\end{defn}
We will show in Corollary \ref{about density computation} that $d(H)$ cab be computed using any corner decomposition. The next result is the $n=1$ case of that result.
\begin{prop}
\label{linear corners}
Let $ L_-, L_0, L_+ $ be linear orders such that $ L_-, L_0 $ have upper bounds and $ L_0, L_+ $ have lower bounds. Let $ L \defeq L_- \dot+ L_0 \dot+ L_+ $. If $ d(L_-), d(L_+) $ are limit ordinals, then $ d(L) = \max\{d(L_-), d(L_0), d(L_+)\}.$
\end{prop}
\begin{proof}
    By Remark \ref{density is increasing} it follows that $ d(L) \ge \max \{d(L_-), d(L_0), d(L_+)\} $.
    
    For the other direction let $ d(L) > \omega \cdot \alpha + p $. We show that $d(L')>\omega\cdot\alpha+p$ for some $L'\in\{L_-,L_0,L_+\}$.
    
    Use Proposition \ref{witness for density} to choose witnesses $ x_0 <^\parens \alpha x_1 <^\parens \alpha x_2 <^\parens \alpha \cdots <^\parens \alpha x_{p + 1} $ in $ L $.
    If $ L' \in \{ L_-, L_0, L_+ \} $ is such that 
      $ x_j \in L' $ for every $ 0 \le j \le p + 1 $, 
      then $ d(L') > \omega \cdot \alpha + p $ by Proposition \ref{witness for density} and we are done.
    
    Suppose $ x_0, x_1 $ are both in $ L_- $. Since $ d(L_-)$ is a limit ordinal, $ L_-/\sim_{\HR(L_-)} = \mbf 1 $ and hence $ x_0 \sim_{\HR(L_-)} x_1 $. Since $ x_0 <^\parens \alpha x_1$, we get $ \HR(L_-)>\alpha$.
    Again since $d(L_-)$ is a limit ordinal, $d(L_-)> \omega \cdot \alpha + p$ as required.
    
    Therefore we may assume that $ \max L_- <^\parens \alpha x_1 $. Similarly we can assume that $ x_p <^\parens \alpha \min L_+ $.
    But then $\max L_-=\min L_0 <^\parens \alpha x_1 <^\parens \alpha \cdots <^\parens \alpha x_p <^\parens \alpha \max L_0=\min L_+ $ are elements of $ L_0 $
      witnessing $ d(L_0) > \omega \cdot \alpha + p $.
\end{proof}

Now we are ready to show that the choice of the corner decomposition of $H$, if exists, does not affect the computation of $d(H)$.
\begin{lem}
\label{invariant}
Suppose $ (\corner^0_-, H^0_0, \corner^0_+) $ and $ (\corner^1_-, H^1_0, \corner^1_+) $ are two corner decompositions of a discrete $n$-hammock $ H $ all of whose projections are unbounded. Then \[ \max \{d(\corner^0_-), d(H^0_0) , d(\corner^0_+)\} = \max \{d(\corner^1_-), d(H^1_0), d(\corner^1_+)\}. \]
\end{lem}
\begin{proof}
First recall from Theorem \ref{rank infinity iff not scattered} and Proposition \ref{densityposetunbdd} that $H$ is not scattered if and only if $d(H)=\infty$. In case these equivalent conditions hold then Proposition \ref{prop: SCattered sum of linear orders} and Theorem \ref{Schroer formula} together ensure that both left and right hand sides equal $\infty$. So we may assume that $H$ is scattered, and hence $d(H)\neq\infty$.

By Corollary \ref{corner equimorphism} $ \corner^0_- $ embeds into $ \corner^1_- $ and vice versa, so $ d(\corner^0_-) = d(\corner^1_-) $.
Similarly $ d(\corner^0_+) = d(\corner^1_+) $. Without loss of generality, assume that $ d(F^1_-) \ge d(F^1_+) $.
If $ d(H^0_0) = d(H^1_0) $ then we are done, so assume without loss of generality that $ d(H^0_0) < d(H^1_0) $. In order to complete the proof, it is enough to show that $ d(H^1_0) < \max\{d(\corner^1_-),d(\corner^1_+)\}$.
          
Suppose not. Then $ d(H^1_0)\ge d(F^1_-)\geq d(F^1_+)$. Proposition \ref{Our formula} gives that both $d(F^1_-)$ and $d(F^0_-)$ are limit ordinals. Let $ d(F^1_-) = \omega \cdot \alpha$, where $\alpha:=\HR(F^1_-)$ is an ordinal.
        
For each $ i \in [n] $, since $\pi_i(H^1_0)$ is bounded, we can choose a bounded proper suffix $ L^i_- $ of $ \pi_i(\corner^0_-) $ and a bounded proper prefix $ L^i_+ $ of $ \pi_i(\corner^0_+) $ such that \begin{equation}
    \label{stretch} \pi_i(H^1_0) \subseteq L^i_- \dot+ \pi_i(H^0_0) \dot+ L^i_+ .
\end{equation}
        
There are two cases to consider.
\begin{itemize}
\item \textsf{Case 1: $ d(H^0_0) < d(F^1_-) $.} Proposition \ref{prop: density of sum of orders} gives that for each $i\in[n]$, we have $\HR(L^i_- \dot+ \pi_i(H^0_0) \dot+ L^i_+) = \max\{\HR(L^i_-), \HR(\pi_i(H^0_0)), \HR(L^i_+))\}$. Since $d(H_0^1)\geq d(F^1_-)$ use Corollary \ref{rank of bdd hammock} to choose $ j \in [n] $ such that $\HR(\pi_j(H^1_0)) \ge \HR(F^1_-)=\alpha$. Hence, in view of Equation \eqref{stretch}, it will be enough to show that $\max\{\HR(L^j_-), \HR(\pi_j(H^0_0)), \HR(L^j_+))\}<\alpha$.

Since $ d(F^1_-) $ is a limit ordinal, and $d(H_0^0)<d(F^1_-)$, we get $\HR(H_0^0)<\HR(F^1_-)=\alpha$. Furthermore, Corollary \ref{rank of bdd hammock} gives that
$\HR(H^0_0) = \max_i \HR(\pi_i(H^0_0)) \ge \HR(\pi_j(H^0_0))$. Thus $\alpha>\HR(\pi_j(H^0_0))$.

Since $ L^i_- $ is bounded, $ d(L^i_-) $ is a successor ordinal by Proposition \ref{bddlindensitysuccessor}. On the other hand, since $d(F^0_-)$ and $d(F^1_-)$ are limit ordinals, $ d(L^i_-) < d(F^i_-) \le d(F^1_-) = \omega \cdot \alpha $. Hence $ \HR(L^i_-) < \alpha $. Similarly we can argue that $ \HR(L^i_+) < \alpha $ to complete the proof of this case.

\item \textsf{Case 2: $ d(H^0_0) \ge d(F^1_-) $.}
Let $L_-:=L^1_- \dot+ L^2_- \dot+ \dots \dot+ L^n_-$ and $L_+:=L^1_+ \dot+ L^2_+ \dot+ \dots \dot+ L^n_+$. Then $L_-$ and $L_+$ are bounded linear suborders of $F_-^0$ and $F_+^0$ respectively. Since $d(L_-)$ and $d(L_+)$ are successor ordinals by Proposition \ref{bddlindensitysuccessor} while $d(F_-^0)$ and $d(F_+^0)$ are limit ordinals by Proposition \ref{Our formula}, we get $\HR(L_-^0)<\HR(F_-^0)=\HR(F_-^1)$ and $\HR(L_+^0)<\HR(F_+^0)=\HR(F_+^1)\leq\HR(F_-^1)$. So we have
\begin{align*}
d(H^1_0) &= d(\pi_1(H^1_0) \dot+ \pi_2(H^1_0) \dot+ \dots \dot+ \pi_n(H^1_0)) && \text{(Theorem \ref{Schroer formula})}\\
&\le d(L^1_- \dot+ \pi_1(H^0_0) \dot+ L^1_+ \dot+ \dots \dot+ L^n_- \dot+ \pi_n(H^0_0) \dot+ L^n_+) &&\text{(Equation \eqref{stretch})}\\
&= d(\pi_1(H^0_0) \dot+ \pi_2(H^0_0) \dot+ \dots \dot+ \pi_n(H^0_0) \dot+ L_-\dot+ L_+) &&\text{(Corollary \ref{order irrelevant})}\\
&=  d(\pi_1(H^0_0) \dot+ \pi_2(H^0_0) \dot+ \dots \dot+ \pi_n(H^0_0))&&(\max\{\HR(L_-),\HR(L_+)\}<\HR(F_-^1)\leq\HR(H^0_0))\\
&= d(H^0_0) && \text{(Theorem \ref{Schroer formula})}
\end{align*}
contrary to our assumption that $ d(H^1_0) > d(H^0_0) $.
\end{itemize}
\end{proof}
\begin{cor}
\label{about density computation}
Suppose $ (F_-, H_0, F_+) $ is a corner decomposition of a discrete $n$-hammock $ H $ all of whose projections are unbounded. Then $ d(H) = \max \{d(\corner_-), d(H_0), d(\corner_+)\} $.
\end{cor}

\begin{proof}
Since $ \corner_-, H_0, \corner_+ $ are all subsets of $ H $, it is clear that $ d(H) \ge \max \{d(\corner_-), d(H_0), d(\corner_+)\} $.

For the other direction, let $ L \subseteq H $ be a bounded linear order.  Thanks to Lemma \ref{invariant} we can choose $H_0$ large enough so that $ L \subseteq H_0$. Then $ d(L) \le d(H_0) \le \max \{d(\corner_-), d(H_0), d(\corner_+)\} $. Since $L$ is arbitrary, Proposition \ref{bounded doesnt matter} gives $ d(H)\leq\max \{d(\corner_-), d(H_0), d(\corner_+)\} $ as required.
\end{proof}

\begin{cor}
\label{finrankdensity}
Suppose $H$ is a discrete $n$-hammock $ H $ all of whose projections are unbounded. If $ \HR(H) < \omega $ then $ d(H) < \omega^2 $. 
\end{cor}
\begin{proof}
Suppose $ (F_-, H_0, F_+) $ is a corner decomposition of $H$. Since $\HR(H)<\omega$, Proposition \ref{Our formula} ensures that $d(F_-)\leq\omega\cdot\HR(F_-)\leq \omega\cdot\HR(H)<\omega^2$. Similarly we can argue that $d(F_+)<\omega^2$. Finally since $H_0$ is bounded and discrete, Theorem \ref{Schroer formula} gives $d(H_0)<\omega^2$. Thus, thanks to Corollary \ref{about density computation}, we get $d(H)=\max \{d(\corner_-), d(H_0), d(\corner_+)\}<\omega^2$.
\end{proof}

The next result documents a special case of the above corollary that will be a useful tool in the next section.

\begin{thm}\label{fpdensity}
Suppose $H$ is a finitely presented discrete $n$-hammock. Furthermore, assume that either for all $i\in[n]$, $\pi_i(H)$ is bounded below (resp. above) or for all $i\in[n]$, $\pi_i(H)$ is unbounded below (resp. above). Then $d(H)<\omega^2$.
\end{thm}
\begin{proof}
Remark \ref{HRfindesc} gives that $\HR(\pi_i(H))<\omega$ for each $i\in[n]$. If the linear orders $\pi_i(H)$ are bounded then Theorem \ref{Schroer formula} gives the conclusion. If the linear orders $\pi_i(H)$ are all unbounded then Proposition \ref{period} ensures that $H$ has both top and bottom corners, and hence Corollary \ref{finrankdensity} gives the conclusion. Finally, if all $\pi_i(H)$ are bounded below but unbounded above, or if they are bounded above but unbounded below then Proposition \ref{period} ensures the existence of only one corner. In that case, an appropriate modification of Corollary \ref{finrankdensity} gives the conclusion.
\end{proof}

\section{The proof of Theorem \ref{oldconj}}
\label{section: Proofs}
We recall the definitions of special biserial algebras and string algebras in \S~\ref{sec: Preliminaries} along with associated combinatorial notations and terminologies. In the same subsection we also recall the description of finite-dimensional indecomposable modules for string algebras and (a basis of) the space of morphisms between them given by Gel'fand and Ponomarev\cite{G&P}, and Crawley-Boevey \cite{crawley1989maps} and Krause\cite{krause1991maps} respectively. In \S~\ref{sec: hammock lin order} and \S~\ref{excpoints}, we will study concrete hammock linear orders and hammock posets for string algebras respectively. In these two subsections, we recall several notations and results about hammocks from \cite{SKSK} and \cite{Schroer98hammocksfor}, and use them to obtain sufficient conditions for hammock linear orders to be finite (Proposition \ref{prop: Finite Hammock}) or scattered (Theorem \ref{prop: Scattered hammock}). The results in \S~\ref{order theory} imply that the hammock posets are finitely presented discrete $2$-hammocks and the density of any of its scattered boxes is bounded above by $\omega^2$ (Theorem \ref{hammockordermain}).  After proving yet one more finiteness result (Lemma \ref{excptfin}), we complete the proof Theorem \ref{oldconj} for the case when $\Lambda$ is a string algebra in \S~\ref{conjproofstringalg}. We achieve this by establishing connections between the stable rank and the maximum density of a maximal scattered box in a hammock poset (Equation \eqref{stablerank6}). In \S~\ref{biserial}, we extend Theorem \ref{oldconj} to the class of special biserial algebras.

\subsection{Preliminaries of string algebras and special biserial algebras}\label{sec: Preliminaries} We begin by recalling the definition of a string algebra as well as a special biserial algebra. Then we introduce the combinatorial glossary of the terms associated with a string algebra, including the definitions of strings and bands, that will be used throughout the rest of the paper. Our presentation mostly follows \cite{GKS20}. At the end of the section we will recall Gel'fand and Ponomarev's \cite{G&P} complete classification of finite-dimensional indecomposable modules for a string algebra, and the description of a basis for all maps between indecomposable modules in terms of the so-called \emph{graph maps} introduced by Crawley-Boevey \cite{crawley1989maps} and Krause \cite{krause1991maps}.

Fix an algebraically closed field $\K$. A \emph{string algebra} (resp. \emph{special biserial algebra}) is a $\mathcal K$-algebra $\Lambda:=\K\Q/\langle\rho\rangle$ presented as a certain quotient of the path algebra of a finite quiver $\Q=(Q_0,Q_1,s,t)$, where $Q_0$ is a finite set of vertices, $Q_1$ is a finite set of arrows, and $s,t:Q_1\to Q_0$ are the source and target functions respectively, by the admissible ideal generated by the set $\rho$ of monomial relations (resp. the set $\rho$ of relations that are linear combinations of paths of length at least $2$ and that share source and target vertices) satisfying the following conditions. We require that the indegree and outdegree of any vertex in the quiver $\Q$ is at most two. The Roman letters $v,w$ (possibly with subscripts) will be used for denoting the elements of $Q_0$ and $a,b,c,d,\cdots $ (possibly with subscripts) will be used to denote the elements of $Q_1$. For every $v \in Q_0$ we define a zero-length path, namely $1_{(v,1)}$. A path (of length $n\geq1$) in the quiver is defined as a finite sequence of arrows $a_na_{n-1}\cdots a_1$, where $s(a_{i+1}) = t(a_i)$ for every $i \in[n-1]$. The set $\rho$ satisfies the following properties.
\begin{itemize}
    \item For any $a \in Q_1$ there exists at most one $b \in Q_1$ such that $s(b) = t(a)$ and $ba \notin \langle\rho\rangle$.
    \item For any $a \in Q_1$ there exists at most one $b \in Q_1$ such that $t(b) = s(a)$ and $ab \notin \langle\rho\rangle$.
\end{itemize} 
For simplicity we assume that no path in $\rho $ has a proper subpath in $\rho$. For technical reasons, we also choose and fix a bisection, i.e., a pair of maps $\sigma,\varepsilon:Q_1\to\{1,-1\}$ satisfying the following conditions.
\begin{itemize}
    \item If $a_1 \neq a_2$ are arrows with $s(a_1) = s(a_2)$, then $\sigma(a_1) = - \sigma(a_2). $
    \item If $a_1 \neq a_2$ are arrows with $t(a_1) = t(a_2)$, then $\varepsilon (a_1 ) = -\varepsilon (a_2).$
    \item If $a_1, a_2 \in Q_1$ such that $t(a_1) = s(a_2)$ and $a_2a_1 \notin\langle\rho\rangle$, then $\sigma (a_2) = -\varepsilon (a_1)$.
\end{itemize}

\begin{exmp}\label{Ex: Standard example}
    Here are some standard examples of string algebras.
\begin{enumerate}
        \item For $n,m \geq 2$, the family $GP_{n,m}$ of algebras that was first studied by Gel'fand and Ponomarev is presented using the quiver
\begin{tikzcd}
v \arrow["b"', loop, distance=2em, in=215, out=145] \arrow["a"', loop, distance=2em, in=35, out=325]
\end{tikzcd} with relations $\rho = \{ab,ba, a^n, b^m\}$. The algebra $GP_{n,m}$ is non-domestic if and only if $n+m \geq 5$.

\item The algebra $\Lambda_2 $ presented using the quiver
\begin{tikzcd}[ampersand replacement=\&]
	v_1 \& v_2 \& v_3 \& v_4
	\arrow["c", from=1-2, to=1-3]
	\arrow["d", curve={height=12pt}, from=1-3, to=1-4]
	\arrow["e", curve={height=-12pt}, from=1-3, to=1-4]
	\arrow["a", curve={height=12pt}, from=1-1, to=1-2]
	\arrow["b", curve={height=-12pt}, from=1-1, to=1-2]
\end{tikzcd} with relations $\rho = \{cb, dc\}$ is domestic.
\end{enumerate}
\end{exmp}

Let us denote by $Q_1^-$ the set of the uppercase Roman letters $A_i$ corresponding to each $a_i \in Q_1$. Set $s(A_i) := t(a_i)$, $t(A_i) := s(a_i)$, $\sigma(A_i) := \varepsilon (a_i)$ and $\varepsilon(A_i) = \sigma (a_i)$. We use lowercase Greek letters $\alpha, \beta, \cdots$ to denote the elements of $Q_1 \cup Q_1^-$; these element will be referred to as \emph{syllables}). The inverse of the zero-length path $1_{(v,1)}$ will be denoted using the notation $1_{(v,-1)}$ and the inverse of the path $\xx = a_n\cdots a_2a_1$ with $n\geq 1$ will be denoted  $\xx^{-1}: = A_1A_2\cdots A_n$. Set $\rho^{-1} := \{\xx^{-1}: \xx \in \rho\} $. 

Now we define two combinatorial entities, namely \emph{strings} and \emph{bands}, introduced by Butler and Ringel that play a significant role in the complete classification of indecomposable objects in $\Lambda\dmod$ (see \S~\ref{conjproofstringalg}).

\begin{defn}
    A \emph{word} $\uu$ is either a zero-length path, its inverse or a finite sequence $\alpha_n\cdots \alpha_2\alpha_1$ of syllables such that $s(\alpha_{i+1}) = t(\alpha_i)$ for each $i \in [n-1]$. A \emph{string} is a word $\uu$ that has no subpath in $\rho \cup \rho^{-1}$ or two consecutive syllables that are inverses of each other. A \emph{band} is string $ \bb = \alpha_n\cdots\alpha_2\alpha_1$ satisfying $\alpha_n \in Q_1$, $\alpha_1 \in Q_1^-$, $\bb^2$ is a string, and that $\bb\neq\bb_1^n$ for any string $\bb_1$ and $n\geq 2$.
\end{defn}
We use the notation $\St$ to denote the set of all strings for $\Lambda$ and $\mathsf{Ba}(\Lambda)$ to denote the set of all bands up to cyclic permutations of syllables. Let $\Ba$ be a fixed set of representatives in $\mathsf{Ba}(\Lambda)$. Call a cyclic permutation of an element in $\Ba$ a cycle. Denote the set of all cycles for $\Lambda$ by $\Cyc$.

\begin{exmp}
For the algebra $GP_{n,m}$ from Example \ref{Ex: Standard example} $aB , aBa, a, 1_{(v,1)}, 1_{(v,-1)} $ are some examples of strings, and $aB, bA$ are some examples of bands. For the algebra $\Lambda_2$ from the same example $eca, caB, eD, dE$ are some examples of strings, and $aB, eD$ are some examples of bands.
\end{exmp}

Words, strings and bands are denoted using lowercase Fraktur letters. We denote the length of a word $\xx$ by $|\xx|$. We extend the functions $s,t,\sigma$ and $\varepsilon$ to all strings as follows. For a string $\xx = \alpha_n\cdots \alpha_2 \alpha_1$ define $s(\xx):= s(\alpha_1)$, $\sigma(\xx) := \sigma(\alpha_1)$, $t(\xx) := t(\alpha_n)$ and $\varepsilon (\xx) := \varepsilon (\alpha_n)$. For any $v \in Q_0$ and $i \in \{ 1,-1 \}$ we define $s(1_{(v,i)})=t(1_{(v,i)}) := v$ and $\sigma(1_{(v,i)})=-\varepsilon (1_{(v,i)}) := -i$. The composition $1_{(v,i)}\uu$ is defined if $t(\uu) = v$ and $\varepsilon(\uu)=i$, and under these conditions we say that $1_{(v,i)}\uu = \uu$. Similarly the composition $\uu1_{(v,i)}$ is defined if $s(\uu) = v$ and $i = -\sigma(\uu)$, and under these conditions we say that $1_{(v,i)}\uu = \uu$.

For strings $\xx$ and $\yy$, we say that $\xx$ is a \emph{left substring} (resp. \emph{proper left substring}) of $\yy$, denoted $\xx \sqsubseteq_l \yy$ (resp. $\xx \sqsubset_l \yy$) if $\yy=\uu\xx$ for some (resp. positive length) string $\uu$. Dually say that $\xx$ is a \emph{right substring} (resp. \emph{proper right substring}) of $\yy$, denoted $\xx \sqsubseteq_r \yy$ (resp. $\xx \sqsubset_r \yy$) if $\yy=\xx\uu$ for some (resp. positive length) string $\uu$.

Suppose $\xx\in\St$ and $|\xx|>0$. We define $\theta_l(\xx), \theta_r(\xx)\in\{1,-1\}$ as follows: $\theta_l(\xx)=1$ if and only if the first syllable of $\xx$ is inverse, and $\theta_r(\xx)=1$ if and only if the last syllable of $\xx$ is inverse. Further, to identify if $\xx$ has any sign changes we define
$$\delta(\xx):=
\begin{cases}
    1&\text{ if all syllables of }\xx\text{ are inverse};\\
    -1&\text{ if all syllables of }\xx\text{ are direct};\\
    0&\text{ otherwise}.
\end{cases}$$

Using the language of strings and bands, now we describe the module category for $\Lambda$. Any indecomposable module for a string algebra is isomorphic to either a string module or a band module--this result is essentially due to Gel'fand and Ponomarev\cite{G&P}. We recall a few details and set up some notation, and refer the reader to \cite[Chapter~2]{LakThe16} for a detailed exposition on the construction of string and band modules, and maps between them.

The string module associated with $\xx\in\St$ is denoted $\mathrm{M}(\xx)$ and this association satisfies $M(\xx)\cong M(\yy)$ if and only if $\xx=\yy$ or $\xx^{-1}=\yy$. There is a family of band modules $\{\mathrm{B}(\bb , n, \lambda)\mid n \in \N^+,\lambda \in \K^*\}$ associated with $\bb\in\QBa_0$ satisfying $\mathrm{B}(\bb_1, n_1, \lambda_1)\cong\mathrm{B}(\bb_2, n_2, \lambda_2)$ if and only if $n_1=n_2$, $\lambda_1=\lambda_2$, and $\bb_1$ is a cyclic permutation of either $\bb_2$ or $\bb_2^{-1}$.

Crawley-Boevey \cite{crawley1989maps} and Krause \cite{krause1991maps} described a basis of the $\K$-vector space $\mathrm{Hom}_\Lambda (M,N) $ for any finite-dimensional indecomposable modules $M$ and $N$ of a string algebra $\Lambda$ in terms of the so-called \emph{graph maps}. To describe graph maps in detail, we need the concept of image and factor substrings.
\begin{defn}
For a string $\uu = \alpha_n\cdots \alpha_2\alpha_1$, a substring $\alpha_j \cdots \alpha_i$ is said to be an \emph{image substring} of $\uu$ if it satisfies the following conditions:
\begin{enumerate}
    \item if $j \neq n$ then $\alpha_{j+1} \in Q_1^-$;
    \item if $i \neq 1$ then $\alpha_{i-1} \in Q_1$.
\end{enumerate}
Dually call a substring $\alpha_j\cdots\alpha_i$ of $\uu$ to be a \emph{factor substring} of $\uu$ if it follows the following conditions:  \begin{enumerate}
    \item if $j \neq n$ then $\alpha_{j+1} \in Q_1$;
    \item if $i \neq 1$ then $\alpha_{i-1} \in Q_1^-$.
\end{enumerate}
\end{defn}
If $\vv$ is an image substring of $\uu_1$ then there is a canonical monomorphism $\mathrm{M}(\vv)\to \mathrm{M}(\uu_1)$. Dually, if $\vv$ is a factor substring of $\uu_2 $ then there is a canonical epimorphism $\mathrm{M}(\uu_2 ) \to \mathrm{M}(\vv)$. A \emph{graph map} $f: \mathrm{M}(\uu_2 ) \to \mathrm{M}(\uu_1 )$ \emph{between two string modules} is the composition of such a canonical epimorphism followed by such a canonical monomorphism, and the string $\vv$ is called the \emph{string associated with} $f$.

If $\vv$ is an image substring of $\prescript{\infty}{}{\bb}^\infty$ for a band $\bb$, then for every $n \in \N^+$ and $\lambda \in \K^*$ there is a family of canonical maps  $\mathrm{M}(\vv)\to \mathrm{B}(\bb , n, \lambda)$ indexed by a basis of $\mathrm{Hom}_{\K}(\K,\K^n)$. The composition of such map with a canonical epimorphic graph map $\mathrm{M}(\uu) \to \mathrm{M}(\vv) $ is called a \emph{graph map from a string module to a band module with associated string} $\vv$. Dually, we can also describe graph maps from a band module to a string module using factor substrings of $\prescript{\infty}{}{\bb}^\infty$. Furthermore, if $\vv$ is simultaneously an image substring of $\prescript{\infty}{}{\bb}^\infty$ and a factor substring of $\prescript{\infty}{}{\bb'}^\infty$ then the composition $\mathrm{B}(\bb' , n' , \lambda') \to \mathrm{M}(\vv)\to \mathrm{B}(\bb , n , \lambda)$ of the canonical morphisms is called a \emph{graph map between band modules with associated string} $\vv$.

\subsection{Hammock linear orders for string algebras}\label{sec: hammock lin order}
Some partially ordered sets called hammocks were introduced by Brenner \cite{brenner1986combinatorial} to study factorizations of maps between finitely generated modules. The simplest version of a hammock introduced by Schr\"oer \cite[\S~3]{schroer2000infinite} in the context of string algebras is a linear order. We recall the definitions of these left and right hammock linear orders associated with a fixed string $\xx_0$, and introduce several related notations and terminologies from \cite{SKSK}. The main result of that paper (Theorem \ref{SKSKThm}) states that such hammocks are bounded, discrete and finite description linear orders. At the end we provide criteria for hammocks to be finite (Proposition \ref{prop: Finite Hammock}) or scattered (Theorem \ref{prop: Scattered hammock}).
\begin{defn}\label{hammock defn}
The left and right hammock sets of a string $\xx_0$ are defined as
$$H_l(\xx_0):=\{\xx\in\St\mid\xx_0 \sqsubseteq_l \xx\},\ H_r(\xx_0):=\{\xx\in\St\mid\xx_0 \sqsubseteq_r \xx\}.$$
The left hammock $H_l(\xx_0)$ can be equipped with a linear order $<_l$, where for $\xx,\yy\in H_l(\xx_0)$ we have $\xx<_l\yy$ if one of the following holds:
\begin{itemize}
    \item $\yy=\uu\alpha\xx$ for some string $\uu$ and $\alpha\in Q_1^-$;
    \item $\xx=\vv\beta\yy$ for some string $\vv$ and $\beta\in Q_1$;
    \item $\xx=\vv\beta\ww$ and $\yy=\uu\alpha\ww$ for some $\alpha\in Q_1^-$, $\beta\in Q_1$ and strings $\uu,\vv,\ww$.
\end{itemize}
The ordering $<_r$ on $H_r(\xx_0)$ is defined as $\xx<_r\yy$ if and only if $\xx^{-1}<_l\yy^{-1}$ in $(H_l(\xx_0^{-1}),<_l)$.
\end{defn}

\begin{exmp}
For the algebra $\Lambda_2$ from Example \ref{Ex: Standard example}, we have $ eca<_l ca<_l Dca<_l ecaBa<_l DcaBa<_l  Ba  $ in $(H_l(a),<_l)$.
\end{exmp}

For $\xx,\yy\in H_l(\xx_0)$, denote by $\xx\sqcap_l\yy$ the maximal common left substring of $\xx$ and $\yy$. If $\xx=\ww(\xx\sqcap_l\yy)$ with $|\ww|>0$ then define $\theta_l(\xx\mid\yy):=\theta_l(\ww)$ and $\delta_l(\xx\mid\yy):=\delta(\ww)$. Similarly, for $\xx,\yy\in H_r(\xx_0)$ we define $\xx\sqcap_r\yy$ to be the maximal common right substring of $\xx$ and $\yy$. If $\xx=(\xx\sqcap_r\yy)\ww$ with $|\ww|>0$ then define $\theta_r(\xx\mid\yy):=\theta_r(\ww)$ and $\delta_r(\xx\mid\yy):=\delta(\ww)$. 

The hammock order $(H_l(\xx_0),<_l)$ can be expressed as $H_l(\xx_0)=H_l^{-1}(\xx_0)\dot+ H_l^1(\xx_0)$, where for $i \in \{1,-1\}$,
$H_l^i(\xx_0):=\{\xx\in H_l(\xx_0)\mid\text{ either }\xx=\xx_0\text{ or }\theta_l(\xx\mid\xx_0)=i\}$. Similarly $H_r(\xx_0)=H_r^{1}(\xx_0)\dot+ H_r^{-1}(\xx_0)$, where for $i \in \{1,-1\}$,
$H_r^i(\xx_0):=\{\xx\in H_r(\xx_0)\mid\text{ either }\xx=\xx_0\text{ or }\theta_r(\xx\mid\xx_0)=i\}$. As a convention, for any $\xx_0 \in \St$, fix $ H_r^0(\xx_0) = H_l^0(\xx_0)  := \{\xx_0\}$.

Almost all strings in a left hammock and a right hammock have an immediate successor as well as an immediate predecessor. Moreover, the order types of such hammocks have been determined.

\begin{thm}\cite[Theorem~11.9]{SKSK}
\label{SKSKThm}
If $\xx_0 \in \St $ then $(H_l(\xx_0 ),<_l),(H_r(\xx_0 ),<_r)\in\LOfd$, and both are bounded discrete linear orders.
\end{thm}

For any $\xx_0 \in \St$, denote by $\mm_{l,i} (\xx_0 ) $ $\left( \text{resp. } \mm_{r,i} (\xx_0 ) \right)$ the minimal element in $H_l^i(\xx_0 )$ and by $\MM_{l,i} (\xx_0 ) $ $\left( \text{resp. } \MM_{r,i} (\xx_0 ) \right)$ the maximal element in $(H_l^i(\xx_0 ), <_l) $ $\left( \text{resp. } (H_r^i(\xx_0 ), <_l) \right)$. Clearly $ \mm_{l,1}(\xx_0 ) = \mm_{r,-1}(\xx_0 ) =\MM_{l,-1}(\xx_0 ) = \MM_{r,1}(\xx_0 ) = \xx_0 $.

\begin{prop}\cite[\S~2.5]{Schroer98hammocksfor}
For any $\xx_0 \in \St$, $\mm_{l,-1}(\xx_0)$ is the longest string in $H_l(\xx_0 )$ satisfying either $\delta_l(\mm_{l,-1}(\xx_0 )\mid\xx_0 )=-1 $ or $\mm_{l,-1}(\xx_0 )=\xx_0$, whereas $\MM_{l,1}(\xx_0)$ is the longest string satisfying either $\delta_l(\MM_{l,1}(\xx_0)\mid\xx_0)=1$ or $\MM_{l,1}(\xx_0 )=\xx_0 $ .

Dually for any $\xx_0 \in \St$, $\mm_{r,1}(\xx_0 )$ is the longest string in $H_r(\xx_0 )$ satisfying either $\delta_r(\mm_{r,1}(\xx_0 )\mid\xx_0 )=1$ or $\mm_{r,1}(\xx_0 )=\xx_0$, whereas $\MM_{r,-1}(\xx_0 )$ is the longest string satisfying either $\delta_r(\MM_{r,-1}(\xx_0 )\mid\xx_0 )=-1$ or $\MM_{r,-1}(\xx_0)=\xx_0 $.
\end{prop}

Intervals in hammocks enjoy a very special property.

\begin{prop}\cite[Proposition~4.3]{SKSK}\label{unique string with minimal length in an interval}
Given a non-empty interval $I$ in $(H_l(\xx_0),<_l)$ (resp. in $(H_r(\xx_0), <_r)$), there is a unique string $\uu$ in $I$ with minimal length. Moreover, $I\subseteq H_l(\uu)$ (resp. $I\subseteq H_r(\uu)$).
\end{prop}

In this paper, we will follow the approach of \cite{SKSK} (also see \cite[Proposition~3.4.2]{GKS20}) to differentiate between domestic and non-domestic string algebras. From \cite[\S~5]{SKSK} recall the relation $\preceq$ on the set $\Ba$ of bands, where $\bb_1\preceq\bb_2$ if there is a string $\uu$ such that $\bb_2\uu\bb_1$ is a string. Note that this relation is reflexive and transitive. Defining the equivalence relation $\approx:=\succeq\cap\preceq$ and setting $\QBa := \Ba / \approx$, we obtain a finite poset $(\QBa, \preceq)$ \cite[Proposition~5.6]{SKSK}. The preorder $(\QBa_0,\preceq)$ plays a crucial role in the computation of the order type of hammocks, as demonstrated in \cite{schroer2000infinite},\cite{GKS20} and \cite{SKSK}. Note that \cite[Corollary~5.5]{SKSK} and \cite[Theorem~3.1.6]{GKS20} together provide an algorithm to compute the poset $\left( \QBa, \preceq \right)$. Say that $\sB \in \QBa$ is \emph{domestic} if $|\sB| =1$; otherwise, we call it \emph{non-domestic}. 

\begin{rem}{\cite[Proposition~5.3]{SKSK}}
The string algebra $\Lambda$ is domestic if and only if $(\Ba,\preceq)$ is anti-symmetric if and only if $\sB$ is domestic for each $\sB \in \QBa$.
\end{rem}

Given $\sB \in \QBa$, $\bb \in \Cyc$ is termed a $\sB$\emph{-cycle} \cite[\S~5]{SKSK} if there exists $\uu_1\uu_2 \in \St$ and $\bb_1,\bb_2 \in \sB$ such that $\bb_1\uu_1\bb\uu_2\bb_2$ is a string. The set containing all $\sB$-cycles is denoted $\CycB$. Note that $\sB \subseteq \CycB$ for any $\sB \in \QBa$.

We begin our investigation of the order type of hammocks by defining a set of bands \emph{relevant} for a hammock, i.e., the bands \emph{visible} from within it. For $\xx\in \St$ and $i \in \{1,0,-1\}$, set
\begin{align*}
    \overline{\QBa_{l,i}} (\xx) := \{\bb \in \QBa_0\mid\text{there exists }\uu \in \St \text{ such that } \bb\uu\xx \in H_l^i(\xx)\}, &\quad \QBa_{l,i}(\xx) := \{ \sB \in \QBa\mid\sB \cap \overline{\QBa_{l,i}}(\xx) \neq \emptyset \},\\
    \overline{\QBa_{r,i}} (\xx) := \{\bb \in \QBa_0\mid\text{there exists }\uu \in \St \text{ such that } \xx\uu\bb \in H_r^i(\xx)\}, &\quad \QBa_{r,i}(\xx) := \{ \sB \in \QBa\mid\sB \cap \overline{\QBa_{r,i}}(\xx) \neq \emptyset \}.\\
\end{align*}
\begin{rem}\label{rem: QBa - l,1 is empty}
It is clear that $\overline{\QBa_{l,0}} (\xx) =\overline{\QBa_{r,0}} (\xx)= \QBa_{l,0}(\xx) = \QBa_{r,0}(\xx) = \emptyset.$    
\end{rem}

Since $\QBa$ is finite, we get finiteness of $\QBa_{l,i}(\xx)$ and  $\QBa_{r,i}(\xx)$ for each $ i \in \{1,0,-1\} $ and $\xx \in \St$. Furthermore, for $\xx \in \St$, we define 
\begin{align*}
    \overline{\QBa_l} (\xx) := \overline{\QBa_{l,1}}(\xx) \cup \overline{\QBa_{l,-1}}(\xx), &\quad  \QBa_l(\xx) := \QBa_{l,1} (\xx) \cup \QBa_{l,-1} (\xx),\\ \overline{\QBa_r} (\xx) := \overline{\QBa_{r,-1}} (\xx) \cup \overline{\QBa_{r,1}} (\xx),
    &\quad \QBa_r (\xx) := \QBa_{r,1} (\xx) \cup \QBa_{r,-1} (\xx).
\end{align*}

\begin{rem}\label{rem: QBa being the downset}
    Let $i \in \{1,-1\},\ \xx \in \St$ and $\sB_1,\sB_2 \in \QBa$. If $\sB_1 \in \QBa_{l,i}(\xx)$ and $\sB_1 \preceq \sB_2$ then $\sB_2 \in \QBa_{l,i}(\xx)$. Similarly, if $\sB_1 \in \QBa_{r,i}(\xx)$ and $\sB_2 \preceq \sB_1$ then $\sB_2 \in \QBa_{r,i}(\xx)$.
\end{rem}

For the computation of the order types of left hammocks, the authors of \cite[\S~6]{SKSK} defined various sets $\OSt{j}{\sB},$ $ \OSt{}{\sB}, \OSt{j}{\xx_0,i;\sB} $ and $\OSt{}{\xx_0,i;\sB} $ associated with the data $\sB \in \QBa,\ \xx_0 \in \St$ and $ i,j \in \{-1,1\} $. In this paper, we need to analyse both left and right hammock orders and hence we define appropriate extensions of these sets. We will see in Proposition \ref{prop: Hammocks as order sum} below that these sets can be used to express a hammock as an order sum of simpler hammocks. For any $\sB \in \QBa$ and $j \in \{1,-1\}$, we define
 $$ \OSt{l,j}{\sB} := \{ \xx \in \St \mid \text{there exists }\uu \in \St\text{ and }\bb \in \sB\text{ such that }\bb\uu\xx \in H_l^j(\xx) \} ,$$ $$ \OSt{r,j}{\sB} := \{ \xx \in \St \mid \text{there exists }\uu \in \St \text{ and }\bb \in \sB \text{ such that }\xx\uu\bb \in  H_r^j(\xx)\}, $$ $$ \OSt{l}{\sB} := \OSt{l,-1}{\sB} \cup \OSt{l,1}{\sB} \text{ and 
} \OSt{r}{\sB} := \OSt{r,-1}{\sB} \cup \OSt{r,1}{\sB} .$$

\begin{rem}\label{rem: reversing the sign of STB}
    For any $ \sB \in \QBa$ and $ j\in \{-1,1\} $, if $\xx \in \OSt{l}{\sB} \left( \text{resp. } \OSt{r}{\sB}\right)$ then there exists $\uu \in \St$ such that $ \uu\xx \in \OSt{l,-j}{\sB}  \left( \text{resp. }\xx\uu \in \OSt{r,-j}{\sB}\right)$ and $\delta(\uu) = j$ whenever $|\uu|>0$.
\end{rem}

Furthermore, for $\xx_0 \in \St$, $i,j \in \{1,-1\}$, $\sB_1 \in \QBa_{l,i}(\xx_0)$ and $\sB_2 \in \QBa_{r,i}(\xx_0)$, define $$ \OSt{l,j}{\xx_0, i ; \sB_1} := \{ \xx \in H_l^i(\xx_0) \mid \xx \in \OSt{l,j}{\sB_1} \text{ or }\delta_l\left( \xx \mid \xx_0 \right) = j \} \cup \{\xx_0\}, $$ $$ \OSt{r,j}{\xx_0, i ; \sB_2} := \{ \xx \in H_r^i(\xx_0) \mid \xx \in \OSt{r,j}{\sB_2} \text{ or }\delta_r\left( \xx \mid \xx_0 \right) = j \} \cup \{ \xx_0 \}, $$ $$ \OSt{l}{\xx_0 , i; \sB_1} := \OSt{l,-1}{\xx_0 , i; \sB_1} \cup \OSt{l,1}{\xx_0 , i; \sB_1} \text{ and } \OSt{r}{\xx_0 , i; \sB_2} := \OSt{r,-1}{\xx_0 , i; \sB_2} \cup \OSt{r,1}{\xx_0 , i; \sB_2}.$$

When we write $\OSt{l,j}{\xx_0 , i ; \sB_1}$ we implicitly imply that $\sB_1 \in \QBa_{l,i}(\xx_0)$. Similarly when we write $\OSt{r,j}{\xx_0 , i ; \sB_2}$ we implicitly imply that $\sB_2 \in \QBa_{r,i}(\xx_0)$. Since the set $\OSt{l}{\xx_0,i;\sB}$ is same as the set $\OSt{}{\xx_0,i;\sB}$ defined in \cite[\S~6]{SKSK}, the following result is a restatement of \cite[Lemma~7.8]{SKSK}.
\begin{prop}\label{prop: Hammocks as order sum}
    For any $\xx_0 \in \St, i \in \{1,-1\}$ and $\sB \in \QBa_{l,i}(\xx_0)$, we have

\begin{equation}\label{hammockorderdecomp}
(H_l^i(\xx_0),<_l) \cong \sum_{\yy \in \left(\OSt{l}{\xx_0,i;\sB},<_l\right)} \left( H_l^{-\varphi_\sB (\yy)}(\yy) , <_l\right),
\end{equation}
where $\varphi_\sB : \OSt{l}{\xx_0,i;\sB} \to \{1,0,-1\}$ is defined as $$ \varphi_\sB(\yy) := \begin{cases}
        0 \quad &\text{if }\yy \in \OSt{l,-1}{\xx_0,i;\sB} \cap  \OSt{l,1}{\xx_0,i;\sB};\\
         1 \quad &\text{if }\yy \notin \OSt{l,-1}{\xx_0,i;\sB};\\
          -1 \quad &\text{if }\yy \notin \OSt{l,1}{\xx_0,i;\sB}.
    \end{cases} $$
\end{prop}

We can write a similar statement for right hammocks. This proposition gives us a surjective order-preserving map $\left( H_l^i(\xx_0 ) , <_l \right) \to \left(\OSt{l}{\xx_0,i;\sB},<_l\right)$ (resp. $\left(H^i_r(\xx_0 ) , <_r\right) \to \left(\OSt{r}{\xx_0 ,i;\sB},<_r\right)$).

The next two results provide necessary and sufficient conditions to determine whether a left hammock is finite or scattered--their proofs rely on the tools developed in \cite{GKS20} and \cite{SKSK}.
\begin{prop}\label{prop: Finite Hammock}
    For $ \xx_0 \in \St$ and $i \in \{1,0,-1\}$, $H_l^i(\xx_0)$ is finite if and only if $\QBa_{l,i}(\xx_0) = \emptyset$. 

    Similarly, $H_r^i(\xx_0)$ is finite if and only if $\QBa_{r,i}(\xx_0) = \emptyset$.
\end{prop}
\begin{proof}
    $(\Longrightarrow)$ We will prove the contrapositive. Suppose $ \QBa_{l,i}(\xx_0) \neq \emptyset$. Take $\bb \in \overline{\QBa_{l,i}}(\xx_0)$ and $\uu \in \St $ such that $\bb\uu\xx_0\in H_l^i(\xx_0)$. Thus $ \{\bb^n\uu\xx_0\mid n \in \N\} \subseteq H_l^i(\xx_0) $. Hence $H_l^i(\xx_0)$ is not finite.
    
    $(\Longleftarrow)$ If ${\QBa_{l,i}}(\xx_0) = \emptyset $ and $\uu\xx_0\in H_l^i(\xx_0)$ then $\uu$ is band-free. Thus  \cite[Proposition~3.1.7]{GKS20} yields $H_l^i(\xx_0)$ is finite. Hence $H_l^i(\xx_0)$ is scattered.
\end{proof}

\begin{thm}\label{prop: Scattered hammock}
    For any $\xx_0 \in \St$ and $i \in \{1,0,-1\}$, $H_l^i(\xx_0)$ is scattered if and only if $ \mathcal{Q}^{\mathrm{Ba}}_{l,i} (\xx_0)$ has only domestic elements. Similarly, $H_r^i(\xx_0)$ is scattered if and only if $ {\QBa_{r,i}} (\xx_0) $ has only domestic elements.
\end{thm}

\begin{proof}
    $(\Longrightarrow)$ We will prove the contrapositive. Suppose there exists a non-domestic $\sB \in \QBa_{l,i} (\xx_0)$. Choose $\bb \in \sB \cap \overline{\QBa_{l,i}} (\xx_0)$ and $\uu \in \St$ such that $\bb\uu\xx_0 \in H_l^i(\xx_0)$. Since $\bb\uu\xx_0 \eqvl \bb^2\uu\xx_0$, we get that $\overline{\QBa_{l,1}}(\bb\uu\xx_0) = \overline{\QBa_{l,1}}(\bb^2\uu\xx_0) $. Thus $\sB$ is non-domestic and minimal for $( \bb\uu\xx_0, \theta_l(\bb))$. Hence \cite[Corollary~11.5, Proposition~10.10]{SKSK} together imply that $ H_l(\bb\uu\xx_0 ) $ is not scattered. Since $ H_l(\bb\uu\xx_0 ) \subseteq H_l^i(\xx_0 )$, we conclude that $H_l^i(\xx_0 )$ is not scattered.

    $(\Longleftarrow)$ Suppose $\mathcal{Q}^{\mathrm{Ba}}_{l,i} (\xx_0)$ contains only domestic elements. We will prove that $H_l^i(\xx_0)$ is scattered by an induction on $\left|{\QBa_{l,i}}(\xx_0)\right|$.

    \textbf{Base Case.} If ${\QBa_{l,i}}(\xx_0) = \emptyset $, then by Proposition \ref{prop: Finite Hammock} we get that $H_l^i(\xx_0)$ is finite, and hence scattered.

    \textbf{Induction Step.} If ${\QBa_{l,i}}(\xx_0) \neq \emptyset$, then choose and fix $\sB \in \QBa$ minimal with respect to $(i,\xx_0)$ from the finite poset $ (\QBa, \preceq)$. By Remark \ref{rem: QBa - l,1 is empty}, $i \neq 0$. Now Equation \eqref{hammockorderdecomp} yields 
    \begin{equation*}\label{eqn:1}
        (H_l^i(\xx_0) , <_l) \cong \sum_{\yy \in\left(\OSt{l}{\xx_0, i; \sB},<_l\right)} (H_l^{-\varphi_\sB (\yy)}(\yy), <_l).
    \end{equation*}
       
    Note that $ {\QBa_{l,-\varphi_\sB (\yy)}}(\yy) \subseteq  {\QBa_{l,i}}(\xx_0) \setminus \{\sB\}$ for each $ \yy \in  \OSt{l}{\xx_0,i;\sB} \subseteq H_l^i(\xx_0)$. Since $\sB \in {\QBa_{l,i}}(\xx_0) $, we get that $| {\QBa_{l,-\varphi_\sB (\yy)}}(\yy)  | < |\QBa_{l,i} (\xx_0)|$ for any $\yy \in \OSt{l}{\xx_0,i;\sB}$. Thus by the induction hypothesis we get that $H_l^{-\varphi_\sB (\yy)}(\yy)$ is scattered for each $\yy \in \OSt{l}{\xx_0,i;\sB}$. Since $\sB$ is domestic, \cite[Corollary~11.8]{SKSK} gives that $\OSt{l}{\xx_0,i;\sB}$ is scattered. Finally, Proposition \ref{prop: SCattered sum of linear orders} applied to the order sum decomposition from Equation \eqref{hammockorderdecomp} gives that $H_l^i(\xx_0)$ is scattered.
\end{proof}

\subsection{Hammock posets for string algebras}\label{excpoints}
Fix $v\in Q_0$. After defining the hammock poset $H(v)$ using the left and right hammocks, we recollect the main results of \S~\ref{order theory} in the context of $H(v)$ (Theorem \ref{hammockordermain}). In particular, we note that $H(v)$ is a finite type abstract $2$-hammock where each scattered box has density strictly bounded above by $\omega^2$. Then we extend $H(v)$ to $\hb(v)$ by adding some periodic infinite strings, and prove a finiteness result (Lemma \ref{excptfin}) for \emph{exceptional} points in $\hb(v)\setminus H(v)$, i.e., the points which lie in the closure of a scattered interval.

Recall from \cite[\S~2.1]{GKS20} that a left $\N$-string is a sequence of syllables $\cdots\alpha_3\alpha_2\alpha_1$ such that each $\alpha_n\cdots\alpha_2\alpha_1$ is a string. Call $\alpha_i$ the $i^\text{th}$ syllable of $\xx$. Similarly, a right $\N$-string is a sequence of syllables $\alpha_1\alpha_2\alpha_3\cdots$ such that each $\alpha_1\alpha_2\cdots\alpha_n$ is a string. The sets of left-$\N$ strings and right-$\N$ strings are denoted $\N$-$\St$ and $\N^*$-$\St$ respectively. Further, a $\Z$-indexed sequence of syllables $\cdots\alpha_{-1}\alpha_0 \alpha_1\cdots$ is said to be a $\Z$-string if $\alpha_{-i}\alpha_{-i+1}\cdots\alpha_{i-1}\alpha_i$ is a string for every $i\in\N$. The set of $\Z$-strings is denoted $\Z$-$\St$.

Say that a sequence $\lan\xx_n\mid n\geq1\ran$ of strings in $H_l(\xx_0) $  is \emph{convergent} if there is a left $\N$-string $\yy$ such that  
\begin{enumerate}
    \item $|\xx_n|\to\infty$ as $n\to\infty$;
    \item there is a sequence $\lan n_k\mid k\in\N^+\ran$ such that the $k^{th}$ syllable of $\yy$ and the $k^{th}$ syllable of $\xx_n$ (from right) are identical for $n\geq n_k$.
\end{enumerate}

The limit of a convergent sequence is unique, and we write $\mathrm{lim}^l_{n \to \infty} \xx_n :=\yy$. Convergence in $\left( H_r(\xx_0 ), <_r \right)$ is defined similarly and is denoted by {the }$ \mathrm{lim}^r_{n \to \infty}$ operator. We define $\widehat H_l(\xx_0) $ $\left( \text{resp. } \widehat H_r(\xx_0) \right)$ to be the extension of $H_l(\xx_0) $ $\left( \text{resp. } H_r(\xx_0) \right)$ by all the left (resp. right) $\N$-strings containing $\xx_0$ as a left substring. Note that the ordering $<_l $ $\left( \text{resp. } <_r \right)$ can be naturally extended to a linear order on $\widehat H_l(\xx_0) $ $\left( \text{resp. } \widehat H_r(\xx_0) \right)$. The order $\widehat H_l(\xx_0) $ $\left( \text{resp. } \widehat H_r(\xx_0) \right)$ is the order-theoretic completion of $H_l(\xx_0) $ $\left( \text{resp. of } H_r(\xx_0) \right) $\cite[Proposition~4.5]{SKSK}.

We now note an important observation regarding hammock orders. 
\begin{rem}\label{rem: Hammock order property}\cite[Lemma~4.5]{schroer2000infinite}
For $\xx,\yy, \zz \in \widehat{H}_l(\xx_0)$ satisfying $\xx\sqcap_l \zz \sqsubset_l \xx\sqcap_l \yy$, $\xx<_l \zz$ if and only if $\yy<_l \zz$. \end{rem}
For any $\xx,\yy\in\widehat H_l(\xx_0)$, there are four possible types of intervals:
\begin{align*}
    [\xx , \yy]_l := \{ \zz \in H_l(\xx_0) \mid \xx\leq_l \zz\leq_l \yy \}, &\qquad (\xx , \yy)_l := \{ \zz \in H_l(\xx_0) \mid \xx <_l \zz<_l \yy \},\\
    [\xx , \yy)_l := \{ \zz \in H_l(\xx_0) \mid \xx\leq_l \zz<_l \yy \}, &\qquad (\xx , \yy]_l := \{ \zz \in H_l(\xx_0) \mid \xx<_l \zz\leq_l \yy \}.
\end{align*}
For $\xx , \yy \in H_r(\xx_0)$, the intervals $ [\xx , \yy]_r , (\xx,\yy)_r , [ \xx,\yy )_r $ and $(\xx , \yy]_r$ are defined in a similar manner. Using these notations, we define the \emph{closure} of an interval in a hammock.
\begin{defn}\label{closure}
For an interval $I \subseteq H_l(\xx_0 ) $, define its closure as $\widehat{I} :=  \{ \zz \in \widehat{H}_l (\xx_0)\mid [\zz, \zz']_l \cup [\zz', \zz]_l\subseteq I \text{ for some }\zz' \in I \}$. Similarly, for any interval $I \subseteq H_r(\xx_0 ) $, define $\widehat{I} :=  \{ \zz \in \widehat{H}_l (\xx_0)\mid [\zz, \zz']_r \cup [\zz', \zz]_r\subseteq I \text{ for some }\zz' \in I \}$.
\end{defn}

Now we recall the definition of the ``$2$-dimensional'' hammock set introduced in \cite{Schroer98hammocksfor}, which is the main object of study in this paper. For $v \in Q_0$, define $$ H(v) := \{ (\uu,\vv) \in \St \times \St \mid s(\uu) = v = t(\vv), \varepsilon(\vv) = 1 = -\sigma (\uu), \uu\vv \in \St  \}.$$ 
Since $H(v)\subseteq H_l\left(1_{(v,1)}\right)\times H_r\left(1_{(v,1)}\right)$, the set $H(v)$ is equipped with an induced poset structure{ under the product order?}. We use the notations $\pi_l$ and $\pi_r$ to denote the projection maps of $H(v)$ to appropriate linear hammocks.

The poset $H(v)$ has several desirable properties.
\begin{thm}\label{hammockordermain}
For any $v\in Q_0$ and $\ii$ a scattered box in $H(v)$, the following are true.
\begin{enumerate}
    \item The poset $H(v)$ is a bounded discrete finite description abstract $2$-hammock.
    \item The poset $H(v)$ is a finite type $2$-hammock.
    \item The scattered box $\ii$ is a finitely presented abstract $2$-hammock.
    \item The density of the scattered box $\ii$ satisfies $d(\ii)<\omega^2$.
\end{enumerate}
\end{thm}
\begin{proof}
It follows from \cite[\S~2.7]{Schroer98hammocksfor} that $H(v)$ is an abstract 2-hammock, and $(1)$ follows from Theorem \ref{SKSKThm}. As a consequence, $(2)$ follows from Lemma \ref{hammockfintype}. Statement $(3)$ follows from Propositions \ref{fd intervals}, \ref{stdproj}, \ref{boxes are hammocks} and Remark \ref{LOfp=LOfd cap Scattered}. As a corollary, $(4)$ follows from Theorem 
\ref{fpdensity}.
\end{proof}

There is a natural extension of $H(v)$ that also includes appropriately pointed periodic $\Z$-strings. $$ \hb(v) := H(v) \cup \{ \left( \prescript{\infty}{}{\bb} , \bb^\infty \right) \in \N\text{-}\St \times \N^*\text{-}\St \mid \bb \in \Cyc, \varepsilon(\bb) = 1 = -\sigma (\bb), s(\bb) = v = t(\bb) \}. $$
Since $\hb(v) \subseteq \widehat H_l\left(1_{(v,1)}\right) \times \widehat H_r\left( 1_{(v,1)} \right)$, it also gets equipped with an induced partial order, which we again denote by $<$. We continue to use the notations $\pi_l$ and $\pi_r$ to denote projection maps onto the completions of linear hammocks.

Now we want to investigate local finiteness properties of the difference $\hb(v)\setminus H(v)$. In view of Definition \ref{closure} of the closure of an interval in a hammock linear order, given a box $ \ii \subseteq H(v) $ define 
\begin{equation*}
\overline{\ii} :=  \ii \cup \{\left(\prescript{\infty}{}{\bb}, \bb^\infty\right) \in \hb(v)\mid \prescript{\infty}{}{\bb} \in \widehat{\pi_l (\ii)} \text{ and } \bb^\infty \in \widehat{\pi_r(\ii )} \}.
\end{equation*}

\begin{defn}
    Say that $ (\infb, \binf) \in \hb(v) $ is \emph{exceptional} if there is a scattered box $\ii\subseteq H(v) $ such that $ (\infb, \binf) \in \ol \ii $. 
    
    Define $\mathsf{Cyc}(v):=\{\bb\in\Cyc\mid(^\infty\bb,\bb^\infty)\in\hb(v)\}$; this set is in bijection with $\hb(v)\setminus H(v)$ Say that $\bb\in\Cyc$ is \emph{exceptional for} $v$ if $\bb\in\mathsf{Cyc}(v)$ and the corresponding point in $\hb(v)$ is exceptional.
\end{defn}

We want to show the finiteness of the set of exceptional points. To this end, let us recall some finite sets of bands originally introduced in \cite[\S~8]{SKSK}. For each $\sB \in \QBa$, define
\begin{align*}
    \BalB := \{\bb \in \sB \mid \text{ for any }\bb_1\in \sB \text{ and }\uu\in \St \text{ if }\bb_1\uu\bb \in \St \text{ then }\prescript{\infty}{}{\bb} \leq_l \prescript{\infty}{}{\bb_1}\uu\bb  \},\\
    \BalbB := \{\bb \in\sB \mid \text{ for any }\bb_1\in \sB \text{ and }\uu\in \St \text{ if }\bb_1\uu\bb \in \St \text{ then }\prescript{\infty}{}{\bb_1}\uu\bb \leq_l \prescript{\infty}{}{\bb}  \}.
\end{align*}

\begin{prop}(\cite[Corollaries~8.26,8.28]{SKSK} and \cite[Theorem~3.1.6]{GKS20})\label{balbfiniteness}
For any $\sB \in \QBa$, the sets $\BalB$ and $\BalbB$ are non-empty and finite.
\end{prop}
The proofs of the cited results also provide algorithms to compute these two sets.
\begin{exmp}
    For the algebra $GP_{2,3}$ from Example \ref{Ex: Standard example}, note that $aB \approx aB^2$ in the preorder $(\QBa_0 , \preceq)$. In fact, $|\QBa|=2$. Take $\sB \in \QBa$ such that $\{aB , aB^2\} \subseteq \sB$. It is readily verified that $aB \in \BalB$ and $aB^2 \in \BalbB$.
\end{exmp}

There is an obvious consequence of the above result for domestic elements of $\QBa$.
\begin{rem}\label{rem: BalB for domestic}
    For every domestic $\sB \in \QBa$, $\BalB = \sB = \BalbB$.
\end{rem}

No two bands in $\bigcup_{\sB\in\QBa}\BalB$, or their cyclic permutations are comparable as strings.
\begin{prop}\label{rem: Bl not being substrings}
    For any $\sB,\sB' \in \QBa$, $\bb_1\in \BalB$ and $\bb_2 \in \mathsf{Ba}_l(\sB')$, if $\bb_1\neq \bb_2$ then $\bb_1\not\sqsubset \prescript{\infty}{}{\bb_2}$ and $\bb_2\not\sqsubset  \prescript{\infty}{}{\bb_1}.$
\end{prop}

\begin{proof}
    We will prove the contrapositive of this statement. Take any $\bb_1 \in \BalB$ and $\bb_2 \in\mathsf{Ba}_l(\sB')$ such that $\bb_1 \sqsubset \prescript{\infty}{}{\bb_2}$. Then there exist $\uu_1,\uu_2 \in \St$ and $n \in \N^+$ such that $\uu_2\bb_1\uu_1=\bb_2^n$. Since $\bb_2^{n+2} = \bb_2\uu_2\bb_1\uu_1\bb_2$ is a string, we get that $\bb_1 \approx \bb_2$, and hence $\sB'=\sB$. Since $ \bb_2 \in \BalB$ we obtain $\prescript{\infty}{}{\bb_2} \leq_l \prescript{\infty}{}{\bb_1}\uu_1\bb_2 $. Similarly, since $\bb_1\in\BalB$ we get that $ \prescript{\infty}{}{\bb_1} \leq_l \prescript{\infty}{}{\bb_2}\uu_2\bb_1$. Combining the data we get $\prescript{\infty}{}{\bb_2} \leq_l \prescript{\infty}{}{\bb_1}\uu_1\bb_2 \leq_l \prescript{\infty}{}{\bb_2}\uu_2\bb_1\uu_1\bb_2=\prescript{\infty}{}{\bb_2}$. Thus $\prescript{\infty}{}{\bb_2}= \prescript{\infty}{}{\bb_1}\uu_1$. Similarly we can argue that $\prescript{\infty}{}{\bb_1}= \prescript{\infty}{}{\bb_2}\uu_2$, and hence $\prescript{\infty}{}{\bb_2}= \prescript{\infty}{}{\bb_1}$. Since no two distinct elements of $\QBa_0$ are cyclic permutations of each other, we conclude $\bb_1=\bb_2$.
    Since $\bb_2^n = \uu_1\bb_1\uu_2$, we get that $ \prescript{\infty}{}{\bb_2}\uu_1\bb_1\uu_2\bb_2 = \prescript{\infty}{}{\bb_2} \leq_l \prescript{\infty}{}{\bb_1}\uu_2\bb_2$. Thus $\prescript{\infty}{}{\bb_2}\uu_1\bb_1 \leq_l \prescript{\infty}{}{\bb_1} $. Since $\bb_1 \in \BalB$, we get that $\prescript{\infty}{}{\bb_2}\uu_1\bb_1 = \prescript{\infty}{}{\bb_1}  $. Thus $\bb_1 = \bb_2$.
\end{proof}

We can define the right counterparts of the sets $\BalB$ and $\BalbB$.
\begin{align*}
    \mathsf{Ba}_r(\sB) := \{\bb \in \sB \mid \text{ for any }\bb_1\in \sB \text{ and }\uu\in \St \text{ if }\bb\uu\bb_1 \in \St \text{ then } \bb^\infty\leq_r\bb\uu\bb_1^\infty \};\\
    \mathsf{Ba}_{\rb}(\sB) := \{\bb \in\sB \mid \text{ for any }\bb_1\in \sB \text{ and }\uu\in \St \text{ if }\bb_1\uu\bb \in \St \text{ then } \bb\uu\bb_1^\infty\leq_r \bb^\infty \}.
\end{align*}
All results stated and proved for $\BalB$ and $\BalbB$ are also true for their counterparts.

Here is the promised result on the finiteness of the set of exceptional points in a hammock--this is the last ingredient in our recipe to compute $\st(\Lambda)$.
\begin{lem}\label{excptfin}
Suppose $v\in Q_0$, $\bb_1\in\mathsf{Cyc}(v)$, and $\bb_1$ is a cyclic permutation of $\bb\in\QBa_0$. If $\bb_1$ is exceptional for $v$ then $\bb \in \bigcup\limits_{\sB \in \QBa}\left(\BalB \cup \BalbB\right)$. Therefore, the set of exceptional points in $\hb(v)\setminus H(v)$ is finite.
\end{lem}
\begin{proof}
Suppose $\bb_1$ is exceptional for $v$. Take $\sB \in \QBa$ such that $\bb\in \sB$. Let $ \ii \subseteq H(v) $ be a scattered box such that  $(\infb_1, \bb_1^\infty) \in \ol\ii $. Since $\prescript{\infty}{}{\bb_1} \in \widehat{\pi_l(\ii)}$, take $\xx \in \pi_l(\ii)$ such that $[\prescript{\infty}{}{\bb_1} , \xx]_l \cup [\xx , \prescript{\infty}{}{\bb_1 }]_l \subseteq \pi_l(\ii)$. Without loss of generality, assume $\xx<_l \prescript{\infty}{}{\bb_1}$ so that $[\xx,\prescript{\infty}{}{\bb_1})_l \subseteq \pi_l(\ii)$. We will show $\bb \in \BalB$ to complete the proof. (If $\prescript{\infty}{}{\bb_1}<_l\xx$ then a similar proof will show that $\bb \in \BalbB$.)

Towards a contradiction, assume $\bb \notin \BalB$. Then by Remark \ref{rem: BalB for domestic}, $\sB$ is non-domestic. Take $\vv_1 , \vv_2 \in \St$ such that $\bb_1 = \vv_1\vv_2$ and $\bb = \vv_2\vv_1$. Then $\prescript{\infty}{}{\bb_1} = \prescript{\infty}{}{(\vv_1\vv_2)} = \prescript{\infty}{}{(\vv_2\vv_1)}\vv_2 = \prescript{\infty}{}{\bb}\vv_2 $. Since $\bb \notin \BalB$ there exist $ \uu\in \St$ and $\bb' \in \sB$ such that $\prescript{\infty}{}{\bb'}\uu\bb <_l \prescript{\infty}{}{\bb}$. Let $\ww := \prescript{\infty}{}{\bb'}\uu\bb \sqcap_l \prescript{\infty}{}{\bb}$. Take $n >0$ such that $|\bb_1^n| > |\xx|$. Note that $\prescript{\infty}{}{\bb_1} \sqcap_l \xx \sqsubset_l \ww\bb^n\vv_2 \sqsubseteq_l \prescript{\infty}{}{\bb_1} \sqcap_l \yy$ for any $\yy \in H_l^{-1}(\ww\bb^n\vv_2)$. Thus, since $\xx <_l \prescript{\infty}{}{\bb_1}$, Remark \ref{rem: Hammock order property} gives $\xx <_l \yy$ for each $\yy\in H_l^{-1}(\ww\bb^n\vv_2)$.

Since $\prescript{\infty}{}{\bb'}\uu\bb <_l \prescript{\infty}{}{\bb}$, we get that $\theta_l(\prescript{\infty}{}{\bb_1} \mid \ww\bb_n\vv_2) =\theta_l(\prescript{\infty}{}{\bb}\vv_2 \mid \ww\bb^n\vv_2) =  \theta_l(\prescript{\infty}{}{\bb} \mid \ww) = 1$, and hence $ \yy <_l \prescript{\infty}{}{\bb_1} $ for each $\yy \in H_l^{-1}(\ww\bb^n\vv_2)$. Thus $H_l^{-1}(\ww\bb^n\vv_2) \subseteq (\prescript{\infty}{}{\bb_1}, \xx]_l$.

Again using $\prescript{\infty}{}{\bb'}\uu\bb <_l \prescript{\infty}{}{\bb}$ we get that $\theta_l(\prescript{\infty}{}{\bb'}\uu\bb^{n+1}\vv_2 \mid \ww\bb^n\vv_2) = \theta_l(\prescript{\infty}{}{\bb'}\uu\bb \mid \ww) = -1 $, and hence $\bb' \in \overline{\QBa_{l,-1}}(\ww\bb^n\vv_2)$ so that $\sB \in \QBa_{l,-1}(\ww\bb^n\vv_2)$. As $\sB$ is non-domestic, Theorem \ref{prop: Scattered hammock} gives that $H_l^{-1}(\ww\bb^n\vv_2)$ is not scattered. Since $H_l^{-1}(\ww\bb^n\vv_2) \subseteq (\prescript{\infty}{}{\bb_1}, \xx]_l\subseteq\pi_l(\ii)$, we also conclude that $\pi_l(\ii)$ is not scattered, which is a contradiction to the hypothesis that $\ii$ is a scattered box via Proposition \ref{stdproj}.

The final statement follows from the finiteness of sets $ \BalB $ and $ \BalbB $ (Proposition \ref{balbfiniteness}), and of $\QBa$ \cite[Proposition~5.6]{SKSK}.
\end{proof}
The proof of the above lemma has the following useful consequence.
\begin{cor}\label{cor: Exceptional point in left intervals}
Suppose $v\in Q_0$, $\bb_1\in\mathsf{Cyc}(v)$ is a cyclic permutation of $\bb\in \QBa_0$ and $\ii \subseteq H(v)$ is an MSB such that $(\prescript{\infty}{}{\bb_1}, \bb_1^{\infty}) \in \ol\ii\setminus \ii$.
\begin{itemize}
    \item If $\bb \notin \bigcup\limits_{\sB \in \QBa}\BalB$ (resp. $ \bb \notin \bigcup\limits_{\sB \in \QBa}\BalbB)$ then $\prescript{\infty}{}{\bb_1}$ is the minimum (resp. maximum) of $\widehat{\pi_l(\ii)}$.
    \item If $\bb \notin \bigcup\limits_{\sB \in \QBa}\BarB$ (resp. $ \bb \notin \bigcup\limits_{\sB \in \QBa}\BarbB$)  then ${\bb}_1^\infty$ is the minimum (resp. maximum) of $\widehat{\pi_r(\ii)}$.
\end{itemize}
\end{cor}

\subsection{The proof for the case of string algebras}\label{conjproofstringalg}
In this subsection, we use the relation between the representation theory of string algebras and the order theory of their hammock posets to translate our bound on the densities of scattered boxes in $ H(v) $ (Theorem \ref{hammockordermain} $(4)$) to a bound on the ranks of graph maps outside $\rad^\infty$, thereby proving Theorem \ref{oldconj} for string algebras.

For $v \in Q_0$, we start by describing the extension $\HH(v)$ of $\hb(v)$ that combinatorially encapsulates the factorizations of base-point preserving morphisms between finite-dimensional indecomposable $\Lambda$-modules (Theorem \ref{hammockgraphmap}). We then relate the ranks of morphisms with densities of closures (in $\hb(v)$) of intervals in $H(v)$ (Theorems \ref{density=rank-1} and \ref{noninfscatteredequivalence}). We then argue that $\st(\Lambda)$ is the supremum of the densities of the closures of MSBs in hammocks (Equation \eqref{stablerank7}), which is then used to conclude Theorem \ref{oldconj} for the case of string algebras.

Fix $v\in Q_0$. Recall the 2-hammock posets $H(v)$ and $\hb(v)$ from \S~\ref{excpoints}. We now describe a further extension of $\hb(v)$ to accommodate the index sets of pointed band modules.
\begin{defn}
For each $\lambda\in\K^*$ and $\bb\in\mathsf{Cyc}(v)$, set $\bb[\lambda]:=\omega^*\times\omega$. Define the \emph{extended hammock poset} $\HH (v)$ to be the poset obtained by replacing in $\hb(v)$ each element $(^\infty\bb,\bb^\infty)\in\hb(v)\setminus H(v)$ by $\bigsqcup_{\lambda\in\K^*}\bb[\lambda]$, where $x\in\bb[\lambda]$ and $y\in\bb[\lambda']$ are incomparable if $\lambda\neq\lambda'$.
\end{defn}

\begin{rem}
The poset $H(v)$ is a subposet of $\HH(v)$ and there is a natural projection map $\pi_v:\HH(v)\to\hb(v)$ that is identity when restricted to $H(v)$ and that sends each point in $\bb[\lambda]$ to $(^\infty\bb,\bb^\infty)\in\hb(v)$. 
\end{rem}

There is a map $\mathsf M:\HH(v)\to\Lambda\dmod$ \cite[\S~5.4]{Schroer98hammocksfor} that associates to $x:=(\xx_1,\xx_2)\in H(v)$ an appropriately pointed string module $\mathsf{M}(x):=\mathrm{M}(\xx_1\xx_2)$ and to $(-n,m)\in\bb[\lambda]=\omega^*\times\omega$ for $\bb\in\mathsf{Cyc}(v)$ an appropriately pointed band module $\mathsf{M}(\bb, \lambda, -n, m)\defeq\mathrm{B}(\bb_1,n+m+1,\lambda)$, where $\bb_1\in\QBa_0$ is the unique element such that $\bb$ is a cyclic permutation of $\bb_1$. The map $\mathsf M$ satisfies the following important property.
\begin{thm}(\cite{crawley1998infinite},\cite{krause1991maps})\label{hammockgraphmap}
If $M,N$ are indecomposable modules, i.e., either a string or a band module in $\Lambda\dmod$, then the set of graph maps in $\mathrm{Hom}(M,N)$ is a basis for the finite-dimensional vector space $\mathrm{Hom}(M,N)$.

There is a necessarily unique base-point preserving graph map $ f_{x,y}^v\in\mathrm{Hom}(\mathsf M(x),\mathsf M(y))$ if and only if $x\leq y$ in $\HH(v)$. Moreover, $f_{x,y}^v=f_{z,y}^v\circ f_{x,z}^v$ whenever $x\leq z\leq y$ in $\HH(v)$.
\end{thm}

For $x\leq y$ in $\HH(v)$, define the corresponding intervals $\ii_{x,y}^v:=\{z\in H(v)\mid \pi_v(x)\leq z\leq\pi_v(y)\}$ and $\ol\ii_{x,y}^v:=\{z\in\hb(v)\mid \pi_v(x)\leq z\leq\pi_v(y)\}$ in $H(v)$ and $\hb(v)$ respectively. Indeed, $\ol\ii_{x,y}^v$ is the closure of $\ii_{x,y}^v$ in $\hb(v)$.

The next result relates the densities of bounded scattered intervals in $\hb(v)$ with the ranks of certain morphisms.
\begin{thm}\label{density=rank-1}
Suppose $x<y$ in $\HH(v)$ and the interval $\ol\ii_{x,y}^v$ in $\hb(v)$ is scattered and infinite. Then
\begin{enumerate}
    \item $\ol\ii_{x,y}^v\setminus\ii_{x,y}^v$ is finite;
    \item $|d(\ol\ii_{x,y}^v)-d(\ii_{x,y}^v)|\leq|\ol\ii_{x,y}^v\setminus\ii_{x,y}^v|$;
    \item $\omega<d(\ii_{x,y}^v)<\omega^2$;
    \item $\rk(f_{x,y}^v)=d(\ii_{x,y}^v)-1$.
\end{enumerate}
\end{thm}

\begin{proof}
Since $\ol\ii_{x,y}^v$ is scattered, $\ii_{x,y}^v=\ol\ii_{x,y}^v\cap H(v)$ is a scattered box in $H(v)$, and hence a finitely presented discrete $2$-hammock by Theorem \ref{hammockordermain}(3). Each element in $\ol\ii_{x,y}^v\setminus\ii_{x,y}^v$ is exceptional, and hence $\ol\ii_{x,y}^v\setminus\ii_{x,y}^v$ is finite by Lemma \ref{excptfin}. As a consequence, $\ii_{x,y}^v$ is infinite.

If $L$ is a finite linear order then $d(L)=|L|-1$, and hence Theorem \ref{Schroer formula} implies that a bounded discrete $2$-hammock $H$ is finite if and only if $d(H)<\omega$. Combining this with the hypothesis that $\ii_{x,y}^v$ is scattered and infinite, we conclude $\omega\leq d(\ii_{x,y}^v)<\infty$. Furthermore, since $\ii_{x,y}^v$ is finitely presented, Theorem \ref{hammockordermain}(4) yields that $d(\ii_{x,y}^v)<\omega^2$. Since $\ol\ii_{x,y}^v\setminus\ii_{x,y}^v$ is finite, and adding a point increases the density by at most $1$, we conclude that $|d(\ol\ii_{x,y}^v)-d(\ii_{x,y}^v)|\leq|\ol\ii_{x,y}^v\setminus\ii_{x,y}^v|$. Thus $\omega\leq d(\ol\ii_{x,y}^v)<\omega^2$. In addiction, the proof of Proposition \ref{bddlindensitysuccessor} generalizes to give that $d(\ol\ii_{x,y}^v)$ is a successor ordinal. Therefore, $d(\ol\ii_{x,y}^v)=\omega\cdot n_0+p_0$ for some $n_0,p_0\in\N^+$.

The finiteness of $\ol\ii_{x,y}^v\setminus\ii_{x,y}^v$ ensures that $\ol\ii_{x,y}^v$ is a hammock poset in the sense of \cite[\S~3.8]{Schroer98hammocksfor}, and hence the proof of the theorem on \cite[p.66]{Schroer98hammocksfor} together with the remark after that goes through for $\ol\ii_{x,y}^v$. This theorem states that, for $n\geq 1$, $f_{x,y}^v\in\rad^{\omega\cdot n+p-1}$ if and only if there is a linear suborder $T$ of $\ol\ii_{x,y}^v$ such that $d(T)=\omega\cdot n+p$. Combining this theorem, Proposition \ref{densityposetunbdd}, which states that $d(\ol\ii_{x,y}^v)=\mathrm{sup}\{d(T)\mid T\text{ is a linear suborder of }P\}$, and the conclusion of the above paragraph, we conclude that $f_{x,y}^v\in\rad^{\omega\cdot n_0+p_0-1}\setminus\rad^{\omega\cdot n_0+p_0}$. In other words, $\rk(f_{x,y}^v)=d(\ol\ii_{x,y}^v)-1$.
\end{proof}

The following result documents various equivalent conditions for morphisms to not lie in the stable radical.
\begin{thm}\label{noninfscatteredequivalence}
Suppose $v\in Q_0$ and $x\leq y$ in $\HH(v)$. Then the following are equivalent.
\begin{enumerate}
    \item $\ii_{x,y}^v\subseteq H(v)$ is scattered.
    \item $\ol\ii_{x,y}^v\subseteq\hb(v)$ is scattered.
    \item $\rk(f_{x,y}^v)\neq\infty$.
\end{enumerate}
\end{thm}

\begin{proof}
We will show $(1)\Rightarrow(2)\Rightarrow(3)\Rightarrow(1)$.

\noindent{}$(1)\Rightarrow(2)$: This follows from Lemma \ref{excptfin}.

\noindent{}$(2)\Rightarrow(3)$: This follows from Theorem \ref{density=rank-1}.

\noindent{}$(3)\Rightarrow(1)$: We show the contrapositive. Suppose $\varphi:\eta\to\ii_{x,y}^v$ is an order embedding. For each $i<j$ in $\eta$, let $p_i:\mathsf M(x)\to\mathsf M(\varphi(i))$, $h_{i,j}:\mathsf M(\varphi(i))\to \mathsf M(\varphi(j))$ and $q_j:\mathsf M(\varphi(j))\to\mathsf M(y)$ be the canonical graph maps. Since $x<\varphi(i)<\varphi(j)<y$, we clearly have $f_{x,y}^v=q_j\circ h_{i,j}\circ p_i$.

\noindent{\textbf{Claim:}} For each ordinal $\alpha\geq 1$ and $i<j$ in $\eta$, $p_i,h_{i,j},q_j\in\rad^\alpha$.

\noindent{\textit{Proof of the claim.}} We prove the claim using transfinite induction on $\alpha$.

The base case $\alpha=1$ is clear from the definition of a graph map since $x<\varphi(i)<\varphi(j)<y$.

If $\alpha=\beta+1$ for some $\beta\geq1$ and the claim is true for $\beta$ then we show that the claim if true for $\alpha$. Choose $k_1<i<k_2<j<k_3$ in $\eta$. Then by the induction hypothesis we have $p_{k_1},q_{k_2},h_{k_1,i},h_{i,k_2},h_{k_2,j},h_{j,k_3}\in\rad^\beta$. Clearly $p_i=h_{k_1,i}\circ p_{k_1},h_{i,j}=h_{k_2,j}\circ h_{i,k_2},q_j=q_{k_3}\circ h_{j,k_3}$, and hence all these morphisms lie in $\rad^{\beta+1}=\rad^\alpha$ using the arithmetic of radical powers.

Finally, if $\alpha$ is a limit ordinal then recall that $\rad^\alpha=\bigcap_{\beta<\alpha}\rad^\beta$. Since by induction hypothesis, $p_i,h_{i,j},q_j\in\rad^\beta$ for all $\beta<\alpha$, we immediately get $p_i,h_{i,j},q_j\in\rad^\alpha$.\hfill$\vartriangle$

Since $\st(\Lambda)$ is an ordinal the claim gives $f_{x,y}^v=q_j\circ h_{i,j}\circ p_i\in\rad^{\st(\Lambda)}=\rad^\infty$, and thus $\rk(f_{x,y}^v)=\infty$.
\end{proof}

Now we are ready to prove Theorem \ref{oldconj} for the case when $\Lambda$ is a string algebra.

\begin{proof}(of Theorem \ref{oldconj} for string algebras)
Suppose $\Lambda$ is a string algebra with at least one band. This hypothesis in conjunction with a celebrated theorem of Auslander, which states that an algebra $\Lambda'$ is of finite representation type if and only if $\st(\Lambda')<\omega$ and $\rad^{\st(\Lambda')}=0$, gives us the lower bound $\omega\leq\st(\Lambda)$. Now
\begin{equation}\label{stablerank}
\st(\Lambda)=\sup\{\rk(f)+1\mid f\text{ is a morphism in }\Lambda\dmod,\ f\notin\rad^\infty\}.
\end{equation}

Since $\Lambda\dmod$ is a Krull-Schmidt category, if $f:\bigoplus_{i\in[n]}A_i\to\bigoplus_{j\in[m]}B_j$ is a morphism for indecomposable $A_i$ and $B_j$ with components $f_{i,j}:A_i\to B_j$, then $\rk(f)=\max_{i,j}\{\rk(f_{i,j})\}$. Thus Equation \eqref{stablerank} becomes
\begin{equation}\label{stablerank2}
\st(\Lambda)=\sup\{\rk(f)+1\mid f\text{ is a morphism between indecomposable modules in }\Lambda\dmod,\ f\notin\rad^\infty\}.
\end{equation}

Now the first part of Theorem \ref{hammockgraphmap} states that each map between indecomposables is a finite linear combination of graph maps, and thus using the fact that each $\rad^\alpha$ is an ideal, Equation \eqref{stablerank2} becomes
\begin{equation}\label{stablerank3}
\st(\Lambda)=\sup\{\rk(f)+1\mid f\text{ is a graph map between indecomposable modules in }\Lambda\dmod,\ f\notin\rad^\infty\}.
\end{equation}

Since each graph map $f$ has an associated string $\vv$ that lies in a hammock $H(v)$ for some $v\in Q_0$, using the second part of Theorem \ref{hammockgraphmap} Equation \eqref{stablerank3} becomes
\begin{equation}\label{stablerank4}
\st(\Lambda)=\max_{v\in Q_0}\ \sup\{\rk(f_{x,y}^v)+1\mid x\leq y\text{ in }\HH(v),\ f_{x,y}^v\notin\rad^\infty\}.
\end{equation}

Now using $(3)\Rightarrow(2)$ of Theorem \ref{noninfscatteredequivalence}, Equation \eqref{stablerank4} becomes
\begin{equation}\label{stablerank5}
\st(\Lambda)=\max_{v\in Q_0}\ \sup\{\rk(f_{x,y}^v)+1\mid x\leq y\text{ in }\HH(v),\ \ol\ii_{x,y}^v\text{ is scattered}\}.
\end{equation}

Since $|\QBa_0|\geq 1$, the connectedness of $\Q$ together with Proposition \ref{prop: Finite Hammock} and its other three variants implies that $H(v)$ is infinite for each $v\in Q_0$. Furthermore, since $H(v)$ is discrete by Theorem \ref{SKSKThm}, we conclude that each MSB in $H(v)$ is infinite, and hence so is its closure in $\hb(v)$. Therefore, using Theorem \ref{density=rank-1}(4) Equation \eqref{stablerank5} becomes
\begin{equation}\label{stablerank6}
\st(\Lambda)=\max_{v\in Q_0}\ \sup\{d(\ol\ii_{x,y}^v)\mid x\leq y\text{ in }\HH(v),\ \ol\ii_{x,y}^v\text{ is scattered}\}.
\end{equation}

Now $\ol\ii_{x,y}^v$ is scattered if and only if $\ii_{x,y}^v$ is scattered, thanks to $(1)\Leftrightarrow (2)$ of Theorem \ref{noninfscatteredequivalence}, and each scattered box in $H(v)$ is contained in an MSB. Thus if $\ii$ is an MSB in $H(v)$ then $d(\ii)=\sup\{d(\ii_{x,y})\mid x\leq y,\ \ii_{x,y}^v\subseteq\ii\}$. Since $H(v)$ is a finite type $2$-hammock by Theorem \ref{hammockordermain}(2), we know that there are only finitely many different order types of MSBs in $H(v)$. Furthermore, since $\ol\ii\setminus\ii$ only consists of exceptional points, Lemma \ref{excpoints} ensures that there are only finitely many order types of MSBs in $\hb(v)$ (by a slight abuse of terminology since we never defined MSBs for $\hb(v)$). Thus Equation \eqref{stablerank6} becomes
\begin{equation}\label{stablerank7}
\st(\Lambda)=\max_{v\in Q_0}\ \max\{d(\ol\ii)\mid \ol\ii\text{ is an MSB in }\hb(v)\}.
\end{equation}
For any MSB $\ol\ii$ in $\hb(v)$, $\ol\ii\cap H(v)$ is an MSB in $H(v)$. Moreover, since $\ol\ii\setminus(\ol\ii\cap H(v))$ is finite, we can argue similar to the proof of Theorem \ref{density=rank-1}(2) to obtain $|d(\ol\ii)-d(\ol\ii\cap H(v))|\leq|\ol\ii\setminus(\ol\ii\cap H(v))|$. Therefore, Equation \eqref{stablerank7} together with Theorem \ref{hammockordermain}(4) gives $\st(\Lambda)<\omega^2$, as required.
\end{proof}

\subsection{The proof for the case of special biserial algebras}\label{biserial}
Now consider $\Lambda = \K\Q / \langle \rho \rangle$ to be a special biserial algebra with at least band. We recall some standard facts about the representation theory of $\Lambda$ in comparison with the representation theory of a string algebra from \cite{skowronski1983representation} and \cite{wald1985tame}; however we follow \cite[\S~5]{schroer2000infinite} for the treatment of the material. Let $\bar{\rho}$ be the set of all the paths $p$ for which $(p-\lambda q) \in \langle \rho \rangle$ for some path $q$ sharing the source and target with $p$ and $\lambda \in \K$. In particular, $\bar\rho$ contains all the monomial relations in $\rho$. Thus the algebra $ \bar\Lambda := \K\Q / \langle \bar\rho \rangle $ is a string algebra. Moreover, $\langle \rho \rangle \subseteq\langle \bar\rho\rangle$. The surjective map $\Lambda\to\bar\Lambda$ induces a canonical $\Lambda$-module structure on every $\bar\Lambda$-module.

A finite-dimensional indecomposable $\Lambda$-module is either a string module, band module or a projective-injective module. In fact, the only indecomposable $\Lambda$-modules not annihilated by $\bar\rho$ are the projective-injective modules, and there is only a finite number of such modules. Each such module can be written as $P(p-\lambda q)$ for some $(p-\lambda q) \in \langle \rho \rangle $ with $\lambda \in \K^*$ and $p,q\notin\langle\rho\rangle$, and occurs in an Auslander-Reiten sequence $$ 0 \to M(D_1) \to M(C_1)\oplus P(p - \lambda q) \oplus M(C_2) \to M(D_2) \to 0,$$ where $M(D_1), M(C_1), M(C_2), M(D_2)$ are string modules over $\bar\Lambda$. The maps $M(D_1) \to P(p - \lambda q)$ and $P(p- \lambda q) \to M(D_2)$ are called the \emph{sink map} and the \emph{source map} respectively.

For any ordinal $\alpha$ and $M_1, M_2\in\Lambda\dmod$, define $\rad^\alpha (M_1 , M_2) := \rad^\alpha \cap \mathrm{Hom}_{\Lambda} (M_1 , M_2)$.
\begin{thm}
If $M_1 , M_2\in\Lambda\dmod$ are indecomposable and $\alpha \geq1$ then $$\dim_\K(\rad^{\omega \cdot \alpha} (M_1 , M_2) / \rad^{\omega \cdot (\alpha +1)} (M_1 , M_2))=\dim_\K(\mathrm{rad}_{\bar\Lambda}^{\omega \cdot \alpha} (M'_1 , M'_2) / \mathrm{rad}_{\bar\Lambda}^{\omega \cdot (\alpha +1)} (M'_1 , M'_2)),$$ where $M_1=M'_1$ if $M_1$ is a string or a band module, and $M_2=M'_2$ if $M_2$ is a string or a band module; otherwise $M'_1$ is the target of the source map of $M_1$ and $M'_2$ is the source of the sink map of $M_2$.
\end{thm}
Recall from the previous section that the proof of Theorem \ref{oldconj} for the string algebra $\bar\Lambda$ yields $\omega\leq\st(\bar\Lambda)<\omega^2$. The theorem above together with the finiteness of (iso-classes of) projective-injective $\Lambda$-modules ensures that $|\st(\Lambda)-\st(\bar\Lambda)|$ is finite. This completes the proof of Theorem \ref{oldconj} for all special biserial algebras.

\section{An algorithm to estimate the stable rank}\label{sec: alg}
Let $\Lambda=\K\Q/\lan\rho\ran$ be a string algebra with at least one band throughout this section. Now that we have established Theorem \ref{oldconj}, which states that $\st(\Lambda)=\omega\cdot n+m$ for some $n\in\N^+$ and $m\in\N$, the purpose of this section is to show that the ordinal $\omega\cdot n+m$ can be algorithmically estimated up to a finite error.

In \S~\ref{sec:H-eq}, we recall the concepts of (left and/or right) $H$-equivalence (Definition \ref{Defn: Hl Hr equivalence}) on the set of strings, and show finiteness of the set of equivalence classes for any of these equivalence relations. This is the key behind the fact that the hammock posets $H(v)$ are of finite type (Corollary \ref{Hlreqdensity}). In \S~\ref{sec:alg msi lin ord}, for $\sB \in \QBa$, we describe immediate neighbours of strings in the linear orders $(\OSt{l}{\xx_0 ,i;\sB }, <_l)$ and $(\OSt{r}{\xx_0 ,i;\sB }, <_l) $ as well as the nearest infinite string of given finite string (Theorem \ref{prop: H-reduced infty bux}). At the end of this subsection, we describe the adaptation of the algorithm given by Schr\"{o}er \cite[\S~4.10]{Schroer98hammocksfor} that could be used in conjunction with Theorem \ref{prop: Scattered hammock} to compute the density of any left or right hammock for a non-domestic string algebra. In \S~\ref{sec: Density of a box}, we devise a method to compute the density of a maximal scattered box using the previous algorithm, and finally we show how to compute the stable rank of the string algebra from this algorithm. At the end of this subsection, we also give a proof of Theorem \ref{reversest}.

\subsection{$H$-equivalence}\label{sec:H-eq}
In this subsection, we recall the concept of $H_l\left(\text{resp. }H_r\right)$-equivalence between strings that was introduced in \cite[\S~5]{sardar2021variations} to capture when their left (resp. right) hammocks are isomorphic. We also introduce the concept of $H_l$-\emph{reduced strings}--loosely speaking these are the strings which are not $H_l$-equivalent to any shorter strings. We prove two finiteness results, namely that of the set of $H_l$-reduced strings (Proposition \ref{prop: H-reduced finitely many}) and of the set of $H_l$-equivalence classes of strings (Proposition \ref{prop: finite H_l-equivalence class}). At the end we define the ``2-dimensional'' analogues of these concepts and state appropriate generalizations of the finiteness results (Propositions \ref{Hlrredfiniteness} and \ref{prop: Hlr reduced representative}). These finiteness results will be used to obtain an alternate proof of the finiteness of the order type of MSBs in the upcoming subsections.
\begin{defn}\label{Defn: Hl Hr equivalence}
Say that $\xx,\zz\in\St$ are \emph{$H_l$-equivalent}, denoted $\xx \eqvl \zz$, if for every $\ww \in \St$ we have $\ww\xx\in\St$ if and only if $\ww\zz\in\St$. 

Similarly, define the dual notion of \emph{$H_r$-equivalence}, and write $\xx \eqvr \zz$.
\end{defn}

\begin{exmp}
In the string algebra $\Gamma_1$ from Figure \ref{Ex: Standard H-equiv}, we have $c \eqvl cAb$ but $ce \not\equiv_H^l cAb $.
    \begin{figure}[h]
\[\begin{tikzcd}[ampersand replacement=\&]
	\&\&v_1 \\
	v_2 \&\&\&\& v_4 \\
	\&\& v_3
	\arrow["e", from=2-1, to=1-3]
	\arrow["d"', from=2-1, to=3-3]
	\arrow["c", from=1-3, to=3-3]
	\arrow["b", from=3-3, to=2-5]
	\arrow["a", from=1-3, to=2-5]
\end{tikzcd}\]

\caption{ $\Gamma_1$ with $\rho = \{ae , bd, bce\}$}
\label{Ex: Standard H-equiv}
\end{figure}
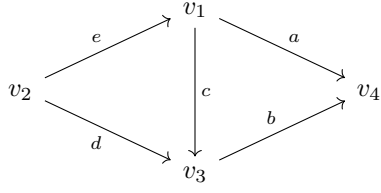
\end{exmp}

Clearly both $\eqvl$ and $\eqvr$ are equivalence relations on $\St$. For any $\xx,\yy \in \St$, if $\xx\eqvl\zz $ $\left( \text{resp. } \xx\eqvr\zz  \right)$ then $\uu\xx\mapsto\uu\zz $ $\left( \text{resp. } \xx\uu\mapsto\zz\uu \right)$ is an isomorphism between their left (resp. right) hammocks.

The next result sheds more light on the structure of two $H_l$-equivalent strings.
\begin{prop}\cite[Proposition~5.4]{sardar2021variations}\label{prop: H-equivalent substrings}
Suppose $\xx,\yy \in \St$ and $|\xx|>0$. If $\xx\yy\equiv_H^l \yy$ or $\yy\xx\eqvr\yy$ then $\xx = \bb^n$ for some $\bb \in \Cyc$ and $n \in \N^+$.
\end{prop}

The following definition provides a finite set of strings containing the set of band-free strings.
\begin{defn}\cite[\S~8]{sardar2021variations} 
Given $\xx_0 \in \St$, $i\in \{1,-1\}$ and $\uu\xx_0\in H_l^i(\xx_0)$, say that $\uu\xx_0$ is \emph{$H_l$-reduced relative to $(\xx_0, i)$} if the following conditions hold.
\begin{itemize}
    \item If $\uu=\uu_2 \vv \uu_1$ and $|\uu_1||\vv|>0$ then $\vv\uu_1\xx_0 \not\equiv_H^l \uu_1\xx_0$.
    \item If $\uu=\uu'_2\uu'_1$, $|\uu'_1|>0$ and $\uu'_1\xx_0\eqvl\xx_0$ then $\uu'_2\xx_0\in H_l^{-i}(\xx_0)$.
\end{itemize}
\end{defn}
Dually the definition of $H_r$-reduced relative to $(\xx_0,i)$ can be given for a string $\xx_0\uu\in H_r^i(\xx_0)$.

Here is the promised result on the finiteness of the set of $H_l$-reduced (resp. $H_r$-reduced) strings that will be one of the crucial ingredients in the estimation of the stable rank of $\Lambda$.

\begin{prop}\label{prop: H-reduced finitely many}
For any $\xx_0 \in \St$ and $i \in \{1,-1\}$, there are only finitely many strings $H_l$-reduced (resp. $H_r$-reduced) relative to $ (\xx_0, i)$.
\end{prop}
\begin{proof}
Let $M$ denote the maximum length of a direct string in $\Lambda$ and $N$ denote the number of strings of the form $aB$, where $a,b\in Q_1$. By the pigeonhole principle, if $\yy \in H_l^i(\xx_0)$ and $|\yy|>2(M+1)(N+1)+|\xx_0|$ then there are $a,b\in Q_1$ such that $\yy=\yy_3aB\yy_2aB\yy_1\xx_0$. Clearly $aB\yy_2aB\yy_1\xx_0\eqvl aB\yy_1\xx_0$, and $ \yy_3 aB\yy_1\xx_0 \in H_l^i(\xx_0) $. Thus if $\yy \in H_l^i(\xx_0)$ is an $H_l$-reduced string relative to $(\xx_0,i)$, then $|\yy|\leq 2(M+1)(N+1)+|\xx_0|$. Since there is only a finite set of strings of length bounded above by $2(M+1)(N+1) + |\xx_0|$, we have shown that the number of $H_l$-reduced strings relative to $(\xx_0,i)$ is finite.
\end{proof}

More generally, say that $\xx \in \St$ is \emph{$H_l$-reduced} (resp. $H_r$-reduced) if there do not exist $\uu_1,\vv,\uu_2 \in \St$ satisfying $\xx=\uu_1 \vv \uu_2$, $|\vv|>0$ and $\vv\uu_2 \equiv_H^l \uu_2$ (resp. $\uu_1\vv \eqvr \uu_1$). It can be shown that there are only finitely many $H_l$-reduced (resp. $H_r$-reduced) strings using a proof similar to the above proposition. As a consequence, we obtain the following finiteness result.

\begin{prop}\label{prop: finite H_l-equivalence class}
There are only finitely many $\eqvl$(resp. $\eqvr$)-equivalence classes of strings.
\end{prop}
\begin{proof}
We only show that $\St\, / \eqvl$ is a finite set; the finiteness of $\St\, / \eqvr$ follows similarly. In view of the above discussion, it is enough to show that for any $\xx \in \St$, there exists $\yy \in \St$ such that $\xx \eqvl \yy$ and $\yy$ is $H_l$-reduced using an induction on $|\xx|$.

\noindent{}\textbf{Base Case.} If $|\xx| = 0$, then $\xx$ is $H$-reduced, and hence the conclusion is trivial.

\noindent{}\textbf{Induction Step.} Suppose $|\xx| >0$. If $\xx$ is $H$-reduced then again the conclusion is trivial. Thus assume $\xx$ is not $H_l$-reduced. Take $\uu_1,\vv,\uu_2 \in \St$ such that $\xx=\uu_1\vv\uu_2, |\vv|>0$ and $\vv\uu_2 \eqvl \uu_2$. Then $\xx=\uu_1\vv\uu_2 \eqvl \uu_1\uu_2$. Since $|\uu_1\uu_2|<|\xx|$, the induction hypothesis yields $\yy \in \St$ such that $\yy$ is $H_l$-reduced and $\yy \eqvl \uu_1\uu_2 \eqvl \xx$.

Since there are only finitely many $H_l$-reduced strings, the above argument completes the proof that $\St / \eqvl$ is a finite set.
\end{proof}

Fix some $v\in Q_0$. Now we define a ``$2$-dimensional'' version of $H$-equivalence. Say that $(\xx_1 , \xx_2) , (\yy_1 , \yy_2) \in H(v)$ are \emph{$H_{lr}$-equivalent}, denoted $(\xx_1 , \xx_2) \equiv^{lr}_H (\yy_1 , \yy_2 )$, if $\xx_1\xx_2 \eqvl \yy_1\yy_2$ and $\xx_1\xx_2 \eqvr \yy_1\yy_2$. Thus if $(\xx_1 , \xx_2) \equiv^{lr}_H (\yy_1 , \yy_2 )$ then $ \left(H^i_l(\xx_1\xx_2),<_l\right) \cong \left( H^i_l(\yy_1\yy_2) , <_l \right) $ and $\left(H^i_r(\xx_1\xx_2),<_l\right) \cong \left( H^i_r(\yy_1 \yy_2 ) , <_l \right)$ via canonical maps for any $i \in \{1,-1\}$. Say that $(\xx_1 , \xx_2 )\in H(v )$ is \emph{$H_{lr}$-reduced} if there do not exist $\yy',\vv,\yy'' \in \St$ with $|\vv| >0$ such that at least one of the following conditions holds;
\begin{itemize}
    \item $\xx_1 = \yy'\vv\yy''$ and $(\yy'\yy'' , \xx_2) \eqvlr (\xx_1 , \xx_2)$,
    \item $\xx_2 = \yy'\vv\yy''$ and $(\xx_1 , \yy'\yy'') \eqvlr (\xx_1,\xx_2) $.
\end{itemize}
The following two results are natural analogues of Propositions \ref{prop: H-reduced finitely many} and \ref{prop: finite H_l-equivalence class} respectively, and can be proved with obvious modifications in the proof.
\begin{prop}\label{Hlrredfiniteness}
There are finitely many $H_{lr}$-reduced elements in $H(v)$.
\end{prop}

\begin{prop}\label{prop: Hlr reduced representative}
    For any $(\xx_1 , \xx_2 ) \in H(v )$ there exists an $H_{lr}$-reduced string $(\yy_1 , \yy_2 ) \in H(v)$ such that $(\xx_1 , \xx_2 ) \eqvlr (\yy_1 , \yy_2 )$. Therefore, there is only a finite number of $\eqvlr$-equivalence classes of elements in $H(v)$.
\end{prop}

\subsection{An algorithm to compute some infinite strings in hammock linear orders}\label{sec:alg msi lin ord}
The results of \cite[\S~9]{SKSK} imply that $\left(\OSt{l}{\xx_0 ,i;\sB},<_l\right)  $ is a discrete linear order; we begin this subsection by recalling the necessary details. The highlight of this rather technical subsection is Theorem \ref{prop: H-reduced infty bux} which states that the nearest infinite string of a given finite string $\xx$ that can be obtained via repeated applications of the successor operator enjoys some special properties, i.e., it is of the form $\prescript{\infty}{}{\bb}\uu\xx$, where both $\bb$ and $\uu$ belong to easily computable finite sets. Furthermore, we show (Proposition \ref{jkbound}) that the density of the interval $[\xx,\prescript{\infty}{}{\bb}\uu\xx)$ can also be effectively computed in finitely many steps. At the end we discuss how we can adapt to the non-domestic case the algorithm in \cite{Schroer98hammocksfor} to compute the density of hammock linear orders for domestic algebras. 

Following \cite[Definition~9.2]{SKSK} we define successor and predecessor functions on some appropriate subsets of $\OSt{l}{\xx_0,i;\sB}$. Note that for any $\xx_0 \in \St$, $i \in \{1,-1\}$ and $\xx \in \OSt{l,1}{\xx_0, i ; \sB} \setminus \{\MM_{l,i}(\xx_0)\}$, there exists a unique $\alpha \in Q_1^-$ such that $\alpha\xx \in \OSt{l}{\xx_0 }$. Define $\LB : \OSt{l,1}{\xx_0, i ; \sB} \setminus \{\MM_{l,i}(\xx_0)\}\to \OSt{l,1}{\xx_0, i ; \sB}$ by choosing $\LB(\xx)$ to be the longest string $\uu\alpha\xx\in\OSt{l,1}{\xx_0,i;\sB}$ satisfying $\delta(\uu)=-1$ whenever $|\uu|>0$. Remark \ref{rem: reversing the sign of STB} ensures that $\LB(\xx)$ is well-defined.

Similarly (consistent with the notations of \cite[\S~2.4]{GKS20}) we define  
\begin{align*}
    &\LbB : \OSt{l,-1}{\xx_0 , i ; \sB} \setminus \{ \mm_{l,i} (\xx_0) \} \to \OSt{l,-1}{\xx_0 , i ; \sB} \text{ for } \sB \in \QBa_{l,i}(\xx_0);\\
    &r_{\sB} : \OSt{r,-1}{\xx_0 , i ; \sB} \setminus \{ \MM_{r,i} (\xx_0) \} \to \OSt{r,-1}{\xx_0 , i ; \sB} \text{ for }\sB \in \QBa_{r,i} (\xx_0);\\
    &\rb_{\sB} : \OSt{r,1}{\xx_0 , i ; \sB} \setminus \{ \mm_{r,i} (\xx_0) \} \to \OSt{r,1}{\xx_0 , i ; \sB} \text{ for }  \sB \in \QBa_{r,i}(\xx_0).
\end{align*}

Note that the function  $\LB$ is defined for a given $\sB \in \QBa$ and a hammock $H_l^i(\xx_0)$. In this paper, for $\xx \in \OSt{l,1}{\sB}$ whenever we write $\LB (\xx)$, then the hammock should be inferred from the context. Since the set $\OSt{l,1}{\xx,i;\sB}$ is the same as $\OSt{1}{\xx,i;\sB}$ in \cite[\S~6]{SKSK}, the following result is a restatement of \cite[Proposition~9.4]{SKSK}.

\begin{prop}\label{prop: lb not in St-1}
    For any $\xx_0 \in \St$ and $i \in \{1,-1\}$, if $\xx \in \OSt{l,1}{\xx_0,i;\sB}\setminus \{\MM_i(\xx_0)\}$ then $ \xx <_l \LB(\xx), \LB(\xx) \notin \OSt{l,-1}{\sB}$ and $H_l^{-1}(\LB(\xx)) = \left( \xx , \LB(\xx) \right]_l$.
\end{prop}

Since $\LB(\xx) \notin \OSt{l,-1}{\sB}$ for any $\xx \in \OSt{l,1}{\xx_0,i;\sB} \setminus \{\MM_{l,i}(\xx_0)\}$ and $\sB \in \QBa_{l,i}(\xx_0)$, we get that $\sB \notin \QBa_{l,-1}(\LB (\xx))$. Hence by Proposition \ref{prop: lb not in St-1} we can make the following observation.

\begin{rem}\label{rem: lb is successor}
    For any $\xx_0 \in \St$, $i \in \{1,-1\}$ and $\sB \in \QBa_{l,i}(\xx_0)$, if $\xx \in \OSt{l,1}{\xx_0, i ; \sB} $ then $$ [\xx, \LB(\xx)]_l \cap \OSt{l}{\xx_0 , i; \sB} = \{\xx\} \cup \left(H_l^{-1}(\LB(\xx)) \cap  \OSt{l}{\xx_0 , i; \sB}\right) = \{\xx, \LB(\xx)\}.$$
\end{rem}

For any $\xx \in \OSt{l,1}{\xx_0, i ; \sB}$ we define powers of $\LB$ by induction: first set $\LB^0 (\xx) := \xx$ and for every $n\in\N$ if $\LB^n (\xx)$ exists and $\LB^n(\xx) \neq \MM_{l,i}(\xx_0)$ then define $\LB^{n+1} (\xx) := \LB \left( \LB^n(\xx) \right)$. Note that for $\xx,\yy\in \OSt{l,1}{\xx_0 , i; \sB}$ if there exists $\alpha \in \Q_1^-$ satisfying $\alpha\xx,\alpha\yy \in \OSt{l,1}{\xx_0 , i; \sB}$ then there exists a string $\vv_1$ such that $|\vv_1|>0$, $\LB(\xx) = \vv_1\xx$ and $\LB(\yy) = \vv_1\yy$. In fact this argument can be used inductively to obtain the following.

\begin{rem}\label{rem: l on strings with same target}
    For any $\xx, \yy \in \OSt{l,1}{\xx_0,i;\sB}$, $n \in \N$ and $\alpha\in Q_1^-$ satisfying $\alpha\xx,\alpha\yy \in \OSt{l}{\xx_0 , i; \sB} $,  if $\LB^n(\xx)$ exists and  $\LB^n(\xx) = \uu\xx$ then $\LB^n(\yy) = \uu\yy$.
\end{rem}

If $\LB^n (\xx)$ exists for every $n \in \N$, then define a left $\N$-string $\brac{1}{\LB} (\xx) := \lim^l_{n \to \infty} \LB^n (\xx)$. Note that for any $\xx \in \OSt{l,1}{\xx_0 , i ; \sB}$, $\brac{1}{\LB}(\xx)$ exists if and only if $\xx \in \OSt{l,1}{\sB}\setminus \{\MM_{l,i} (\xx_0)\}$.

\begin{rem}\label{rem: substrings of <1,lb>}
    For any $\xx \in \OSt{l,1}{\sB}$ and $\yy\in \St$ such that $\xx \sqsubset_l \yy\sqsubset_l \brac{1}{\LB}(\xx)$, $ \theta_l( \brac{1}{\LB}(\xx) \mid \yy) =1 $ if and only if there exists $n \in \N$ such that $\LB^n(\xx) =\yy$.
\end{rem}

\begin{prop}\label{prop: <1,LB> is an order preserving map}
    If $\sB \in \QBa$ and $\xx \in \OSt{l,1}{\xx_0,i;\sB}\cap \OSt{l,1}{\sB} \setminus \{ \MM_{l,i}(\xx_0) \}$, then $[ \xx,\brac{1}{\LB} (\xx))_l \cap \OSt{l}{\sB} = \{\LB^n(\xx)\mid n \in \N \}$.
\end{prop}

\begin{proof}
    Since $\LB^n(\xx) <_l \LB^{n+1}(\xx)$ for every $n \in \N$ by Proposition \ref{prop: lb not in St-1}, we get that
    \begin{alignat*}{3}
    [ \xx,\brac{1}{\LB} (\xx))_l &= \bigcup_{n \in \N} [\LB^n(\xx), \LB^{n+1}(\xx))_l.& \\
       \therefore [ \xx,\brac{1}{\LB} (\xx))_l \cap \OSt{l}{\sB} &= \bigcup_{n \in \N} [ \LB^n(\xx), \LB^{n+1}(\xx) )_l \cap \OSt{l}{\sB}&\\
        &= \bigcup_{n \in \N} \{\LB^n(\xx)\} = \{ \LB^n(\xx): n \in \N \}&\hspace{25mm}\text{ by Remark \ref{rem: lb is successor}.}
    \end{alignat*}
\end{proof}

The following results state that any infinite string obtained by an application of the $ \brac{1}{\LB} $ operator can in fact be determined using the knowledge of a pair of distinct sufficiently long $ H_l $-equivalent left substrings.
\begin{prop}\label{prop: H-equialance in <1,lb>}
Suppose $\xx_0,\yy,\bb\in \St, i \in \{1,-1\},\sB \in \QBa$ and $\bb\yy\eqvl\yy$ so that $\prescript{\infty}{}{\bb}\yy$ is a left $\N$-string in view of Proposition \ref{prop: H-equivalent substrings}.
\begin{itemize}
    \item If $\xx \in \OSt{l,1}{\xx_0,i;\sB}\cap \OSt{l,1}{\sB} \setminus \{ \MM_{l,i}(\xx_0)\}$ and $\xx \sqsubset_l \yy \sqsubset_l \bb\yy \sqsubset_l \brac{1}{\LB} (\xx)$ then $\brac{1}{\LB} (\xx) = \prescript{\infty}{}{\bb}\yy$.
    \item If $\xx \in \OSt{l,-1}{\xx_0,i;\sB}\cap \OSt{l,-1}{\sB} \setminus\{\mm_{l,i}(\xx_0)\}$ and $\xx \sqsubset_l \yy \sqsubset_l \bb\yy \sqsubset_l \brac{1}{\LbB} (\xx)$ then $\brac{1}{\LbB} (\xx) = \prescript{\infty}{}{\bb}\yy$.
\end{itemize}
\end{prop}

\begin{proof} Since there exists $m \in \N$ such that $\bb\yy \sqsubset_l \LB^m(\xx)$ and $\LB^m(\xx) \in \OSt{l,1}{\sB}$, we get that $\bb\yy \in \OSt{l}{\sB}$.
    
Towards a contradiction, assume $\brac{1}{\LB} (\xx) \neq \prescript{\infty}{}{\bb}\yy$. Take minimal $n > 1$ such that $\bb^n\yy \not\sqsubset_l \brac{1}{\LB}(\xx)$. Since $\bb\yy \eqvl \yy$, using induction we can show that $ \bb^n\yy \eqvl \bb\yy\eqvl  \yy $. Hence $\bb\yy\in \OSt{l}{\sB}$ implies $\bb^n\yy \in \OSt{l}{\sB}$. Note that $\xx <_l \bb^n\yy$. We have two cases.

\begin{itemize}
\item If $\xx <_l \prescript{\infty}{}{\bb}\yy <_l \brac{1}{\LB}(\xx) $ then since $\bb^n\yy \sqcap_l \brac{1}{\LB}(\xx) \sqsubset_l \bb^n\yy$,  Remark \ref{rem: Hammock order property} yields $\bb^n\yy <_l \brac{1}{\LB}(\xx)$. Since $\xx <_l \bb^n\yy <_l \brac{1}{\LB}(\xx)$, Proposition \ref{prop: <1,LB> is an order preserving map} yields $k \in \N$ such that $\LB^k(\xx) = \bb^n\yy$, which is a contradiction to $\bb^n\yy \not\sqsubset_l \brac{1}{\LB}(\xx)$.

\item If $ \xx <_l \brac{1}{\LB}(\xx) <_l \prescript{\infty}{}{\bb}\yy $, then choose a partition $\bb = \uu''\uu'$ such that $\uu'\bb^{n-1}\xx = \brac{1}{\LB}(\xx) \sqcap_l \prescript{\infty}{}{\bb}\yy$. Since $\uu''\uu'\bb^{n-1}\xx= \bb^n\xx \not\sqsubset_l \brac{1}{\LB}(\xx)$, we get that $|\uu''|>0$. Since $ \brac{1}{\LB}(\xx) <_l \prescript{\infty}{}{\bb}\yy $, we get that $\theta_l(\uu'') = 1$. Take $k \in \N$ such that $\uu'\bb^{n-1}\yy\sqsubset_l \LB^k(\xx) $. Since $ \brac{1}{\LB}(\xx) <_l \prescript{\infty}{}{\bb}\yy $, $\theta_l (\LB^k(\xx)\mid \uu'\bb^{n-1}\yy) =- \theta_l(\prescript{\infty}{}{\bb}\yy\mid  \uu'\bb^{n-1}\yy) = -1$. Since $\LB^k(\xx) \in \OSt{l,1}{\sB}$, we get that $ \uu'\bb^{n-1}\yy \in\OSt{l,-1}{\sB} $. Since $ \bb^n\yy \eqvl \yy $, we get that $\uu'\bb^n\yy \eqvl \uu'\yy $. Hence $\uu'\yy \in \OSt{l,-1}{\sB}$. Since $\bb\yy\sqsubset_l\brac{1}{\LB}(\xx)$ and $\theta_l ( \brac{1}{\LB}(\xx)\mid \uu'\yy ) = \theta_l (\bb\yy\mid \uu'\yy) = \theta_l (\uu'') = 1$, Remark \ref{rem: substrings of <1,lb>} yields $p >0$ such that $\LB^p(\xx) = \uu'\yy \in \OSt{l,-1}{\sB}$, which is a contradiction to Proposition \ref{prop: lb not in St-1}.
\end{itemize}
\end{proof}

\begin{cor}\label{prop: <1,lb> contains a band}
    For any $\xx_0 \in \St, i \in \{1,-1\}, \sB \in \QBa_{l,i}(\xx_0)$ and $\xx \in \OSt{l,1}{\xx_0 , i ; \sB} \cap \OSt{l,1}{\sB} \setminus \{ \MM_{l,i}(\xx_0) \}$, there exists $\bb \in \QBa_0$ and $\uu \in \St$ such that $\brac{1}{\LB} (\xx) = \prescript{\infty}{}{\bb}\uu\xx$.
\end{cor}

\begin{proof}
Using the pigeonhole principle and Proposition \ref{prop: finite H_l-equivalence class}, we get the existence of $p<q$ in $\N$ such that $\LB^p(\xx) \eqvl \LB^q(\xx)$. Let $\LB^p(\xx)= \uu'\xx$ and $\LB^q(\xx)= \zz\uu'\xx$. Since $\zz\uu'\xx \eqvl \uu'\xx$, by Proposition \ref{prop: H-equialance in <1,lb>} we get that $\brac{1}{\LB}(\xx) = \prescript{\infty}{}{\zz}\uu'\xx$. Proposition \ref{prop: H-equivalent substrings} yields $\bb \in \QBa_0$ such that $\zz$ is a finite power of a cyclic permutation of $\bb$. Therefore $\brac{1}{\LB}(\xx) = \prescript{\infty}{}{\zz}\uu'\xx = \prescript{\infty}{}{\bb}\uu\xx $ for some $\uu \sqsubseteq_l \prescript{\infty}{}{\zz}\uu'$.
\end{proof}

There are only a finite number of bands $\bb \in \QBa_0$ satisfying the conclusions of the above corollary as the following result describes the left $\N$-strings of the form $\brac{1}{\LB}(\xx)$ using $H_l$-reduced strings and bands from the finite set $\bigcup_{\sB\in\QBa}\BalB$.

\begin{thm}\label{prop: H-reduced infty bux}
Suppose $\xx_0 \in \St$, $i \in \{-1,1\}$, $\sB\in\QBa$ and $\xx \in \OSt{l,1}{\xx_0,i;\sB} \cap \OSt{l,1}{\sB} \setminus \{\MM_{l,i}(\xx_0)\}$. Then there exist $\ol\uu\in\St$, $\sB'\in \QBa$, $\bb\in \mathsf{Ba}_l(\sB')$ and a cyclic permutation $\ol\bb$ of $\bb$ such that $\brac{1}{\LB}(\xx) = \prescript{\infty}{}{\ol\bb}\ol\uu\xx$. Furthermore, if $|\ol\uu|>0$ then $\theta_r(\ol\bb)=-\theta_r(\ol\uu)$, and 
\begin{enumerate}
    \item if $\ol\vv\sqsubset_l\ol\uu$ then $\ol\vv\xx$ is $H_l$-reduced relative to $(\xx,1)$;
    \item if $\ol\uu\xx$ is not $H_l$-reduced relative to $(\xx,1)$ then $\theta_l(\ol\bb)=-1$.
\end{enumerate}
\end{thm}
\begin{proof}
Corollary \ref{prop: <1,lb> contains a band} applied to the hypotheses yields $\bb \in \overline{\QBa_{l,1}}(\xx_0)$ and $\uu \in \St$ with minimal $|\uu|$ such that $ \brac{1}{\LB}(\xx) = \prescript{\infty}{}{\bb}\uu\xx $. Clearly $\theta_l(\bb\uu)=1$ and $\bb\uu\xx\in \OSt{l}{\sB}$. Take $\sB' \in \QBa$ such that $\bb \in \sB'$.

Now we show that $\bb \in \mathsf{Ba}_l(\sB')$. Towards a contradiction, assume $\bb \notin \mathsf{Ba}_l(\sB')$. Then there are $\bb'\in \sB'$ and $\uu'\in \St$ such that $\bb'\uu'\bb\in\St$ and $\theta_l(\prescript{\infty}{}{\bb'}\uu'\bb \mid \prescript{\infty}{}{\bb}) = -1$ so that $ \xx <_l \prescript{\infty}{}{\bb'}\uu'\bb\uu\xx <_l \prescript{\infty}{}{\bb}\uu\xx $. Choose $n \in\N$ such that $\LB^n(\xx) <_l \prescript{\infty}{}{\bb'}\uu'\bb\uu\xx<_l \LB^{n+1}(\xx) $. Take $k \in \N$ such that $ \left|\LB^{n+1}(\xx)\right| < |(\bb')^k\uu'\bb\uu\xx| $ so that $ \LB^{n+1}(\xx) \sqcap_l (\bb')^k\uu'\bb\uu\xx \sqsubset_l  (\bb')^k\uu'\bb\uu\xx$. Thus Remark \ref{rem: Hammock order property} yields $ \LB^n(\xx) <_l(\bb')^k\uu'\bb\uu\xx <_l \LB^{n+1}(\xx) $. Since $\bb'\approx\bb$, there exists $\widehat{\uu}$ such that $\bb\widehat{\uu}\bb' \in \St$. Combining this with $\bb\uu\xx \in \OSt{l}{\sB}$ we obtain that  $(\bb')^k\uu'\bb\uu\xx \in \OSt{l}{\sB}$. This is a contradiction to Remark \ref{rem: lb is successor} as $(\bb')^k\uu'\bb\uu\xx \in (\LB^n(\xx), \LB^{n+1}(\xx))_l $. 

By reducing the length of $\uu$ if necessary, we can find a minimal length $\ol\uu\sqsubseteq_l\uu$ and a cyclic permutation $\ol\bb$ of $\bb$ satisfying $\prescript{\infty}{}{\bb}\uu\xx=\prescript{\infty}{}{\ol\bb}\ol\uu\xx$, and $\theta_r(\ol\uu)=\theta_r(\ol\bb)$ if $|\ol\uu|>0$.

Suppose $|\ol\uu|>0$. If $(1)$ fails then there is $\ol\vv\sqsubset_l\ol\uu$ such that $\ol\vv\xx$ is not $H_l$-reduced relative to $(\xx,1)$. There are two cases.
\begin{itemize}
    \item If there is a partition $\ol\uu=\uu_2\vv\uu_1$ with $|\uu_1||\vv|>0$ such that $\vv\uu_1\xx\equiv_H^l \uu_1\xx$ then Proposition \ref{prop: H-equialance in <1,lb>} gives that $ \brac{1}{\LB} (\xx) = \prescript{\infty}{}{\vv}\uu_1\xx $. Proposition \ref{prop: H-equivalent substrings} applied to $\vv\uu_1\xx \eqvl \uu_1\xx$ yields that $\vv$ is a finite power of a cyclic permutation of a band. Choose a decomposition $\vv = \uu''\uu', n \in \N^+$ and $\bb' \in \overline{\QBa_{l,1}}(\xx_0)$ such that $(\bb')^n = \uu'\uu''$ and $|\uu' |< |\bb'|\leq|\vv|$. Thus  $\prescript{\infty}{}{\vv}\uu_1\xx = \prescript{\infty}{}{\bb' }\uu'\uu_1\xx$. Since $|\uu'\uu_1|<|\vv\uu_1|\leq|\ol\uu|$ we obtain a contradiction to the minimality of $|\ol\uu|$.
    \item If there is a partition $\ol\uu=\uu'_2\uu'_1$ with $|\uu'_2||\uu'_1|>0$ and $\theta_l(\uu'_2)=1$ satisfying $\uu'_1\xx\eqvl\xx$ then $\theta_l(\uu'_1)=\theta_l(\uu'_2)=1$ implies that the first syllables of $\uu'_1$ and $\uu'_2$ are the same, say $\alpha\in Q_{-1}$. Let $\uu'_1=\ww\alpha$ and $\vv:=\alpha\ww$. Since $\uu'_1\xx\eqvl\xx$ we conclude that $\delta(\vv)=\delta(\uu'_1)=0$, and hence $|\vv|=|\uu'_1|\geq2$. As a consequence, $\xx\sqsubset_l\alpha\xx\sqsubset_l\uu'_1\xx\sqsubset_l\alpha\uu'_1\xx=\vv\alpha\xx\sqsubset_l\brac{1}{\LB}(\xx)$. Moreover, $\uu'_1\xx\eqvl\xx$ implies $\vv\alpha\xx\eqvl\alpha\xx$. Therefore, Proposition \ref{prop: H-equialance in <1,lb>} gives $ \brac{1}{\LB} (\xx) = \prescript{\infty}{}{\vv}\alpha\xx$. Now the proof of the above case can be followed, where Proposition \ref{prop: H-equivalent substrings} is applied to $\vv\alpha\xx\eqvl\alpha\xx$, to eventually obtain a contradiction to the minimality of $|\ol\uu|$.
\end{itemize}

Now that we have shown that $(1)$ holds, let us show that $(2)$ holds. Suppose $\ol\uu\xx$ is not $H_l$-reduced relative to $(\xx,1)$. Using $(1)$ we conclude that $\ol\uu\xx\eqvl\xx$. Then both $\bar\uu^2\xx$ and $\ol\bb\ol\uu\xx$ are strings, and hence the minimality of $|\ol\uu|$ ensures that $\theta_l(\ol\bb)=-\theta(\ol\uu)=-1$. Thus $(2)$ indeed holds.
\end{proof}

For the rest of this section, choose and fix $\xx_0\in\St, i \in \{-1,1\}$, $\sB\in\QBa$ and $\zz_0\in \OSt{l,1}{\xx_0,i;\sB} \cap \OSt{l,1}{\sB} \setminus \{\MM_{l,i}(\xx_0)\}$. Let $\yy_0:=\brac1\LB(\zz_0)=\prescript{\infty}{}{\bb}\uu\zz_0$ be as obtained in the above theorem. Then
\begin{equation}\label{hammocksimp2}
[\zz_0, \yy_0)_l\cong \{\zz_0\} + \sum_{n \in \omega} (\LB^n(\zz_0), \LB^{n+1}(\zz_0)]_l\cong \{\zz_0\} + \sum_{n \in \omega} H_l^{-1}(\LB^{n+1}(\zz_0))\cong\sum_{n \in \omega} H_l^{-1}(\LB^{n+1}(\zz_0)),
\end{equation}
where the second isomorphism follows from \cite[Proposition~9.3]{SKSK} and the third from the fact that the sum is infinite, discrete and bounded below. Since there are finitely many $H_l$-equivalence classes by Proposition \ref{prop: finite H_l-equivalence class}, choose the least $j>0$ such that $\LB^j(\zz_0)\equiv^l_H \LB^{j+k}(\zz_0)$ for some $k>0$; also assume that $k>0$ is the least such for the chosen $j$.

Let $\ol\xx:=\LB^j(\zz_0)$. Then Proposition \ref{prop: H-equivalent substrings} yields $\ol\bb\in\Cyc$ and $k'\geq 1$ such that $\LB^{j+k}(\zz_0)= \overline{\bb}^{k'}\overline{\xx}$. Since for any $\bb' \in \Cyc$, $\xx' \in \St$ and $n \in \N$ we have $(\bb')^n\xx'\eqvl\bb'\xx'$ whenever $\bb'\xx' \in \St$, the minimality of $k$ yields $k'=1$. Then Proposition \ref{prop: H-equialance in <1,lb>} gives that $\prescript{\infty}{}{\overline{\bb}}\overline{\xx} = \brac{1}{\LB} (\zz_0)=\yy_0 = \prescript{\infty}{}{\bb}\uu\zz_0$. Thus $\overline{\bb}$ is a cyclic permutation of $\bb$.

The next result shows that $j,k$ can be effectively computed.
\begin{prop}\label{jkbound}
The above chosen values of $j,k$ satisfy $j+k \leq 2|Q_1|$.    
\end{prop}

\begin{proof}
Towards a contradiction assume $j+k > 2| Q_1|$. Note that for every $\xx \in \{\LB^n(\zz_0 ) : n \in \N\}$ there exists $\alpha\in Q_1^-$ such that $\alpha\xx \in \OSt{l}{\sB}$. Then the pigeonhole principle yields $\alpha_1 \in Q_1^-$ and $0\leq n_1<n_2<n_3\leq j+k-1$ such that $\alpha_1\left(\LB^{n_1}(\zz_0)\right), \alpha_1\left(\LB^{n_2}(\zz_0)\right), \alpha_1\left(\LB^{n_3}(\zz_0)\right) \in \OSt{l}{\sB}$. Take $\uu_1,\uu_2, \xx \in \St$ such that $ \LB^{n_1} (\zz_0) = \xx, \LB^{n_2}(\zz_0) = \uu_1\xx$ and $\LB^{n_3}(\zz_0) = \uu_2\uu_1\xx$. There are two cases.

  \begin{itemize}
      \item If $n_2 - n_1 \leq n_3-n_2$, then $2n_2 - n_1 \leq n_3 \leq j+k-1$. Since $\alpha_1\xx , \alpha_1\uu_1\xx \in \OSt{l}{\sB}$ and $\LB^{n_2 - n_1}\left( \xx \right) = \uu_1\xx$, Remark \ref{rem: l on strings with same target} gives us that $\LB^{2n_2 - n_1}\left(\zz_0\right) = \LB^{n_2 - n_1}\left( \uu_1\xx \right) = \uu_1^2\xx$. A simple induction yields $ \uu_1^{m+1}\xx = \LB^{m(n_2-n_1)}(\uu_1\xx)$ for every $m>0$, and hence $\prescript{\infty}{}{\uu_1}\xx = \brac{1}{\LB}(\zz_0) = \prescript{\infty}{}{\overline{\bb}}\overline{\xx}$. Since $\overline{\bb} \in \Cyc$, we get that $\uu_1$ is a finite power of a cyclic permutation of $\overline{\bb}$ and hence $|\overline{\bb}| \leq |\uu_1|$. Combining this with $|\uu_1^2\xx|=|\LB^{2n_2-n_1}(\zz_0)| < |\LB^{j+k}(\zz_0)| = |\overline{\bb}\overline{\xx}|$ we get that $|\uu_1\xx| <|\overline{\xx}|$. Thus $0<n_2 < j$. Since $\LB^{n_2 }(\zz_0) = \uu_1\xx \equiv_H^l \uu_1^2\xx = \LB^{2n_2-n_1}(\zz_0)$, we obtain a contradiction to the minimality of $j$.

      \item If $n_3 - n_2 < n_2 - n_1$, then let $t_1 := n_3 - n_2$. Since $\alpha_1\uu_1\xx, \alpha_1\uu_2\uu_1\xx \in \OSt{l}{\sB}$ and $\LB^{t_1}(\uu_1\xx) = \LB^{n_3}(\zz_0) = \uu_2\uu_1\xx$, Remark \ref{rem: l on strings with same target} yields $\LB^{2t_1}\left(\uu_1\xx\right) = \LB^{t_1}(\uu_2\uu_1\xx) = \uu_2^2\uu_1\xx$. A simple induction yields $\uu_2^m\uu_1\xx =\LB^{mt_1}(\uu_1\xx)$ for every $m >0$. Since $\alpha_1\uu_1\xx, \alpha_1\xx \in \OSt{l}{\sB}$, Remark \ref{rem: l on strings with same target} gives $\LB^{mt_1}(\xx) = \uu_2^m\xx$ for every $m>0$. Thus $\prescript{\infty}{}{\uu_2}\xx =\brac{1}{\LB}(\xx) = \brac{1}{\LB}(\zz_0 ) = \prescript{\infty}{}{\overline{\bb}}\overline{\xx} $. Since $\overline{\bb}\in \Cyc$, we get that $\uu_2$ is a finite power of cyclic permutation of $\overline{\bb}$ and hence $ |\overline{\bb}|\leq |\uu_2| $. Since $\uu_2^2\xx = \LB^{2t_1}(\xx) = \LB^{2t_1+n_1}(\zz_0) \sqsubset_l \LB^{j+k}(\zz_0) = \overline{\bb}\overline{\xx}$, we get that $\LB^{t_1+n_1}(\zz_0) = \uu_2\xx \sqsubset_l \overline{\xx}  =\LB^{j}(\zz_0)$. Thus $0<t_1+n_1 < j $. Since $\LB^{t_1 + n_1} (\zz_0) = \uu_2\xx \equiv_H^l \uu_2^2\xx = \LB^{2t_1 + n_1} (\zz_0) $, we obtain a contradiction to the minimality of $j$.
  \end{itemize}
\end{proof}

Recall from \S~\ref{sec:H-eq} that if $\zz_1\eqvl\zz_2$ then $\ww\zz_1\mapsto\ww\zz_2$ defines an order isomorphism $H_l(\zz_1)\to H_l(\zz_2)$. Combining this fact with an argument involving Remark \ref{rem: l on strings with same target} as in the proof of the above proposition, we can show that $H_l^{-1}(\LB^{j+n}(\zz_0))\cong H_l^{-1}(\LB^{j+n+mk}(\zz_0))$ for any $n,m\in\N$. Thus Equation \eqref{hammocksimp2} further simplifies as
\begin{equation}\label{hammocksimp3}
[\zz_0,\yy_0)_l \cong \sum_{n=1}^{j-1} H_l^{-1}(\LB^n(\zz_0)) + \left( \sum_{n=0}^{k-1} H_l^{-1}(\LB^{j+n}(\zz_0))\right) \cdot \omega.
\end{equation}

At the end of this section, we recall an algorithm by Schr\"oer to compute the density of a hammock linear order for a domestic string algebra from \cite[\S~4]{Schroer98hammocksfor} and describe how to modify it for scattered half-hammocks for all string algebras.

Let $\zz \in \St$ and $j \in \{1,-1\}$. Theorems \ref{prop: Scattered hammock} and \ref{rank infinity iff not scattered} together imply that if the set $\QBa_{l,j}(\zz)$ contains a non-domestic element then $d(H_l^j(\zz))=\infty$. On the other hand, when $\QBa_{l,j}(\zz)$  consists exclusively of domestic elements, i.e., when $\overline{\QBa_{l,j}}(\zz)$ is a finite set, the algorithm developed by Schr\"oer \cite[\S~4.10]{Schroer98hammocksfor} to compute $d(H_l^j(\zz))$ becomes applicable. To see this, recall that if $\Lambda$ is domestic then the set $\QBa_0$ is finite. However, for a domestic $\Lambda$ and a zero-length string $\zz$, the proof of the theorem in \cite[\S~4.10]{Schroer98hammocksfor} uses only the finiteness of $\overline{\QBa_{l}}(\zz)$ to compute $d(H_l(\zz))$, instead of the full strength of the finiteness of $\QBa_0$. Thus for any string algebra $\Lambda$, if $\overline{\QBa_{l}}(\zz)$ is finite and non-empty then the same proof can be utilised to compute $d(H_l(\zz))$.

This algorithm employs a finite combinatorial gadget called the \emph{bridge quiver}--the details of its construction can be found in \cite[\S~4]{Schroer98hammocksfor}. Furthermore, if $\overline{\QBa_{l,j}}(\zz)$ is finite and non-empty, an appropriate subquiver of the bridge quiver that generates strings in $H_l^j(\zz)$, in the sense of \cite[\S~4.9]{Schroer98hammocksfor}, can be utilized to adapt the same algorithm to compute $d(H_l^j(\zz))$. This adaptation enables us to make the following conclusion from the proof of the theorem in \cite[\S~4.10]{Schroer98hammocksfor}.
\begin{prop}\label{subbridgequiver}
For $\zz \in \St$ and $j \in \{1,-1\}$, if $H_l^j(\zz)$ is scattered and infinite then $ \omega \cdot n \leq d \left( H_l^j(\zz) \right) < \omega \cdot (n +1)$, where $n$ is the maximal length of a path in the appropriate subquiver of the bridge quiver reachable from $\zz$.
\end{prop}
The conclusion of the result above also follows from an alternative detailed algorithm described in \cite{SardarKuberHamforDom} using a variation of the bridge quiver called the \emph{arch bridge quiver}.

Propositions \ref{jkbound} and \ref{subbridgequiver} ensure that if $[\zz_0,\yy_0)_l$ is scattered then $\{ d\left( H_l^{-1}(\LB^n(\zz_0)) \right) \mid n \in [j+k-1] \}$ from Equation \eqref{hammocksimp3} can be effectively computed.

\subsection{An algorithm to compute the density of a maximal scattered box}\label{sec: Density of a box} 
In this subsection, we give an algorithm to compute the density of an MSB. Each MSB $\ii$ contains a special element, which we call a \emph{pivot}, say $(\xx_1,\xx_2)$ (Proposition \ref{boxpivot}). This pivot enjoys a special property that the maximal scattered intervals in left and right hammocks of the string $\xx_1\xx_2$ containing $\xx_1\xx_2$ are isomorphic to $\pi_l(\ii)$ and $\pi_r(\ii)$ respectively. The density of $\ii$ can be computed using the density of finitely many intervals of $\pi_l(\ii)$ and $\pi_r(\ii)$ (Proposition \ref{prop: density of a box}). The boundaries of such intervals can be effectively computed (Propositions \ref{boundaryprojMSB} and \ref{prop: boundary of the MSI}). Using the computation of the densities of MSBs, we show that the stable rank $\st(\Lambda)$ can be effectively computed up to a finite error term (Equation \eqref{stablerank8}). We also show a detailed computation of the stable rank of a string algebra (Example \ref{ex: Proof of theorem 2}), and use it to prove Theorem \ref{reversest}.

For any $v \in Q_0 $, recall that if $\ii \subseteq H(v)$ is a box then $\ii = \pi_l(\ii) \otimes \pi_r(\ii) $, where $\pi_l(\ii)$ and $\pi_r(\ii)$ are intervals in $H_l(1_{(v,1)})$ and $H_r(1_{(v,1)})$ respectively. First, we look at an extension of Proposition \ref{unique string with minimal length in an interval} for boxes in hammocks. 

\begin{prop}\label{boxpivot}
For any $v \in Q_0$, if $\ii \subseteq H(v)$ is a non-empty box then there exists a unique $(\xx_1 , \xx_2 )\in \ii$ such that  for any $(\yy_1 , \yy_2) \in \ii$ we have $\xx_1 \sqsubseteq_l \yy_1$ and $\xx_2 \sqsubseteq_r \yy_2$.
\end{prop}
\begin{proof}
Since $\pi_l(\ii)$ and $\pi_r(\ii)$ are non-empty intervals in $H_l(1_{(v,1)})$ and $H_r(1_{(v,1)})$, Proposition \ref{unique string with minimal length in an interval} yields $\xx_1 \in \pi_l(\ii)$ and $\xx_2 \in \pi_r(\ii)$ such that $\pi_l(\ii) \subseteq H_l(\xx_1)$ and $\pi_r (\ii) \subseteq H_r(\xx_2)$. Since $\xx_1 \in \pi_l(\ii)$, take $\yy \in \pi_r (\ii)$ such that $(\xx_1, \yy) \in \ii$. Then $\xx_1\yy\in\St$. Thus using $\xx_2\sqsubseteq_l\yy$ we conclude $(\xx_1, \xx_2)\in \pi_l(\ii) \otimes \pi_r(\ii)=\ii$.
\end{proof}
We will refer to the pair $(\xx_1,\xx_2)$ obtained in the above proposition as the \emph{pivot} of the box. The above proof also shows the following.

\begin{rem}
The pivots of two distinct MSBs are different.
\end{rem}

For the rest of this section, fix $v \in Q_0 $ and an MSB $\ii \subseteq H(v)$ with pivot $ (\xx_1 , \xx_2)$. We will give an algorithm to compute the density of $\ii$. For simplicity, we deal with half intervals in $\pi_l (\ii)$ and $\pi_r (\ii)$ that are defined for $i\in\{1,-1\}$ as
$$\pi_{l,i} (\ii) := \pi_l (\ii) \cap H_l^i(\xx_1), \quad \pi_{r,i} (\ii) := \pi_r (\ii) \cap H_r^i(\xx_2).$$
Note that $ \pi_{l,1}(\ii), \pi_{r,-1} (\ii) \in \dLOfpb{1}{0} \cup \dLOfpb{1}{1} $ and $\pi_{l,-1}(\ii), \pi_{r,1} (\ii) \in \dLOfpb{0}{1} \cup \dLOfpb{1}{1} $. Thus Proposition \ref{period} yields $\ii^{(l,1)}_{0}, \ii^{(l,1)}_+, \ii^{(r,1)}_0, \ii^{(r,1)}_-, \ii^{(l,-1)}_0, \ii^{(l,-1)}_-, \ii^{(r,-1)}_0, \ii^{(r,-1)}_+ \in \dLOfpb{1}{1}$ such that 
\begin{align}\label{lefthammockbreakdown}
    \left( \pi_{l,1}(\ii), <_l \right) \cong \ii^{(l,1)}_{0} +  \ii^{(l,1)}_+ \cdot \omega, &\qquad \left( \pi_{r,1}(\ii), <_r \right) \cong  \ii^{(r,1)}_- \cdot \omega^* +\ii^{(r,1)}_0 ,\\\label{righthammockbreakdown}
    \left( \pi_{l,-1}(\ii), <_l \right) \cong \ii^{(l,-1)}_- \cdot \omega^* + \ii^{(l,-1)}_0 , &\qquad \left( \pi_{r,-1}(\ii), <_r \right) \cong 
\ii^{(r,-1)}_0 +  \ii^{(r,-1)}_+ \cdot \omega.
\end{align}
To find the density of $ \ii$, we must compute the values of all $8$ summands on the right hand sides. We present algorithms to compute $d\left(\ii^{(l,1)}_{0}\right), d\left(\ii^{(l,1)}_+\cdot \omega\right)$; the rest of the densities can be computed in a similar fashion.

We first need to find an interval isomorphic to $\pi_{l,1}(\ii)$ that is easier to deal with.

\begin{rem}\label{rem: box with pivot-1}
If $\ii \subseteq H(v)$ is a box pivoted at $(\xx_1, \xx_2)$, then $\{\uu\xx_2 \mid \uu \in \pi_l (\ii)\} $ is an interval in $H_l(\xx_1\xx_2)$ containing $\xx_1\xx_2$. Similarly $\{\xx_1\uu \mid \uu \in \pi_r (\ii)\} $ is an interval in $H_r(\xx_1\xx_2)$ containing $\xx_1\xx_2$.
\end{rem}

Consider the injective monotone maps $f_l:H_l(\xx_1\xx_2) \to H_l (v)$ and $ f_r: H_r(\xx_1\xx_2) \to H_r (v)$ defined by $f_l(\uu\xx_1\xx_2) := \uu\xx_1$ and $f_r(\xx_1\xx_2\vv) := \xx_2\vv $. Remark \ref{rem: box with pivot-1} yields $\pi_l(\ii) \subseteq \mathrm{Img}(f_l)$ and $\pi_r(\ii) \subseteq\mathrm{Img}(f_r)$.

\begin{prop}\label{prop: Image is interval}
The set $\mathrm{Img}(f_l)$ is an interval in $H_l(v)$.
\end{prop}
\begin{proof}
Towards a contradiction, assume $\mathrm{Img}(f_l)\subseteq H_l(v)$ is not an interval. Then there are $\vv_1\xx_1\xx_2, \vv_2\xx_1\xx_2 \in H_l(\xx_1\xx_2)$ and $\vv' \in H_l (v)\setminus\mathrm{Img}(f_l)$ such that $\vv_1\xx_1 <_l \vv' <_l \vv_2\xx_1$. If $\xx_1 \not\sqsubseteq_l\vv'$ then $\theta_l(\vv_1\xx_1 \mid \vv') = \theta_l(\xx_1 \mid \vv') = \theta_l (\vv_2\xx_2 \mid \vv') = 1$, which is a contradiction to $ \vv_1\xx_1 <_l \vv'$. Thus $\xx_1 \sqsubseteq_l\vv'$.

Since $\vv' \notin \mathrm{Img}(f_l)$, we get that $\vv'\xx_2$ is not a string. Thus there exist positive length strings $\zz_1 \sqsubseteq_l \vv'$ and $\zz_2 \sqsubseteq_r \xx_2$ such that $\zz_1\zz_2\in\rho\cup\rho^{-1}$. Without loss of generality, assume that $\zz_1\zz_2 \in \rho$. Since $\vv_1\xx_1\xx_2$ is a string, we get that $\zz_1 \not\sqsubseteq_l\vv_1\xx_1$. Since $\delta(\zz_1) = -1 $, we get $\theta_l(\zz_1 \mid \vv_1\xx_1) = -1$, and hence $\zz_1 <_l \vv_1\xx_1$. Since $\vv_1\xx_1 \sqcap_l \vv' \sqsubset_l \zz_1 = \zz_1 \sqcap_l \vv'$, Remark \ref{rem: Hammock order property} gives $\vv' <_l\vv_1\xx_1$, which is clearly a contradiction.
\end{proof}

\begin{cor}\label{rem: box with pivot-2}
For any intervals $I_1 \subseteq H_l(\xx_1\xx_2)$ and $I_2 \subseteq H_r(\xx_1\xx_2)$ containing $\xx_1\xx_2$, there exists a box $\ii' \subseteq H(v)$ pivoted at $(\xx_1,\xx_2)$ such that $ \pi_l (\ii')  = f_l(I_1)$ and $ \pi_r(\ii') = f_r(I_2)$.
\end{cor}
\begin{proof}
Define $\ii': = \{ (\uu\xx_1 , \xx_2\vv) \in H(v) \mid \uu\xx_1\xx_2 \in I_1 \text{ and } \xx_1\xx_2\vv \in I_2\}$. If $\uu\xx_1\xx_2 \in I_1$ then $(\uu\xx_1,\xx_2)\in H(v)$, and hence $\uu\xx_1\in\pi(\ii')$. Thus $f_l(I_1)\subseteq\pi(\ii')$. Conversely, if $\uu\xx_1\in\pi_l(\ii')$ then  for some $\xx_1\xx_2\vv\in I_2$ we have $(\uu\xx_1,\xx_2\vv)\in\ii'$ and $\uu\xx_1\xx_2\in I_1$. Hence $\uu\xx_1\in f_l(I_1)$, and thus $\pi_l(\ii')= f_l(I_1)$. Similarly we can argue that $\pi_r(\ii')=f_r(I_2)$. Finally, Proposition \ref{prop: Image is interval} gives that $\pi_l(\ii')$ and $\pi_r(\ii')$ are intervals, and thus $\ii'$ is indeed a box in $H(v)$.
\end{proof}

Thus Corollary \ref{stdprojcor} and Proposition \ref{prop: Image is interval} together imply that $f_l^{-1}[\pi_l(\ii)]$ is the maximal scattered interval in $H_l(\xx_1\xx_2)$ containing $\xx_1\xx_2$. Furthermore, $f_l^{-1}[\pi_{l,1}(\ii)] $ is the maximal scattered interval in $H_l^1(\xx_1\xx_2)$ containing $\xx_1\xx_2$. Since $\pi_{l,1}(\ii) \cong f_l^{-1}[\pi_{l,1}(\ii)]$ it is enough to study the maximal scattered interval in $H_l^1(\xx_1\xx_2)$ containing $\xx_1\xx_2$. For brevity set $\xx_0 := \xx_1\xx_2$.

If each $\sB \in\QBa_{l,1}(\xx_0)$ is domestic then Theorem \ref{prop: Scattered hammock} gives that $H_l^1 (\xx_0)$ is scattered. From the discussion in the paragraph above, we also get $f_l^{-1}[\pi_{l,1} (\ii)] = H_l^1 (\xx_0)\cong\pi_{l,1}(\ii)$. Since $H_l^1 (\xx_0)$ is a bounded order we get $\ii^{(l,1)}_+ \cdot \omega = 0$ and $H_l^1(\xx_0) \cong \ii^{(l,1)}_0$. The algorithm for computing $d\left(\ii^{(l,1)}_0\right)=d(H_l^1(\xx_0))$ is outlined after Proposition \ref{jkbound}.
 
Now suppose $\QBa_{l,1}(\xx_0)$ contains a non-domestic element. Set
$$ X := \{ \brac{1}{\LB}(\xx_0)\mid\sB \in \QBa_{l,1}(\xx_0), \sB\text{ is non-domestic}\}.$$

\begin{prop}\label{boundaryprojMSB}
If $\yy_0$ is the $<_l$-least element in X, then $[\xx_0, \yy_0)_l = f_l^{-1}[\pi_{l,1} (\ii)].$
\end{prop}
\begin{proof}
    Recall that $f_l^{-1}[\pi_{l,1} (\ii)]$ is the maximal scattered interval in $H_l^1(\xx_0)$ containing $\xx_0$. Take a non-domestic $\sB \in \QBa$ such that $\brac{1}{\LB}(\xx_0) = \yy_0$. There are two cases.
    \begin{enumerate}
        \item If $[\xx_0, \yy_0)_l \subseteq f_l^{-1}[\pi_{l,1} (\ii)]$, choose $\vv \in H_l^1(\xx_0)$ satisfying $\vv>_l\yy_0$. Take $n \in \N$ such that $\left| \LB^n (\xx_0) \right| > |\vv|$. Let $\yy' := \LB^n (\xx_0).$ Since $\yy' <_l \yy_0 <_l \vv$ and $|\yy'| > |\vv|$, we get that $\yy'' <_l \vv$ whenever $\yy' \sqsubseteq_l \yy''$, and hence $H_l(\yy') \subseteq [\xx_0 , \vv)_l$. Since $\yy' = \LB^n (\xx_0) \in \OSt{l,1}{\sB}$ and $\sB$ is non-domestic, Theorem \ref{prop: Scattered hammock} gives that $H_l(\yy')$ is not scattered, and thus $ [\xx_0,\vv)_l $ is not scattered for any $\vv>_l\yy_0$. Since $ f_l^{-1}[\pi_{l,1} (\ii)]$ is maximal scattered, $ f_l^{-1}[\pi_{l,1} (\ii)] \subset [\xx_0 , \vv)_l $ for any $\vv>_l\yy_0$. Thus we obtain $ f_l^{-1}[\pi_{l,1} (\ii)] = [\xx_0 , \yy_0)_l . $

        \item If $ f_l^{-1}[\pi_{l,1} (\ii)] \subseteq [\xx_0, \yy_0)_l,$ then due to the maximality of $ f_l^{-1}[\pi_{l,1} (\ii)]$ it is enough to show that $ [\xx_0, \yy_0)$ is scattered to conclude $f_l^{-1}[\pi_{l,1} (\ii)] = [\xx_0, \yy_0)_l$. Towards a contradiction assume $  [\xx_0, \yy_0)_l $ is not scattered. Since
        
$$  [\xx_0, \yy_0)_l \cong \sum_{n \in \omega} [\LB^n(\xx_0), \LB^{n+1}(\xx_0))_l,$$
        Proposition \ref{prop: SCattered sum of linear orders} yields $n \in \omega$ such that $[\LB^n(\xx_0) , \LB^{n+1} (\xx_0) ]_l$ is not scattered. Then $[\xx_0, \LB^{n+1} (\xx_0)]_l$ is not scattered. Let $\langle 
        \zz_i \mid i \in \eta\rangle$ be a sequence in $ [\xx_0 , \LB^{n+1} (\xx_0) ]_l $ such that $\zz_i <_l \zz_j$ if $i<j$. Since the number of strings with length at most $ \left|\LB^{n+1} (\xx_0)\right|$ is finite, without loss of generality assume $|\zz_i|> |\LB^{n+1} (\xx_0)|\geq |\xx_0|$ for all $i \in \eta.$ Let $$X := \{ \zz \in H_l^1(\xx_0) \mid \zz \sqsubseteq_l \zz_i \text{ for some }i \in \eta \text{ and } |\zz| = |\LB^{n+1} (\xx_0)|+1\}.$$
        Clearly $X$ is a finite subset of $[\xx_0, \LB^{n+1} (\xx_0)]_l$. For $\zz \in X$, define $A_{\zz} := \{ \zz_i\mid\zz \sqsubseteq_l \zz_i \}$. Since
        $$ \sum\limits_{\zz \in (X,<_l)} \left(A_{\zz}, <_l\right) \cong (\{\zz_i\mid i \in \eta\}, <_l) \cong \eta,$$
        Proposition \ref{prop: SCattered sum of linear orders} yields $\zz_0 \in X$ such that $(A_{\zz_0}, <_l)$ is not scattered. Since $A_{\zz_0} \subseteq H_l(\zz_0)$ and $|\zz_0|>|\LB^{n+1}(\xx_0)| > |\xx_0|$, we get that $ H_l(\zz_0) $ is not scattered and $H_l(\zz_0 ) \subseteq [\xx_0, \LB^{n+1}(\xx_0)]_l \subseteq [\xx_0 ,\yy_0)_l.$ Theorem \ref{prop: Scattered hammock} yields a non-domestic $\sB' \in \QBa_l(\zz_0)$ so that $\zz_0 \in \OSt{l}{\sB'}$. Then Remark \ref{rem: reversing the sign of STB} yields $\zz' \in H_l(\zz_0)\cap\OSt{l,1}{\sB'}$. Since ${\QBa_l}(\zz_0) \subseteq {\QBa_{l,1}}(\xx_0)$, we get that $ \sB' \in {\QBa_{l,1}}(\xx_0)$. Moreover, since $|\zz'| > \left|\LB^{n+1}(\xx_0)\right|$ and $\zz' \in [\xx_0 , \LB^{n+1} (\xx_0)]_l$, we get that $\brac{1}{l_{\sB'}}(\zz') <_l \LB^{n+1} (\xx_0) <_l \yy_0.$ Finally, Proposition \ref{prop: <1,LB> is an order preserving map} implies $\brac{1}{l_{\sB'}}(\xx_0) \leq_l \brac{1}{l_{\sB'}}(\zz') <_l \yy_0 $, which is a contradiction to $<_l$-minimality of $\yy_0$ in $X$.
    \end{enumerate}
\end{proof}

Let $\yy_0=\brac1{l_{\sB_0}}(\xx_0)$ for some $\sB_0\in\QBa_{l,1}(\xx_0)$. Although the result above gives a way to compute the boundary of a projection of an MSB as the minimum of a finite set $X$, the strings in $X$ are infinite. We rectify this with describing a sufficiently long finite substring of $\yy_0$ as the minimum of a finite set of finite strings. Let $Y$ be the set of all $\ol\bb_1^3\ol\uu_1\xx_0\in H_l^1(\xx_0)$ satisfying the following properties:
\begin{enumerate}
    \item $\ol\bb_1\in\Cyc$ and $\ol\bb_1$ is a cyclic permutation of a band $\bb_1$, where $\bb_1\in\mathsf{Ba}_l(\sB'')$ for some $\sB''\in\QBa$;
    \item $\ol\bb_1\ol\uu_1\xx_0\in\OSt{l}{\sB'}$ for some non-domestic $\sB'\in\QBa$;
    \item if $|\ol\uu_1|>0$ then
    \begin{enumerate}
        \item $\theta_r(\ol\bb_1)=-\theta_r(\ol\uu_1)$;
        \item if $\ol\vv_1\sqsubset_l\ol\uu_1$ then $\ol\vv_1\xx_0$ is $H_l$-reduced relative to $(\xx_0,1)$;
        \item if $\ol\uu_1\xx_0$ is not $H_l$-reduced relative to $(\xx_0,1)$ then $\theta_l(\ol\bb_1)=-1$.
    \end{enumerate}
\end{enumerate}

Theorem \ref{prop: H-reduced infty bux} applied to $\xx_0\in H_l^1(\xx_0)$ and $\sB_0\in\QBa$ yields $\sB\in\QBa$, $\bb\in\BaB$, $\ol\bb$ a cyclic permutation of $\bb$, $\uu\in\St$ and $\ol\uu\sqsubseteq_l\uu$ such that $\ol\bb^3\ol\uu\xx_0\in Y$ and $\yy_0=\brac1{\sB_0}(\xx_0)=\prescript{\infty}{}{\bb}\uu\xx_0=\prescript{\infty}{}{\ol\bb}\ol\uu\xx_0$. Thus $Y\neq\emptyset$. 
\begin{prop}\label{prop: boundary of the MSI}
Using the notation above, $\ol\bb^3\ol\uu\xx_0$ is the the $<_l$-minimal element in $Y$.
\end{prop}

\begin{proof}
Suppose not. Then there are $\sB',\sB''\in\QBa$, $\bb_1\in\mathsf{Ba}_l(\sB')$, a cyclic permutation $\ol\bb_1$ of $\bb_1$ and $\ol\uu_1\in\St$ such that $\ol\bb_1^3\ol\uu_1\xx_0\in Y$ and $\ol\bb_1^3\ol\uu_1\xx_0<_l\ol\bb^3\ol\uu\xx_0$. Let $\vv\xx_0:=\ol\bb^3\ol\uu\xx_0\sqcap\ol\bb_1^3\ol\uu_1\xx_0$.

If $\ol\bb^3\ol\uu\xx_0\sqsubseteq_l\ol\bb_1^3\ol\uu_1\xx_0$, then using clause $(3b)$ of the definition of $Y$ and $\ol\bb^2\ol\uu\xx_0\eqvl\ol\bb\xx_0$, we get $\ol\bb^2\ol\uu\xx_0\not\sqsubset_l\ol\uu_1\xx_0$. Similarly, $\ol\uu_1\xx_0\sqsubseteq_l\ol\bb^2\ol\uu\xx_0\sqsubset_l\ol\bb^3\ol\uu\xx_0\sqsubseteq_l\ol\bb_1^3\ol\uu_1\xx_0$. Then $\ol\bb\sqsubseteq\ol\bb_1$, which is a contradiction to Proposition \ref{rem: Bl not being substrings}. Thus $\ol\bb^3\ol\uu\xx_0\not\sqsubseteq_l\ol\bb_1^3\ol\uu_1\xx_0$, and hence $\vv\xx_0\sqsubset_l\ol\bb^3\ol\uu\xx_0$. Similarly we can show that $\vv\xx_0\sqsubset_l\ol\bb_1^3\ol\uu_1\xx_0$. Then Remark \ref{rem: Hammock order property} applied to $\vv\xx_0\sqcap_l\ol\bb_1^3\ol\uu_1\xx_0=\vv\xx_0\sqsubset_l\prescript{\infty}{}{\ol\bb}\ol\uu\xx_0\sqcap_l\ol\bb^3\ol\uu\xx_0$ gives $\ol\bb_1^3\ol\uu_1\xx_0<_l\prescript{\infty}{}{\ol\bb}\ol\uu\xx_0$. Since $\ol\bb_1\ol\uu_1\xx_0\in\OSt{l}{\sB'}$ for a non-domestic $\sB'$ by the definition of $Y$, we also have $\ol\bb_1^3\ol\uu_1\xx_0\in\OSt{l}{\sB'}\cap H_l^1(\xx_0)$. There are two cases.
\begin{itemize}
    \item If $\brac1{l_{\sB'}}(\xx_0)<_l\ol\bb_1^3\ol\uu_1\xx_0$ then $\ol\bb^3\ol\uu\xx_0<_l\prescript{\infty}{}{\ol\bb}\ol\uu\xx_0=\yy_0$ implies $\brac1{l_{\sB'}}(\xx_0)<_l\yy_0$, a contradiction to the $<_l$-minimality of $\yy_0$ in $X$.
    \item If $\ol\bb_1^3\ol\uu_1\xx_0<_l\brac1{l_{\sB'}}(\xx_0)$ then Proposition \ref{prop: <1,LB> is an order preserving map} yields some $n\geq1$ such that $ l_{\sB'}^n(\xx_0) = \ol\bb_1^3\ol\uu_1\xx_0 $. Thus $\brac1{l_{\sB'}}(\xx_0)=\prescript{\infty}{}{\ol\bb_1}\ol\uu_1\xx_0$. Then Remark \ref{rem: Hammock order property} applied to $\yy_0\sqcap_l\brac1{l_{\sB'}}(\xx_0)=\vv\xx_0\sqsubset_l\ol\bb_1^3\ol\uu_1\xx_0=\ol\bb_1^3\ol\uu_1\xx_0\sqcap_l\brac1{l_{\sB'}}(\xx_0)$ gives $\brac{1}{l_{\sB'}}(\xx_0) <_l \yy_0$, which is again a contradiction to the $<_l$-minimality of $\yy_0$ in $X$. 
\end{itemize}
\end{proof}

Thanks to Propositions \ref{balbfiniteness} and \ref{prop: H-reduced finitely many}, the set $Y$ is finite can be computed effectively. Now recall from Proposition \ref{boundaryprojMSB} and Equation \eqref{hammocksimp3} that
\begin{equation}\label{hammocksimp}
\pi_{l,1}(\ii) \cong [\xx_0, \yy_0)_l\cong \sum_{n=1}^{j-1} H_l^{-1}(\LB^n(\xx_0)) + \left( \sum_{n=0}^{k-1} H_l^{-1}(\LB^{j+n}(\xx_0)) \right) \cdot \omega,
\end{equation}
so that we can choose $\ii^{(l,1)}_0:=\sum_{n=1}^{j-1} H_l^{-1}(\LB^n(\xx_0))$ and $\ii^{(l,1)}_+:=\sum_{n=0}^{k-1} H_l^{-1}(\LB^{j+n}(\xx_0))$.

We have shown at the end of \S~\ref{sec:alg msi lin ord} that the set $\{ d\left( H_l^{-1}(\LB^n(\xx_0)) \right) \mid n \in [j+k-1] \}$ can be effectively computed. If a finite sum of bounded discrete linear orders is finite then its density can be computed using counting. On the other hand, if the sum is infinite then its density can be computed using Proposition \ref{prop: omega times linear order} in combination with Remark \ref{rem: plus and plusdot}. This completes the description of an algorithm to compute $d(\ii_0^{(l,1)})$ and $d(\ii_+^{(l,1)})$ while $d(\ii_+^{(l,1)}\cdot\omega)$ can be computed using Proposition \ref{prop: omega times linear order}.

Finally we are ready to compute the density of the MSB $\ii$.
\begin{prop}\label{prop: density of a box}
If $\ii$ is an MSB in $H(v)$ withe unbounded projections then using the notations above,
$$d\left(\ii\right) = \max \left\{d\left( \ii_{0}^{(l,-1)}\dot+ \ii_{0}^{(l,1)} \dot+ \ii_{0}^{(r,-1)} \dot+ \ii_{0}^{(r,1)}\right), d\left(\ii_+^{(l,1)}\cdot \omega\right),d\left(\ii_-^{(l,-1)}\cdot \omega\right), d\left( \ii_{-}^{(r,1)} \cdot \omega \right), d\left( \ii_+^{(r,-1)} \cdot \omega \right) \right\}.$$
\end{prop}

\begin{proof}
Corollary \ref{about density computation} applied to the corner decomposition of $\ii$ given by Equations \eqref{lefthammockbreakdown}, \eqref{righthammockbreakdown} and \eqref{hammocksimp} followed by Theorem \ref{Schroer formula} and Proposition \ref{Our formula} gives
$$d\left(\ii\right) = \max \left\{ d\left(\ii_{0}^{(l,-1)}+ \ii_{0}^{(l,1)} \dot+ \ii_{0}^{(r,-1)} \dot+ \ii_{0}^{(r,1)}\right), d\left((\ii_+^{(l,1)}+\ii_+^{(r,-1)})\cdot\omega \right), d\left((\ii_-^{(l,-1)}+\ii_{-}^{(r,1)}) \cdot \omega \right)\right\}.$$

Furthermore, Propositions \ref{prop: density of sum of orders} and \ref{Our formula} simplify the densities of corners as the maximum of two quantities to give the expression in the statement.
\end{proof}

\begin{exmp}\label{ex: density of a box}
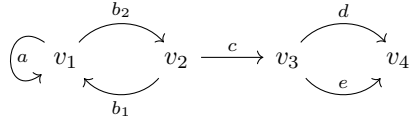
\begin{figure}[h]
\begin{tikzcd}
v_1 \arrow["a", loop, distance=2em, in=215, out=145] \arrow[r, "b_2", bend left=49] & v_2 \arrow[l, "b_1", bend left=49] \arrow[r, "c"] & v_3 \arrow[r, "d", bend left=49] \arrow[r, "e", bend right=49] & v_4
\end{tikzcd}
    \caption{$\Gamma_2$ with $\rho = \{ ab_1 , b_2a, a^2 , (b_1b_2)^3 , cb_2 , dc \}$}
    \label{fig: computing the density}
\end{figure}
Consider the string algebra $\Gamma_2$ from Figure \ref{fig: computing the density}. We choose and fix $\sigma $ and $\varepsilon$ maps as follows: $\sigma(a)=\varepsilon (b_1)=\varepsilon (b_2)=\sigma(c)=\sigma (e)=\varepsilon (d) := -1$ and $\sigma (b_1)=\varepsilon (c) =\varepsilon (e)=\sigma(b_2) =\varepsilon (a) =\sigma (d): = 1 $. It is easy to verify that the poset $\QBa$ contains only four elements, say $\QBa = \{\sB_1 , \sB_2, \sB'_1 , \sB'_2\}$, where $\sB_2=\{eD\}$ and $\sB'_2=\{dE\}$ are domestic while the other are non-domestic. Furthermore, $\mathrm{Ba}_l(\sB_1) = \mathrm{Ba}_{\rb} (\sB_1) = \{aB\}$, $\mathrm{Ba}_{\lb} (\sB_1) = \mathrm{Ba}_r (\sB_1) = \{aB^2\} $, $\mathrm{Ba}_l(\sB'_1) = \mathrm{Ba}_{\rb} (\sB'_1) = \{b^2A\}$, and $\mathrm{Ba}_{\lb} (\sB'_1) = \mathrm{Ba}_r (\sB'_1) = \{bA\} $.

We will compute the density of the MSB $\ii \subseteq H(v_2)$ pivoted at $( B_2, B_1)$.

As per the discussion after Corollary \ref{rem: box with pivot-1} we know that $\pi_{l,1} (\ii) $ is order isomorphic to the maximal scattered interval in $H_l^1(B_2B_1)$ containing $B_2B_1$. Using Proposition \ref{prop: boundary of the MSI}, we get that the maximal scattered interval in $H_l^1(B_2B_1)$ containing $B_2B_1$ is $\left[B_2B_1, \prescript{\infty}{}{(aB_2B_1)}B_2B_1\right)_l$. Note that $\langle 1, l_{\sB_1}\rangle (B_2B_1) = \prescript{\infty}{}{(aB_2B_1)}B_2B_1$ is indeed the minimum of the set $X$ defined before Proposition \ref{boundaryprojMSB}. Since $l_{\sB_1}^2(B_2B_1) = a(B_2B_1)^2 \eqvl (a(B_2B_1))^2 = l_{\sB_1}^4(B_2B_1)$, we plug in $j=k=2$ in Equation \eqref{hammocksimp} to get $\ii_0^{(l,1)} \cong H_l^{-1}(B_1B_2B_1)$ and $\ii_+^{(l,1)} \cong H_l^{-1}(a(B_2B_1)^2) + H_l^{-1} (B_1a(B_2B_1)^2)$. Following any algorithm at the end of \S~\ref{sec:alg msi lin ord}, we get $d(\ii_0^{(l,1)}) = \omega + 1$ and $d(\ii_+^{(l,1)}\cdot \omega) = \omega\cdot 2 $.

Similarly we obtain $d(\ii^{(l,-1)}_0) = -1$, $ d(\ii^{(l,-1)}_- \cdot \omega ) = \omega $, $d(\ii^{(r,1)}_0)= d(\ii^{(r,-1)}_0)<\omega$ and $ d(\ii^{(r,1)}_- \cdot \omega) =d(\ii^{(l,-1)}_- \cdot \omega ) = \omega $. Therefore, by Proposition \ref{prop: density of a box}, we get that $d(\ii) = \omega \cdot 2$.
\end{exmp}

\begin{cor}\label{Hlreqdensity}
Suppose $\ii_1$ and $\ii_2$ are MSBs in $H(v)$ with pivots $(\xx_1,\xx_2)$ and $(\yy_1,\yy_2)$ respectively. If $(\xx_1,\xx_2)\eqvlr(\yy_1,\yy_2)$ then $d(\ii_1)=d(\ii_2)$.
\end{cor}

\begin{proof}
Recall from the discussion after Corollary \ref{rem: box with pivot-2} that since $\ii_1$ is an MSB with pivot $(\xx_1 , \xx_2)$ then for any $j\in\{1,-1\}$ the interval $\pi_{l,j}(\ii_1)$ in $H_l(v)$ (resp. $\pi_{r,j}(\ii_1)$ in $H_r(v)$) is order isomorphic to the maximal scattered interval containing $\xx_1\xx_2$ in $H^j_l(\xx_1\xx_2)$ (resp. in $H_r^j(\xx_1\xx_2)$). A similar statement is true for the MSB $\ii_2$. Therefore, if $(\xx_1,\xx_2)\eqvlr(\yy_1,\yy_2)$ then Proposition \ref{prop: density of a box} gives $d(\ii_1)=d(\ii_2)$ as required.
\end{proof}

Finally for the stable rank computation, recall from Equation \eqref{stablerank7} that 
\begin{equation*}
\st(\Lambda)=\max_{v\in Q_0}\ \max\{d(\ol\ii)\mid \ol\ii\text{ is an MSB in }\hb(v)\}.
\end{equation*}
Theorem \ref{hammockordermain} and Lemma \ref{excptfin} ensure that there are only finitely many order isomorphism classes of MSBs in $\hb(v)$. Recall that if $\ol\ii$ is an MSB in $\hb(v)$ then $\ii:=\ol\ii\cap H(v)$ is an MSB in $H(v)$ and $|d(\ol\ii)-d(\ii)|$ is finite (Theorem \ref{density=rank-1}). Thanks to Lemma \ref{excptfin} and Proposition \ref{balbfiniteness}, an upper bound on this difference can be effectively computed. We have already described an algorithm to effectively compute $d(\ii)$ for an MSB $\ii$ in Proposition \ref{prop: density of a box}. In view of Propositions \ref{Hlrredfiniteness} and \ref{prop: Hlr reduced representative}, Corollary \ref{Hlreqdensity} ensures that we can choose, in a computationally effective manner, a representative set of isomorphism classes of MSBs in $H(v)$--this representative set consists of MSBs with $H_{lr}$-reduced pivots. Thus we have shown that it is possible to effectively compute $\st(\Lambda)$ for a string algebra $\Lambda$ up to a finite error, as expressed by the following analogue of the theorem on page 67 of \cite{Schroer98hammocksfor}
\begin{equation}\label{stablerank8}
0\leq\st(\Lambda)-\max_{v\in Q_0}\ \max\{d(\ii)\mid \ii\text{ is an MSB in }H(v)\text{ with }H_{lr}\text{-reduced pivot}\}\leq\max_{v\in Q_0}|E(v)|,
\end{equation}
where $E(v)$ is the set of exceptional points in $\hb(v)$.

\begin{exmp}\label{ex: Proof of theorem 2}
Continuing from Example \ref{ex: density of a box}, we will show that $\st(\Gamma_2) = \omega\cdot 2$.

Using the algorithm from \cite[\S~11]{SKSK}, we compute the order types of all hammock linear orders.
\begin{align*}
    (H_l(1_{(v_1 , 1)}), <_l ) &\cong (\omega+\omega^*)\cdot \omega + (\zeta\cdot\zeta)\cdot \eta + (\omega+\omega^*)\cdot\omega^*,  &\quad (H_r(1_{(v_1 , 1)}), <_l ) \cong \omega + \zeta \cdot \eta + \omega^*,\\
    (H_l(1_{(v_2 , 1)}), <_l ) &\cong (\omega+\omega^*)\cdot \omega + (\zeta\cdot\zeta)\cdot \eta + (\omega+\omega^*)\cdot\omega^*, &\quad (H_r(1_{(v_2 , 1)}), <_l ) \cong \omega + \zeta \cdot \eta + \omega^*, \\
    (H_l(1_{(v_3 , 1)}), <_l ) &\cong \omega+ \omega^* ,  &\quad (H_r(1_{(v_3 , 1)}), <_l ) \cong \omega + \zeta \cdot \eta + \omega^*,\\
    (H_l(1_{(v_4 , 1)}), <_l ) &\cong \omega+ \omega^*, &\quad (H_r(1_{(v_4 , 1)}), <_l ) \cong \omega + \zeta \cdot \eta + \omega^*.
\end{align*}

Recall that $\st (\Gamma_2) = \max_{v\in Q_0} \sup \{d(\ol\ii) \mid \ii \text{ is an MSB in } H(v)\}$. We will show that for every MSB $\ii$ in $\Gamma_2$ $d(\ol\ii)\leq \omega \cdot 2$, and there exists an MSB $\ii$ such that $d(\ol\ii) = \omega \cdot 2$.  
\begin{itemize}
    \item Recall that $(H(v_1),<)= (H_l(v_1), <_l) \otimes (H_r(v_1 ), <_r)$. From the above order type computations, we get that $d(\ii) \leq \omega \cdot 2$ for any MSB $\ii \subseteq H(v_1)$. We will show that $d(\ol\ii) = d(\ii)$ for any MSB $\ii \subseteq H(v_1)$.
    
    Suppose $\ii \subseteq H(v_1) $ satisfies $d(\ii) < d(\ol\ii)$. Recall from Theorem \ref{density=rank-1}(1) that $\ol\ii\setminus \ii$ is finite. If this set is non-empty then choose $\bb_1\in\mathsf{Cyc}(v)$, a cyclic permutation of $\bb\in \QBa_0$, such that $(\prescript{\infty}{}{\bb_1},\bb_1^\infty ) \in \ol\ii\setminus\ii$. Thanks to Proposition \ref{excptfin}, we know that $\bb \in \{a(B_2B_1)^2 , aB_2B_1, (b_1b_2)^2A , b_1b_2A, eD, dE\}$. Since $s(\bb_1) = v_1 = t(\bb_1)$ and $-\sigma(\bb_1) = 1 = \varepsilon(\bb_1)$, we get that $\bb \in \{ a(B_1B_2)^2, aB_1B_2 \}$. If $\bb = a(B_1B_2)^2 $ then $\bb \notin \bigcup\limits_{\sB \in \QBa}\left(\mathrm{Ba}_{l} (\sB) \cup \mathrm{Ba}_{\rb} (\sB)\right)$. Therefore Corollary \ref{cor: Exceptional point in left intervals} gives that $(\prescript{\infty}{}{\bb_1},\bb_1^{\infty})$ is a corner element of $\ol\ii$ not comparable to any other element of $\ii$. A similar incomparability argument works if $\bb=aB$. Therefore we have shown that $d(\ol\ii) = d(\ii)$.
    \item An argument similar to the above case shows that $d(\ol\ii)\leq\omega\cdot2$ for any MSB $\ii$ in $H(v_2)$. Moreover, the MSB in Example \ref{ex: density of a box} attains this bound.
    \item It is clear from the order type computations above that if $\ii$ is an MSB in either $H(v_3)$ or $H(v_4)$ then $d(\ii)\leq\omega+1$. Since $\ol\ii\setminus\ii$ is finite, we get that $d(\ol\ii)<\omega\cdot 2$ for any such $\ii$.
\end{itemize}
\end{exmp}
Therefore, we have completed the proof of $\st (\Gamma_2) = \omega \cdot 2$. In fact, for every $n\geq2$, we can similarly show that the string algebra $\Gamma_n$ from Figure \ref{ex: Gamma n} satisfies $\st(\Gamma_n)=\omega \cdot n$. 
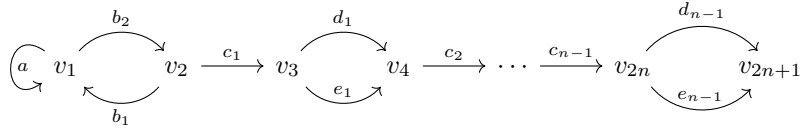
\begin{figure}[h]
\begin{tikzcd}
v_1 \arrow["a", loop, distance=2em, in=215, out=145] \arrow[r, "b_2", bend left=49] & v_2 \arrow[l, "b_1", bend left=49] \arrow[r, "c_1"] & v_3 \arrow[r, "d_1", bend left=49] \arrow[r, "e_1", bend right=49] & v_4 \arrow[r, "c_2"] & \cdots \arrow[r, "c_{n-1}"] & v_{2n} \arrow[r, "d_{n-1}", bend left=49] \arrow[r, "e_{n-1}", bend right=49] & v_{2n+1}
\end{tikzcd}
    \caption{$\Gamma_n$ with $\rho = \{ ab_1 , b_2a, a^2 , (b_1b_2)^3 , c_1b_2  \} \cup \{d_ic_i \mid i \in [n-1]\} \cup \{c_{i+1}d_i \mid i \in [n-2]\}$}
    \label{ex: Gamma n}
\end{figure}

Schr\"{o}er \cite[Theorem~1]{schroer2000infinite} proved that for any $n , d \in (\N^+ \times \N^+)\setminus \{(1,1)\}$ there exists a domestic string $\Lambda$ such that $\st (\Lambda) = \omega \cdot n + d$. Examples of string algebras with stable ranks of $\omega $ and $\omega + 1$ were given in \cite[Example~4.3.6]{GKS20}. Above we have shown that for every integer $n\geq2$ there exists a string algebra of stable rank $\omega \cdot n$. This completes the proof of Theorem \ref{reversest}.

\section{Stable rank trichotomy conjecture}\label{sec: new conjecture}
Fix an algebraically closed field $\K$ throughout this section. Without loss, assume that $\Lambda$ is a basic connected finite-dimensional $\K$-algebra. Choose and fix a presentation $\Lambda\cong\K\Q/\langle\rho\rangle$ for a finite connected quiver $\Q$ and a finite set $\rho$ of relations. We believe that all ``finiteness results'' for $\Lambda$ are {owed to} this finite presentation.

Recall that $\Lambda$ is a tame representation type algebra if, for all $d\in\N^+$, all but finitely many iso-classes of indecomposable modules of total dimension $d$ lie in one of $\mu(d)$-many $1$-parametrized families. Using this language we discuss salient features of the proof of Theorem \ref{oldconj} for a string algebra $\Lambda$.

Choose and fix a vertex $v$ of the quiver. The elements of $\HH(v)$ are in bijective correspondence with the iso-classes of $v$-pointed indecomposable modules in $\Lambda\dmod$, i.e., those with a composition factor the simple module $S(v)$; this set is partially ordered by the existence of $v$-pointed maps, i.e., whose image has a composition factor $S(v)$. The subposet $H(v)$ of $\HH(v)$ consists of only the finitely many iso-classes of exceptional modules for each $d\in\N^+$, i.e., those that do not belong to any $1$-parameter family. The elements of $\hb(v)\setminus H(v)$ are in bijective correspondence with the union of the sets of $1$-parameter families for $d$. In particular, $H(v)$ and $\hb(v)$ are countable. The interplay between the hammocks $H(v),\hb(v)$ and $\HH(v)$ lies at the heart of the proof of Theorem \ref{oldconj}--the order-theoretic aspects of this proof are mostly restricted to $H(v)$, and one of the central results in this direction is that the hammock poset $H(v)$ is of finite type.

There seems to be a loose parallel between the concepts in organic chemistry and in the representation theory of string algebras. Band-free strings and bands are analogous to aliphatic and aromatic compounds respectively. The concept of $H$-equivalence plays a central role in all the finiteness results for string algebras in \cite{sardar2021variations}, \cite{SardarKuberHamforDom}, \cite{SKSK} and this paper. Loosely speaking, two strings are $H$-equivalent if the set of ``functional groups'' that could be attached to them are the same. It is worth adding that the regions in the Auslander-Reiten quiver that are ``accessible'' by two $H$-equivalent strings have the same shape, where the word accessible refers to the set of strings containing a copy of them.

Amongst several finiteness results, the finiteness of the set of certain bands, defined in \cite{GKS20} as \emph{prime bands}, is a fundamental one since prime bands \emph{generate} all bands. Continuing with the chemistry analogy, the benzene ring could be thought of as a prime band whereas naphthalene could not. In this paper, we have exploited the finiteness of the subset $\bigcup_{\sB\in\QBa}(\BalB\cup\BalbB\cup\mathsf{Ba}_r(\sB)\cup\mathsf{Ba}_{\rb}(\sB))$ of the set of prime bands to capture the set of points in $\hb(v)\setminus H(v)$ relevant for the computation of the stable rank.

Now suppose $\Lambda$ is an arbitrary tame representation type algebra. Recall that the algebra $\Lambda$ is ``finitely presented'' using the pair $(\Q,\rho)$. In the early representation theory literature (e.g., \cite{skowronski1983representation}), $\alpha(\Lambda)$ denotes the maximum number of middle terms in an Auslander-Reiten sequence for $\Lambda$. We expect that the hammocks for $ \Lambda $ are bounded discrete abstract $\alpha(\Lambda)$-hammocks; here discreteness is a consequence of the existence of rank $1$ component maps between indecomposable modules in the Auslander-Reiten sequences. Say that a linear suborder $L$ of a poset $P$ is \emph{saturated} if it is maximal with respect to inclusion amongst linear suborders of $P$. We expect that each saturated linear order in $H(v)$ is a finite description linear order. Furthermore, we also expect that the hammock poset $H(v)$ is of finite type. 

The classes $\LOfp$ and $\LOfd$ of finitely presented and finite description linear orders respectively are the cornerstones of the order-theoretic components of the proof. Recall that these two classes were first introduced by L\"auchli and Leonard in a model-theoretic study of linear orders. The operations on linear orders used to construct these two types of linear orders are finitary, and each such order is constructed using only finite number of valid operations. We have already shown that the density of a finitely presented linear order is strictly bounded above by $\omega^2$.

We would like to close $\LOfd$ under more complex yet finitary operations to investigate hammocks for wild representation type algebras. However it seems that logic provides a natural bound on the density of such more complex orders. The Cantor normal form of any ordinal $\alpha<\omega^\omega$ expresses it as a finite sum of finite powers of $\omega$. Moreover, as a consequence of the works of Amit, Myers, Schmerl and Shelah, it is proved that $\omega^\omega$ is the smallest ordinal that is not finitely axiomatizable \cite[Corollary~13.24,Theorem~13.35]{rosenstein}. Thus we expect that, owing to the finite presentation of a finite-dimensional algebra, its density cannot attain $\omega^\omega$.

A classic theorem of Auslander states that a finite-dimensional algebra $\Lambda$ is of finite representation type if and only if $\st(\Lambda)<\omega$ and $\rad^{\st(\Lambda)}=0$. Prest and Ringel conjectured \cite[Conjecture~1.5]{prest2002serial} that an algebra $\Lambda$ is of domestic representation type if and only if $\rad^{\omega^2}=0$. Based on the discussion so far, the third author would like to extend this conjecture to classify all finite-dimensional algebras.
\begin{conj}(Characterization of Drozd and Crawley-Boevey's trichotomy in terms of stable rank)\label{newconj} If $\Lambda$ is a finite-dimensional $\K$-algebra then $\st(\Lambda)<\omega^\omega$. Furthermore, we have the following:
\begin{itemize}
    \item $\st(\Lambda)<\omega$ if and only if $\Lambda$ is of finite representation type;
    \item $\omega\leq\st(\Lambda)<\omega^2$ if and only if $\Lambda$ is of tame representation type but not of finite representation type. Furthermore, $\Lambda$ is of domestic representation type if and only if $\rad^{\st(\Lambda)}=0$;
    \item $\omega^2\leq\st(\Lambda)<\omega^\omega$ if and only if $\Lambda$ is of wild representation type.
\end{itemize}
\end{conj}
\section*{Declarations}
\subsection*{Ethical approval}
Not applicable

\subsection*{Competing interests}
Not applicable


\subsection*{Funding}
Not applicable

\subsection*{Availability of data and materials}
Not applicable

\printbibliography
\vspace{0.2in}
\noindent{}Suyash Srivastava\\
Indian Institute of Technology Kanpur\\
Uttar Pradesh, India\\
Email: \texttt{suyash20@iitk.ac.in}

\vspace{0.2in}
\noindent{}Vinit Sinha\\
Indian Institute of Technology Kanpur\\
Uttar Pradesh, India\\
Email: \texttt{vinitsinha20@iitk.ac.in}

\vspace{0.2in}
\noindent{}Corresponding Author: Amit Kuber\\
Indian Institute of Technology Kanpur\\
Uttar Pradesh, India\\
Email: \texttt{askuber@iitk.ac.in}\\
Phone: (+91) 512 259 6721\\
Fax: (+91) 512 259 7500
\end{document}